\documentclass[reqno]{amsart}
\usepackage[margin=1.1in]{geometry}
\usepackage{amsmath,amsthm,amssymb,enumerate,fancyhdr,mathtools}
\usepackage{enumitem}
\usepackage{physics}
\usepackage{mathpazo}
\usepackage{pgf,tikz,pgfplots}
\usepackage{mathrsfs}
\usetikzlibrary{arrows}
\usepackage{hyperref} 
\usepackage[caption=false]{subfig}
\hypersetup{colorlinks=true, linkcolor=blue, citecolor=blue, filecolor=blue, urlcolor=blue}
\usepackage{comment}
\usepackage{chngcntr}
\usepackage{etoolbox}
\counterwithin*{equation}{section}
\counterwithin*{equation}{subsection}

\usepackage[comma,sort&compress, numbers]{natbib}

\newcommand{\one}{{\bf 1}}

\newcommand{\vep}{\varepsilon}

\newtheorem{theorem}{Theorem}[section]

\newtheorem{corollary}{Corollary}[section]

\newtheorem{prop}{Proposition}[section]

\newtheorem{lemma}{Lemma}[section]
\newtheorem{definition}{Definition}[section]

\newtheorem{remark}{Remark}[section]

\newtheorem{example}{Example}[section]

\def\Var{{\rm Var}}

\begin{document}

\title[Spectra of Laplacian]{Spectral properties for the Laplacian of a generalized Wigner matrix}
\author[A. Chatterjee, R.S. Hazra]{Anirban Chatterjee and Rajat Subhra Hazra}
\address{Statistics Department, The University of Pennsylvania}
\email{anirbanchatterjee052@gmail.com}
\address{Mathematical Institute, Universiteit Leiden, Netherlands}
\email{rajatmaths@gmail.com}
\keywords{Adjacency matrices, Inhomogeneous Erd\H{o}s-R\'enyi random graph, Largest eigenvalue, Scaling limit, Stochastic block model}
\subjclass[2000]{60B20, 05C80, 46L54 }

\newcommand{\acr}{\newline\indent}

\begin{abstract}
In this article, we consider the spectrum of a Laplacian matrix, also known as Markov matrices where the entries of the matrix are independent but have a variance profile. Motivated by recent works on generalized Wigner matrices we assume that the variance profile gives rise to a sequence of graphons. Under the assumption that these graphons converge, we show that the limiting spectral distribution converges. We give an expression for the moments of the limiting measure in terms of graph homomorphisms. In some special cases, we identify the limit explicitly. We also study the spectral norm and derive the order of the maximum eigenvalue. We show that our results cover Laplacians of various random graphs including inhomogeneous Erd\H{o}s- R\' enyi random graphs, sparse W-random graphs, stochastic block matrices and constrained random graphs. 
\end{abstract}

\maketitle

\section{Introduction}\label{sec:intro} 
The Laplacian of a graph is used in various areas of combinatorics, statistical physics, and probability. Given a graph $G$ on $N$ vertices, the Laplacian is given by $A_N- D_N$, where $A_N$ is the adjacency matrix and $D_N$, is the degree matrix, that is, a diagonal matrix with $i$-th diagonal entry being the degree of the graph. If $G$ is a simple graph then the entries of $A_N$ are either $0$ or $1$. When the entries of $A_N$ are no longer restricted to $0$ and $1$, the Laplacian is referred to as the Markov matrix. In this article, we study the behavior of the eigenvalues of the Laplacian when $A_N$ is a (generalized) Wigner matrix where the entries are independent but have a variance profile. These setups come up in random graph models when edge weights are independent but not identically distributed and the variance depends on the size of the graph. For example, we can consider the case of inhomogeneous Erd\H{o}s-R\'enyi graphs where the vertex set is $[N]=\{1,\ldots, N\}$ and any two vertices $i$ and $j$ are connected independently with probability $p_{i,j}$. Other examples include the adjacency matrix of a configuration model where the edges are no longer independent. This article aims to study the behavior of the empirical spectral distribution of the Laplacian matrix under such a variance profile and analyze the behavior of the spectral norm.

A Wigner matrix is a Hermitian random matrix whose entries are i.i.d random variables up to the symmetry constraint, and have zero expectation and variance $1$. It is well known that for Wigner matrices the empirical spectral distribution (ESD) converges weakly almost surely to the semicircle law. The constant variance condition and the i.i.d requirement has subsequently been relaxed in \citet{erdHos2012bulk, erdos2010universality} to show the convergence of ESD to the semicircle law under the setup where entries can have different variances and each column of the variance profile is stochastic. Wigner matrices with a variance profile has also been considered in \citet{ajanki2017,hachem:2007,shlyakhtenko:1996,chakrabarty:2017, anderson:zeitouni}. Non-symmetric random matrices with variance profiles were considered in \citet{cook:hachem:najim:nonsymmetric}. \citet{zhu2020graphon} introduce a graphon approach to finding the moments of the limiting distribution of ESD of a Wigner-type random matrix where the entries satisfy a Lindeberg type assumption and we shall follow the setup of that article. In Wigner matrices with variance profile, generally, the variance matrix is assumed to have some structure and in particular, it was assumed in \citet{zhu2020graphon} that it gives rise to an empirical graphon which converges to a graphon. In that case, the limiting spectral distribution can be described in terms of this limiting graphon. In many important cases, the limit is not Wigner's semicircle law. The importance of assuming a variance profile lies in the fact that it can be used to model various stochastic block matrices (\citet{abbe2017community}). Under some non-sparsity assumption (average degree goes to infinity) it is known that the ESD of adjacency of homogeneous Erd\H{o}s--R\'enyi converges to the semicircle law (\citet{tran2013sparse}). The inhomogeneous extension was done subsequently in \citet{chakrabarty2018spectra}, which falls in the setup of the Wigner matrix with a variance profile.

The graph Laplacian is a counterpart of the continuous Laplacian which is well-known in the theory of diffusions and also related to a flow in the network. The spectral graph theory is the study of the properties of a graph in relationship to the characteristic polynomial, eigenvalues, and eigenvectors of its adjacency or Laplacian matrix. Laplacian matrix has relations with the number of spanning forest of the graph (through Kirchof\/f's theorem), the algebraic connectivity, number of connected components (through the multiplicity of zeroes). We refer the readers to the monograph \citet{chung} for applications of spectral analysis to graph theory. In recent statistical and machine learning applications, it has found good use in the spectral clustering techniques (\citet{luxburg:belkin:boysquet:2008, zhou:amini, couillet:BGF}) and community detection algorithms (\citet{chen:xi:lin}).

Fundamental work on random Laplacian matrices was done in \citet{bryc2006spectral} and the convergence of ESD of Laplacian matrices under the i.i.d. setup was determined. The limiting law turns out to be free convolution between the semicircle law and the standard Gaussian distribution. The ESD of Laplacian of sparse Erd\H{o}s--R\'enyi is considered in \citet{jiang2012empirical}. The normalized Laplacian in the non-sparse setting was considered in \citet{chi}. The local laws Laplacian of the Erd\H{o}s--R\'enyi graph was considered in \citet{huang2015spectral}. They showed that the Stieltjes transform of the empirical eigenvalue distribution is well-approximated by the Stieltjes transform of free convolution of the semicircle law and a standard Gaussian down to the scale $N^{-1}$. They also show that the gap statistics and averaged correlation functions coincide with the Gaussian Orthogonal Ensemble in the bulk. \citet{ding2010spectral} discusses the convergence of ESD of adjacency and Laplacian of random graphs, under the assumption that the variance of entries of $N\times N$ adjacency matrix is constant, only depending on $N$. There have been few studies on the spectral norm of the Laplacian matrix. \citet{bryc2006spectral} show that for the mean-centered case the order becomes $O(\sqrt{N\log N})$ whereas it changes to $O(N)$ when mean centering is not considered. The law of large numbers for the spectral norm and the largest eigenvalues under the assumption of independence was discussed in \citet{ding2010spectral}. They restricted the setup such that the entries satisfy symmetric constraint and the mean and variance depends only on $N$ and not on the index of the entries and showed that the order remains the same for the mean-centered version. \citet{bordenave2014markov} considered the asymmetric setup and showed that the order of growth of the largest singular value remains $O(\sqrt{N\log N})$ tallying with the previous two works. The same rate of growth of the spectral norm was studied in \citet{jiang2012low, ding2010spectral}. We are not aware of any literature which deals with the fluctuations of the spectral norm in these settings.

{\bf Main contribution of the article:} As mentioned above, we take $A_N$ to be a matrix with independent entries but having a variance profile. Our main assumption is similar to that of \citet{zhu2020graphon}, that is, the variance profile matrix gives rise to a graphon $W_N$ which converges in the cut-metric to a limiting graphon $W$. In \citet{zhu2020graphon} it was shown that the limiting spectral distribution of scaled $A_N$ can be identified through its moments. It is well-known that if $C_k$ is the $k$-th Catalan number (or $2k$-th moment of limiting spectral distribution of scaled Wigner matrix) then $C_k$ also counts the number of planer trees on $k+1$ vertices. In the homogeneous setting when all the variances are the same, each planar tree contributes 1. In the inhomogeneous setting, each planer tree $T$ has a non-negligible contribution, namely, it contributes $t(T, W)$, which indicates the number of copies of a planar tree $T$ in graphon $W$ (more explicitly, see \eqref{moments:zhu} in the next section).

In the case of Laplacian, the identification of the moments and limit becomes a significantly difficult problem. One can show that ESD of scaled $A_N-D_N$ is the same as ESD of $A_N- \widehat D_N$ where $\widehat D_N$ is independent of $A_N$ and same in distribution as $D_N$. Since $A_N$ is turning out to be a Wigner matrix with a variance profile, so it is not immediate that scaled $A_N$ and $\widehat D_N$ are asymptotically freely independent as in the i.i.d. setting. It can be shown that when the variance profile or the limiting graphon is multiplicative, then free independence helps us to characterize the limit. We explore the combinatorial expression of the moments in terms of graph homomorphisms. We show that the moments can be expressed in terms of a mixture of Gaussian moments and $t(\widetilde T, W)$ where $\widetilde T$ will a modification of the planar tree and the expression $t(\widetilde T, W)$ indicate the number of copies of this modified tree in the limiting graphon. We are not aware of such existing expressions for the limits of the moments of random Laplacian matrix. It is well known that the limit in the case of adjacency matrix is known as operator valued semi-circular law and it has connections to freeness over amalgamation (\citet{Mingo:speicher:book,Nica:shlyakhtenko:speicher}). We strongly believe that this connection extends to Laplacian case too  but we don't explore this aspect in the present article.

We derive various interesting examples, especially in random graphs which fall in our setting, for example, inhomogeneous Erd\H{o}s--R\'enyi, Sparse $W$-random graphs and constrained random graphs. The limit is explicitly identified in some special cases when the entries have constant variance and the limiting graphon has a multiplicative structure. We derive the order of the spectral norm when the entries satisfy a bit more restrictive condition. Inspired by the methodology of \citet{bryc2006spectral}, we use strong Gaussian approximation which imposes some restrictions on the entries of $A_N$. We show that their methodology can be extended to a large extent to cover the inhomogeneous setting. It would be interesting to derive the fluctuations of the spectral norm in the above setting. We leave this aspect of analysis for future work.

{\bf Outline of the article:} The article is arranged in the following way. In Section~\ref{sec:result} we introduce the graphon setting briefly and state the main results about the empirical spectral distribution (Theorem~\ref{thm:laplacian} and Corollary~\ref{corollary: dingjiang}). We identify the limiting spectral distribution in the multiplicative setting in Theorem~\ref{thm:multi_structure}. Subsection~\ref{sec:moment_description} is dedicated to the description of the moments of the limiting distribution and we compute some lower-order moments to give an idea of how the expression can be used. In Theorem~\ref{thm:spectral_norm_bound}, Theorem~\ref{thm: Normidentifylim} we derive almost sure bounds on the spectral norm in the centered case.  We discuss some examples which satisfy our assumptions in Section~\ref{sec:examples}.  In Section~\ref{sec:simulations} we show some simulations on how the LSD looks like for different graphons. Section~\ref{sec:proofESD} and Section~\ref{sec:proofSpectralnorm} are dedicated to the proof of the results on ESD and spectral norm respectively.

\subsection*{Acknowledgment}
 The research of RSH was supported by the MATRICS grant of SERB. A part of the work was done by AC for the master dissertation in Indian Statistical Institute, Kolkata.

\section{The setup and the results}\label{sec:result}
\subsection{Laplacian Matrix}  For any symmetric $N\times N$ matrix $A$ with eigenvalues $\lambda_{1},\cdots ,\lambda_{N}$, the \textit{empirical spectral distribution} (ESD) of $A$ is defined by the probability measure
$$
    \text{ESD}(A)=\frac{1}{N}\sum_{i=1}^{N}\delta_{\lambda_{i}}
$$
In this paper we would study a particular class of structured random matrices called the \textit{generalised Wigner matrices}. A generalised Wigner matrix is a random matrix $A_{N}=((X_{i\wedge j,i\vee j}))_{N\times N}$ satisfying
\begin{itemize}
    \item $\{X_{i,j}: 1\leq i\leq j\leq N\}$ are independent real valued random variables;
    \item $\mathbb{E}[X_{i,j}]<\infty$, \, \,  $\forall \,\,  1\leq i\leq j\leq N$;
    \item $\mathbb{E}[X_{i,j}^2]<\infty$, \,\,  $\forall \,\, 1\leq i\leq j\leq N$.
\end{itemize} 
The random variables also depend on $N$, but for notational simplification we remove the dependency. It is easy to draw parallel between such matrix and the adjacency matrix of a graph on $N$ vertices having edge weight $X_{i,j}$ on the edge between the vertices $i$ and $j$. As a result we would sometimes use the term adjacency matrix to denote the generalised Wigner matrices. Correspondingly we can define the Laplacian of $A_{N}$ as
\begin{align}\label{def: laplacian}
\Delta_{N}(i,j) =  \begin{dcases}
     A_{N}(i,j)\quad\text{  if }i\neq j\\
     -\sum_{k=1,k\neq i}^{N}A_{N}(i,k) \quad \text{ if $i=j$.}
   \end{dcases}
\end{align}
Since the row sum of $\Delta_N$ is zero and the infinitesimal generators of continuous-time Markov processes on finite state spaces are given by matrices with row-sums zero. Such matrices are also referred to as Markov matrix in literature (see \citet{bryc2006spectral}). This paper is mainly concerned with the mean centered version  of the above matrices, which we denote by
\begin{align}
    A_{N}^{0}&=\frac{1}{\sqrt{N}}(A_{N}-\mathbb{E}(A_{N}))\\
    \Delta_{N}^{0}&=\frac{1}{\sqrt{N}}(\Delta_{N}-\mathbb{E}(\Delta_{N})).\label{centered_laplacian}
\end{align}
Define the variance profile matrix corresponding to $A_{N}$ by $\Sigma_{N}=((\sigma_{i,j}^{2}))_{N\times N}$, where $\sigma_{i,j}^{2}=\mathbb{E}[(X_{i,j}-\mathbb{E}(X_{i,j}))^{2}]>0$ for all $1\leq i,j\leq N$. Some kind of convergence assumption on $\Sigma_N$ is necessary for getting a limit result for the above matrices. We shall assume that the variance profile gives rise to a graphon in the limit. 

\subsection{Graphons and Convergence of Graph Sequences}
Understanding large networks is a fundamental problem in modern graph theory and to properly define a limit object, an important issue is to have a good definition of convergence for graph sequences. The theory of graphons (\citet{lovasz2006limits}) as limits of dense graph sequences aims to provide a solution to this problem.

In our approach, we would define the variance profile matrix $\Sigma_N$ as a graphon sequence. The convergence of empirical spectral distribution is connected to the convergence of this graphon sequence associated with $\Sigma_N$. We  provide a brief introduction to graphon theory and for more details refer to \citet{lovasz2012large}.

A \textit{graphon} is a measurable function $W: [0,1]^2\rightarrow[0,1]$ such that $W(x,y)=W(y,x)$ for all $x,y\in[0,1]$. Let $\mathcal{W}$ to be the space of all graphons. To define the cut-metric on  $\mathcal{W}$, let $\Phi$ denote the set of all bijective, Lebesgue measure preserving $\sigma:[0,1]\rightarrow[0,1]$. For two graphons $W_1$ and $W_2$, the cut-distance is defined as 
\[
    d_{\Box}(W_1, W_2)=\sup_{S,T\subseteq[0,1]}\Big|\int_{S\times T}\left(W_1(x,y)-W_2(x,y)\right) \, dxdy\Big|,
    \]
where $S$ and $T$ ranges over all measurable subsets of $[0,1]$.

Then the cut metric is given by 
\begin{align*}
    \delta_{\Box}(W_{1},W_{2})=\inf_{\sigma\in\Phi}d_{\Box}(W_{1},W_{2}^{\sigma})
\end{align*}
where $W_{2}^{\sigma}=W_{2}(\sigma(x),\sigma(y))$. This forms a pseudo-metric and hence one says $W_1\sim W_2$ if $\delta_{\Box}(W_1, W_2)=0$. Let $\widetilde{\mathcal W}$ be the space of all equivalence classes. It is known that $(\widetilde{\mathcal W}, \delta_{\Box})$ is a compact metric space.
Every weighted graph $G$ can be associated with a graphon.
\begin{definition}\label{def:empiricalgraphon}
Consider a weighted graph $G=(V,E,(w_{e})_{e\in E})$ and for $j \in \{ 1, \dots, |V|-1\}$ define 
\[
I_{j}=\left[\frac{j-1}{|V|},\frac{j}{|V|}\right)\text{ and } I_{|V|}= \left[ 1- \frac{1}{|V|}\, , \, 1\right].
\]
 Then we define the \textit{empirical graphon of G} as
\begin{align*}
    W^{G}(x,y) =
   \begin{dcases}
     w_{e} \,\, \text{ if } e=(i,j)\in E(G),\quad (x,y)\in I_{i}\times I_{j}\\
     0 \,\, \text{ otherwise}.
   \end{dcases}
\end{align*}
\end{definition}

Observe that any empirical graphon $W^{G}\in\mathcal{W}$, if the weights lie in $[0,1]$. Using the cut metric, we are able to compare two graphs with different sizes and measure their similarity, which defines a type of convergence of graph sequences whose limiting object is the graphon. Another way of defining convergence of graphs is to consider the graph homomorphisms
\begin{definition}
For any graphon $W$ and a finite simple graph $F=(V,E)$ (without loops), define the \textit{homomorphism density} from $F$ to $W$ as
\begin{align}
    t(F,W)=\int_{[0,1]^{|V|}}\prod_{\{i,j\}\in E}W(x_{i},x_{j})\prod_{i\in V} dx_i.\label{def:graphhom}
\end{align}
\end{definition}
It is natural to think two graphon $W_{1}$ and $W_{2}$ are similar if they have similar homomorphism densities from any finite graph $G$. Let $\{W_{n}\}_{n\in \mathbb{N}}$ be a sequence of graphons. We say $\{W_{n}\}_{n\in \mathbb{N}}$ is \textit{convergent from the left} if $t(F,W_{N})$ converges for any finite simple (no loops, no multi-edges, no directions) graph F.

The homomorphism density characterises convergence under the cut metric. \citet[Theorem 11.5]{lovasz2012large} gives a   characterisation of convergence in the space $\mathcal{W}$. Let $\{W_{n}\}_{n\in\mathbb{N}}$ be a sequence of graphons in $\mathcal{W}_{0}$ and let $W\in\mathcal{W}_{0}$. Then $t(F,W_{n})\rightarrow t(F,W)$ for all finite simple graphs $F$ if and only if $\delta_{\Box}(W_{n},W)\rightarrow0$. We now describe the assumptions needed for our results. They are very similar to the ones mentioned in \citet{zhu2020graphon}.


\subsection{Limiting spectral distribution of Laplacian}\label{subsec:ESD of L}
Let $A_{N}$ be a $N\times N$ \textit{generalised Wigner matrix} with a variance profile matrix $\Sigma_{N}$ satisfying the following conditions:
\begin{enumerate}[label=\textbf{L.\arabic*}]
    \item \label{itm:A1} ({\bf Bounded variance})
    There exists a constant $C>0$ such that$$\mathrm{Var}( X_{i,j})\le C, \quad \forall 1\leq i,j\leq N, N\geq 1.$$ Without loss of generality we assume $C\leq 1$.
    \item \label{itm:A2} 
    ({\bf Lindeberg's Condition}) for any constant $\eta>0$,
    \begin{align}
        \lim_{n\rightarrow\infty}\frac{1}{N^{2}}\sum_{1\leq i,j\leq N}\mathbb{E}\left[|X_{i,j}-\mathbb{E}[X_{i,j}]|^{2}\mathbf{1}\left(|X_{i,j}-\mathbb{E}[X_{i,j}]|\geq \eta\sqrt{N}\right)\right]=0.
    \end{align}
    \item \label{itm:A3} ({\bf Graphon convergence})
    Consider the graph $$G^{\Sigma_{N}}=\bigg([N],\{(i,j):1\leq i\leq j\}, (\sigma_{i,j}^{2})_{1\leq i\leq j\leq N}\bigg)$$ and the corresponding empirical graphon $W_{N}$. We assume there exists a graphon $W\in\mathcal{W}$ such that
    \[
    \delta_{\Box}(W_{N},W)\rightarrow 0.
    \]
\end{enumerate}

\begin{remark}
There are multiple examples of random matrices where the above assumptions are satisfied. We deal later with some examples arising out of random graphs like inhomogeneous Erd\H{o}s-R\' enyi, $W$-sparse random graphs, constrained random graphs and stochastic block model. In some cases the assumptions were already verified in \citet{zhu2020graphon}.
\end{remark}
Before stating the result on convergence of ESD of the centered Laplacian matrix \eqref{centered_laplacian}, for the sake of completeness let us take a look at the result on the matrix $A_{N}$.
To describe the limiting moments we will need the definition of rooted planar tree and this play a crucial role also in the description of moments of the Laplacian.  

The \textit{rooted planar tree} is a planar graph with no cycles, with one distinguished vertex as a root, and with a choice of  ordering at each vertex. The ordering defines a way to explore the tree starting at the root. One of the algorithms used for traversing the rooted planar trees is \textit{depth-first search}. An enumeration of the vertices of a tree is said to have depth-first search order if it is the output of the depth-first search.

It was shown in \citet[Theorem 3.2]{zhu2020graphon} that under the assumptions \ref{itm:A1}--\ref{itm:A3}, 
\begin{align*}
    \lim_{N\rightarrow\infty} \text{ESD}\big(A_{N}^{0}\big)=\mu \text{ weakly almost surely},
\end{align*}
where $\mu$ denotes the unique probability measure identified by the following moments
\begin{equation}\label{moments:zhu}
    \int x^{2k}d\mu=\sum_{j=1}^{C_{k}}t(T_{j}^{k+1},W),\ \int x^{2k+1}d\mu=0,  \,\, k\ge 0,
\end{equation} 
where $T_{j}^{k+1}$ is the $j^{th}$ rooted planar tree with $k+1$ vertices and $C_{k}$ is the $k^{th}$ Catalan number.

Then for the centered Laplacian defined in \eqref{centered_laplacian} we have the following result
\begin{theorem}\label{thm:laplacian}
Under assumptions \ref{itm:A1}--\ref{itm:A3},
\begin{align*}
    \lim_{N\rightarrow\infty} \text{ESD}\big(\Delta_{N}^{0}\big)=\nu \text{ weakly in probability}
\end{align*}
where $\nu$ is the unique symmetric probability measure on $\mathbb{R}$. Further if there exists an open set $U\subseteq [0,1]^{2}$ such that $W>0$ on $U$, then $\nu$ has unbounded support.

\end{theorem}

\subsection{Identification of Limiting Spectral Distribution}
The limiting spectral distribution can be identified with standard distributions under certain simplifications. Along with the assumption of boundedness (\ref{itm:A1}) we assume that
\begin{enumerate}[resume,label=\textbf{L.\arabic*}]
    \item $\mathbb{E}X_{i,j}=\mu_{N}$ and $\sigma_{i,j}^{2}=\sigma_{N}^{2}>0$ for all $1\leq i,j\leq N, N\geq 1$.\label{itm: A4}
    \item There exists $\delta>0$ such that 
    $$
        \sup_{1\leq i,j\leq N,N\geq 1}\mathbb{E}\left[\left|\frac{X_{i,j}-\mu_{N}}{\sigma_{N}}\right|^{2+\delta}\right]<\infty.
    $$\label{itm:A5}
\end{enumerate}
Defining $\{\lambda_{i}(A): 1\le i\le N\}$ as the eigenvalues of a $N\times N$ matrix $A$, we have the following result.
\begin{corollary}
\label{corollary: dingjiang}
Suppose \ref{itm:A1}, \ref{itm: A4} and \ref{itm:A5} holds. Set
    \begin{align*}
        \widetilde{F}_{N}(x)=\frac{1}{N}\sum_{i=1}^{N}I\left\{\frac{\lambda_{i}(\Delta_{N})-N\mu_{N}}{\sqrt{N}\sigma_{N}}\leq x\right\},\quad\forall x\in\mathbb{R}.
    \end{align*}
Then, as $N\rightarrow\infty$, $\widetilde{F}_{N}$ converges to a distribution function $F$ in probability where $F$ denotes the distribution of the free additive convolution $\gamma_{M}$ of the semicircle law and the standard normal distribution. 
\end{corollary}
The above result identifies the limiting measure in terms of additive free convolution of two measures.  The identification of the limiting measure can be achieved in the case when $W$ has a \textit{multiplicative structure}, that is
\begin{align}\label{eq: multi_structure}
    W(x,y)=r(x)r(y),\quad x,y\in[0,1]
\end{align}
for some continuous function $r:[0,1]\rightarrow[0,1]$. The statement is based on the theory of self-adjoint operators affiliated with a $W^{\star}$-probability space. It is important to note that we do not assume the initial variances to have a multiplicative structure but the multiplicative structure only appears in the limit. We recall some definitions mentioned in \citet{chakrabarty2018spectra}. For details of free probability of unbounded operators we refer to \citet{anderson2010introduction}.

Recall $(\mathcal A, \tau)$ is called a $W^\ast$ probability space if $\mathcal A$ is a closed (in weak operator topology) $C^\ast$ algebra of bounded operators on some Hilbert space and $\tau$ is a state. A densely defined operator $T$ is said to be affiliated to $\mathcal A$ if for every bounded measurable function $h$, $h(T)\in \mathcal A$. For an affiliated operator $T$, its law $\mathcal{L}(T)$ is the unique probability measure on $\mathbb R$ satisfying 
$$
    \tau(h(T))=\int_{\mathbb R}h(x)(\mathcal{L}(T))dx.
$$
For two or more self adjoint operators $T_{1},\cdots, T_{n}$, a description of their joint distribution is a specification of 
$$
    \tau(h_{1}(T_{i1}),\cdots, h_{k}(T_{ik})),
$$
for all $k\geq 1$, all $i_{1},\cdots i_{k}\in \{1,2,\cdots ,n\}$, and all bounded measurable functions $h_{1},\cdots, h_{k}$ from $\mathbb R$ to itself.


\begin{definition}
Let $(\mathcal{A},\tau)$ be a $W^{\star}$-probability space and $a_{1}, a_{2}\in\mathcal{A}$. Then $a_{1}$ and $a_{2}$ are freely independent if
$$
    \tau(p_{1}(a_{i_{n}}),\cdots, p_{n}(a_{i_{n}}))=0
$$
for all $n\geq 1$, all $i_{1},\cdots, i_{n}\in \{1,2\}$ with $i_{j}\neq i_{j+1}, j=1,\cdots,n-1$, and all polynomials $p_{1},\cdots, p_{n}$ in one variable satisfying 
$$
    \tau(p_{j}(a_{i_{j}}))=0,\quad j=1,\cdots, n
$$
For operators $a_{1},\cdots, a_{k}$ and $b_{1},\cdots, b_{m}$ affiliated with $\mathcal{A}$, the collections $(a_{1},\cdots, a_{k})$ and $(b_{1},\cdots b_{m})$ are freely independent if and only if 
$$
    p(h_{1}(a_{1}),\cdots h_{k}(a_{k})) \text{ and } q(g_{1}(b_{1}),\cdots, g_{m}(b_{m}))
$$
are freely independent for all bounded measurable $h_{1},\cdots, h_{k}$ and $g_{1},\cdots, g_{m}$, and all polynomials $p$ and $q$ in $k$ and $m$ non-commutative variables, respectively. It is immediate that the two operators in the above display are bounded, and hence belong to $\mathcal{A}$.
\end{definition}
Now we are ready to state our theorem.
\begin{theorem}\label{thm:multi_structure}
If $W$ is as in \eqref{eq: multi_structure}, then under \ref{itm:A1}-\ref{itm:A3} with the cut-metric convergence in \ref{itm:A3} replaced by $W_{N}\overset{L_{1}}{\rightarrow}W$ the limiting measure $\nu$ can be identified as, 
\begin{align}
    \nu=\mathcal{L}\left(r^{1/2}(T_{u})T_{s}r^{1/2}(T_{u})+\alpha r^{1/4}(T_{u})T_{g}r^{1/4}(T_{u})\right)
\end{align}
where 
$$
    \alpha=\left(\int_{0}^{1}r(x)dx\right)^{1/2}
$$
Here, $T_{g}$ and $T_{u}$ are commuting self-adjoint operators affiliated with a $W^{\star}$-probability space $(\mathcal{A},\tau)$ such that, for bounded measurable functions $h_{1},h_{2}$ from $\mathbb R$ to itself,
\begin{align}
    \tau\left(h_{1}(T_{g})h_{2}(T_{u})\right)=\left(\int_{-\infty}^{\infty}h_{1}(x)\phi(x)dx\right)\left(\int_{0}^{1}h_{2}(u)du\right)
\end{align}
with $\phi$ the standard normal density. Furthermore, $T_{s}$ has a standard semicircle law and is freely independent of $(T_{g},T_{u})$.
\end{theorem}

Such identification was first achieved in the results of \citet{chakrabarty2018spectra} in the case of inhomogeneous Erd\H{o}s-R\' enyi random graphs. We show that similar limit appears in the general case too. 
\begin{remark}
Under such multiplicative structure the limiting spectral distribution of the matrix $A_{N}$ can also be identified. If we keep the assumption \eqref{eq: multi_structure} on the limiting graphon $W$ and the empirical graphon $W_{N}$ corresponding to the variance profile as in Theorem \ref{thm:multi_structure}, then the limiting spectral distribution of the scaled and centered matrix $A_N^0$ is given by
\begin{align}\label{eq:multi_adj_limit}
    \mu=\mathcal{L}\left(r^{1/2}(T_{u})T_{s}r^{1/2}(T_{u})\right)
\end{align}
where $T_{u}$ and $T_{s}$ are as defined in Theorem \ref{thm:multi_structure}. One should note that the measure $\mu$ in \ref{eq:multi_adj_limit} is the same as the free multiplicative convolution of the standard semicircle law and the law of $r(U)$, where $U$ is a standard uniform random variable. 
\end{remark}

\subsection{Description of the moments}\label{sec:moment_description}
In this subsection, we briefly describe the moments of the limiting measure $\nu$. It turns out to be an interesting combinatorial problem to come up with an expression for the moments. To describe the limiting moments, we need to introduce some notions.
Fix $k$ in $2\mathbb N$. Consider the multiset $\{m_{1},m_{2},\cdots,m_{k},n_{1},n_{2},\cdots n_{k}\}\in \mathbb N$ such that $\sum_{p=1}^{k}(m_{p}+n_{p})=k$ and $m=\sum_{p=1}^{k}m_{p}$ is even. Then consider a rooted planar tree $T$ on $(\frac{m}{2}+1)$ vertices and take $\widetilde{i}=i_{1}i_{2}\cdots i_{m+1}$ to be the depth first search on $T$. Then we must have $i_{m+1}=i_{1}$. So for notational simplicity we identify $1$ by $m+1=1+\sum_{p=1}^{k}m_{p}$. Then consider $\{s_{1},\cdots, s_{\eta_{s}}\}\subseteq\{1,2,\cdots, k\}$ such that the indices in $\widetilde{i}$ corresponding to the set $\{1+\sum_{1}^{s_{1}}m_{p},\cdots, 1+\sum_{1}^{s_{\eta_{s}}}m_{p}\}$, that is $\{i_{1+\sum_{1}^{s_{1}}m_{p}},\cdots, i_{1+\sum_{1}^{s_{\eta_{s}}}m_{p}}\}$ represents the $s^{th}$ vertex of $T$. Observe that $\eta_{s}$ can be equal 0. Now suppose that $\sum_{j=1}^{\eta_{s}}n_{s_{j}}$ is even for all vertex $s$ in $T$. Then define $\widehat{n}_{s}=\frac{1}{2}\sum_{j=1}^{\eta_{s}}n_{s_{j}}$ (if $\eta_{s}$ is 0, then the sum is also 0). Construct a new graph $\widetilde{T}$ by attaching $\widehat{n}_{s}$ many vertices to the vertex $s$ of $T$, for all $s\in \{1,2,\cdots, \frac{m}{2}+1\}$. Only modification that we are doing to $T$ is by adding leaf nodes. Hence the modified graph is still a tree. An example of such modification at the vertex $s$ is shown in Figure \ref{fig:Modified_s}.\\

For $t\in \mathbb N$, define 
$$
    \mathbf{m}_{t}=\mathbb{E}\left[Z^{t}\right], \ Z\sim N(0,1).
$$
Then define the function
\begin{equation}\label{def:f}
    f(T)=
    \begin{cases}
    \prod_{s=1}^{\frac{m}{2}+1}\mathbf{m}_{2\widehat{n}_{s}} & \text{ if } 2\widehat{n}_{s} \text{ is even for all } s,\\
    0 & \text{ otherwise}.
    \end{cases}
\end{equation}

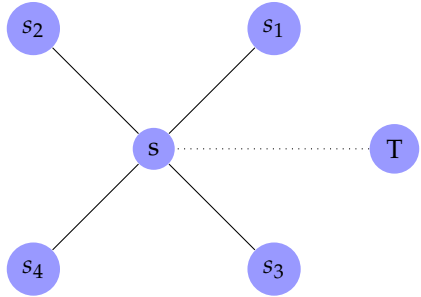
\begin{figure}[!ht]
\centering
\begin{tikzpicture}
  [scale=.8,auto=left,every node/.style={circle,fill=blue!40}]
  \node (n6) at (0,0) {s};
  \node (n5) at (2,2)  {$s_{1}$};
  \node (n4) at (-2,2) {$s_{2}$};
  \node (n3) at (-2,-2)  {$s_{4}$};
  \node (n2) at (2,-2)  {$s_{3}$};
  \node (n1) at (4,0) {T};

  \foreach \from/\to in {n6/n5,n6/n4,n6/n3,n6/n2}
    \draw (\from) -- (\to);
  \foreach \from/\to in {n1/n6}
    \draw[dotted] (\from) -- (\to);
\end{tikzpicture}
\caption{The modification of graph $T$ at vertex $s$ with $\widehat{n}_{s}=4$}\label{fig:Modified_s}
\end{figure}

\begin{example}
We provide an example of such an modification. Consider $k=12$, $\widetilde{m}=\{m_{i}\}_{i=1}^{12}$ such that $m_{i}=2$  for $i=1,2,4$ and $0$ otherwise and finally $\widetilde{n}=\{n_{i}\}_{i=1}^{12}$ such that $n_{i}=2$ for $i=1,2,5$ and $0$ otherwise. Then $m=\sum_{i=1}^{k}m_{i}=6$. Hence we consider the rooted planar tree $T_{4}$ as in Figure \ref{fig:Example_graph}.

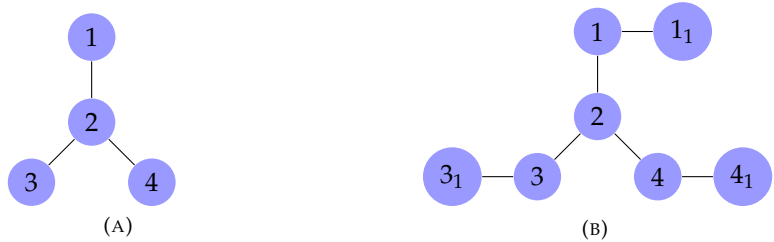
\begin{figure}%
    \centering
   \subfloat[][\centering\label{fig:Example_graph}]
   {\begin{tikzpicture}[scale=.8,auto=left,every node/.style={circle,fill=blue!40}]
  \node (n6) at (0,0) {$2$};
  \node (n5) at (0,1.414)  {$1$};
  \node (n4) at (1,-1) {$4$};
  \node (n3) at (-1,-1)  {$3$};
  \foreach \from/\to in {n6/n5,n6/n4,n6/n3}
    \draw (\from) -- (\to);
\end{tikzpicture}}%
    \hspace{0.2\textwidth}
   \subfloat[][\centering\label{fig: Example_modified}]
  {\begin{tikzpicture}
  [scale=.8,auto=left,every node/.style={circle,fill=blue!40}]
  \node (n6) at (0,0) {$2$};
  \node (n5) at (0,1.414)  {$1$};
  \node (n4) at (1,-1) {$4$};
  \node (n3) at (-1,-1)  {$3$};
  \node (n2) at (1.414,1.414)  {$1_{1}$};
  \node (n1) at (-2.414,-1) {$3_{1}$};
  \node (n7) at (2.414,-1) {$4_{1}$};
  \foreach \from/\to in {n6/n5,n6/n4,n6/n3,n5/n2,n4/n7,n3/n1}
    \draw (\from) -- (\to);
\end{tikzpicture}}%
    \caption{(a) Rooted planar tree $T_{4}$, (b) Modified tree $\widetilde{T}_{4}$}%
\end{figure}

Then consider the depth first search $1\rightarrow 2\rightarrow 3\rightarrow 2\rightarrow 4\rightarrow 2\rightarrow 1=i_{1}\rightarrow i_{2}\rightarrow i_{3}\rightarrow i_{4}\rightarrow i_{5}\rightarrow i_{6}\rightarrow i_{7}$. Now observe that 
\begin{align*}
    1+\sum_{i=1}^{j}m_{i}=
    \begin{cases}
        3 & j=1\\
        5 & j=2,3\\
        7 & \text{otherwise}
    \end{cases}
\end{align*}
Observe that by definition $i_{1}=i_{7}$ corresponds to the vertex $1$. Now let us look at $1^{st}$ vertex, that is vertex $1$. Observe that $i_{1+\sum_{l=1}^{j}m_{l}}$ falls on vertex $1$ if $4\leq j\leq 12$. Then by definition $\eta_{1}=9$ and the set $\{s_{1},\cdots,s_{\eta_{s}}\}$ for $s=1$ is $\{4,5,\cdots 12\}$. Then $\widehat{n}_{1}=\frac{1}{2}\sum_{j=1}^{9}n_{s_{j}}=1$. Then we look at vertex $2$. Observe that $i_{2},i_{4},i_{6}$ corresponds to vertex $2$. But $2,4,6 \not\in\{1+\sum_{l=1}^{j}m_{l}:\forall 1\leq j\leq 12\}$. Hence $\eta_{2}=0$. Lets look at $4^{rd}$ vertex. Observe that only $i_{1+\sum_{l=1}^{j}m_{l}}$ for $j=2,3$ falls on vertex 4. Then by definition $\eta_{4}=2$ and the set $\{s_{1},\cdots, s_{\eta_{s}}\}$ for $s=4$ is $\{2,3\}$. Then $\widehat{n}_{4}=1$. Similarly we can show that $\widehat{n}_{3}=1$. The modified graph is as in Figure \ref{fig: Example_modified}
\end{example}
One thing to note is that the modification is independent of the labeling of the vertices. It only depends upon the number of overlap certain indices have with the vertices.

Let $\mathcal{P}_{2k}$ be the multiset of all numbers $(m_{1},\cdots, m_{2k}, n_{1},\cdots, n_{2k})$ which appear when we expand $(a+b)^{2k}$ for two non-commutative variables $a$ and $b$ and write it as
$$(a+b)^{2k} = \sum_{(m_{1},\cdots, m_{2k}, n_{1},\cdots, n_{2k})\in \mathcal{P}_{2k}} \prod_{i=1}^{2k}a^{m_i}b^{n_i}.$$
Observe that $\sum_{j=1}^{2k}(m_{j}+n_{j})=2k$.
We identify the moments of the limiting measure $\nu$ through the following formula.
\begin{equation}\label{eq:moment-formula}
   \int x^{2k}d\nu=\sum_{\mathcal{P}_{2k}}\sum_{r=1}^{C_{\frac{m}{2}}}t\left(\widetilde{T}_{r}^{\frac{m}{2}+1},W\right)f\left(T_{r}^{\frac{m}{2}+1}\right)\one_{\{m\in 2\mathbb{N}\cup\{0\}\}},\ \int x^{2k+1}d\nu=0,\ \forall k\geq 0
\end{equation}
where $T_{r}^{q}$ denote the $r^{th}$ rooted planar tree on $q$ many vertices with  $\widetilde{T}_{r}^{q}$ the corresponding modification as stated in Section \ref{sec:moment_description}.

\begin{remark}
Since the expression of moments in \eqref{eq:moment-formula} look complicated, we elucidate by calculating the second and the fourth moments.\\
First, let us look at the second moment. The possible choices for the vector $(m_{1},m_{2},n_{1},n_{2})$ are $$(2,0,0,0), (1,0,1,0), (0,1,1,0) \text{ and } (0,0,2,0). $$ Observe that for the second and third choice, there wont be any contributions since $m=m_{1}+m_{2}$ is odd. For the first choice number of trees is $C_{m/2}=C_{1}=1$, that is a single edge ($T_{1}^{2}$). Since here $n_{1}=n_{2}=0$, then we don't need any modification of the tree. Thus 
\begin{align*}
    t\left(\widetilde{T}_{1}^{2},W\right)=\int Wdxdy\text{ and }f(T_{1})=1
\end{align*}
Now for the choice $(0,0,2,0)$, $m=0$ and hence number of trees are $C_{0}=1$. The tree is basically a single vertex $T_{1}^{1}$. Then the depth first search would  yield just $i_{1}=1$. Observe that $1+m_{1}=1+m_{1}+m_{2}=1$. So $\widehat{n}_{1}=\frac{1}{2}(n_{1}+n_{2})=1$. Hence the modified tree is same as $T_{1}^{2}$. Thus the contribution becomes 
\begin{align*}
    t(\widetilde{T}_{1}^{1},W)=t(T_{1}^{2},W)=\int W(x,y) dx\, dy\text{ and }f(T_{1}^{1})=1
\end{align*}
Hence the second moment is 
\begin{align*}
    2\int W(x,y)dx\, dy
\end{align*}
Going as in the second moment we can determine the contributory terms in the expression of the fourth moment and their corresponding contributions. We provide explicit calculations for a few terms, the rest follows similarly. Let us look at the term corresponding to $(2,0,0,0,2,0,0,0)$. Here $m=2$ and $C_{m/2}=C_{1}=1$. So the contributory rooted planar tree is a single edge $(T_{1}^{2})$. The depth first search would give the closed walk, $i_{1}\rightarrow i_{2}\rightarrow i_{3}=i_{1}$. Now observe that
$$
    1+\sum_{l=1}^{j}m_{l}=3,\quad j=1,2,3,4.
$$
Then the number of overlaps of the form $i_{1+\sum_{l=1}^{j}m_{l}}$ with the first vertex is $\eta_{1}=4$ and the values of such $j$ are $j=1,2,3,4$. Hence 
$$
    \sum_{j=1}^{4}n_{j}=2\implies \widehat{n}_{1}=1
$$
Then the modification of the first vertex is given by adding a single edge to the vertex $1$. Observe that no such overlap is possible for the second vertex. Hence there is no modification for the second vertex. So for this case the modified tree $\widetilde{T}_{1}^{2}$ is given as in Figure \ref{fig:modification_1} and the contribution from this tree would be
$$
    t\left(\widetilde{T}_{1}^{2},W\right)\mathbf{m}_{2}
$$
Next let us look at the term corresponding to(1,1,0,0,2,0,0,0). Here $m=2$ and $C_{m/2}=1$. So again the contributory rooted planar tree is $T_{1}^{2}$. The depth first search is same as the previous case. Observe that here
$$
    1+\sum_{l=1}^{j}m_{l}=
    \begin{cases}
        2 & j=1\\
        3 & j=2,3,4
    \end{cases}
$$
Then the number of overlaps of the form form $i_{1+\sum_{l=1}^{j}m_{l}}$ with the first vertex is $\eta_{1}=3$ and the values of such $j$ are $j=2,3,4$. But then it is easy to observe that $\widehat{n}_{1}=0$. Hence there is no modification of the first vertex. There is a single overlap with the second vertex given by $i_{1+m_{1}}$. Then $\widehat{n}_{2}=1$ and hence we attach a single leaf node to the second vertex. So here again the modified tree is the same as the previous case, but with different labeling, given in Figure \ref{fig: modified_2}. Going similarly for the other terms, the fourth moment is given by
$$
    (2+4\mathbf{m}_{2}+\mathbf{m}_{4})t(T_{1}^{3},W)
$$
where $T_{1}^{3}$ is the rooted planar tree on 3 vertices. (Note that there are 2 rooted planar trees on 3 vertices, but the homomorphism density corresponding to both are equal, hence we can take any one of the trees). 
\begin{figure}%
    \centering
    \subfloat[][\centering\label{fig:modification_1}]{{\begin{tikzpicture}
  [scale=.8,auto=left,every node/.style={circle,fill=blue!40}]
  \node (n6) at (0,0) {$1$};
  \node (n5) at (0,2)  {$2$};
  \node (n4) at (2,0) {$1_{1}$};

  \foreach \from/\to in {n6/n5,n6/n4}
    \draw (\from) -- (\to);
\end{tikzpicture}}}%
\hspace{0.2\textwidth}
    \subfloat[][\centering\label{fig: modified_2}]{{\begin{tikzpicture}
  [scale=.8,auto=left,every node/.style={circle,fill=blue!40}]
  \node (n6) at (0,0) {$2$};
  \node (n5) at (0,2)  {$1$};
  \node (n4) at (2,0) {$2_{1}$};

  \foreach \from/\to in {n6/n5,n6/n4}
    \draw (\from) -- (\to);
\end{tikzpicture}}}%
    \caption{Modified trees for the fourth moment}%
\end{figure}
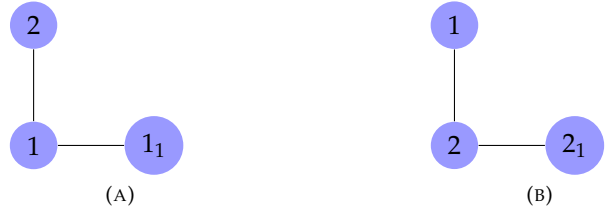

\end{remark}

\subsection{Spectral Norm of Laplacian}  In this subsection we discuss the asymptotics of the \textit{spectral norm} of the Laplacian matrix $\Delta_{N}$. In the case when entries are taken to be i.i.d. some results are known. The first order asymptotics, that is, the order of growth of the Laplacian was described in \citet{bryc2006spectral}.

For a $N\times N$ matrix $\mathbf{M}$, the spectral norm of $\mathbf{M}$ is defined as 
\begin{align*}
    \|\mathbf{M}\|=\sup_{\substack{\mathbf{x}\in\mathbb{R}^{N}\\ \|x\|_{2}=1}}\|\mathbf{Mx}\|_{2}
\end{align*}
where $\|\mathbf{x}\|_{2}=\sqrt{x_{1}^{2}+\cdots+ x_{N}^{2}}$ for $\mathbf{x}=(x_{1},x_{2},\cdots,x_{N})^{T}\in\mathbb{R}^{N}$.

Observe that using symmetry of $A_{N}$ and $\Delta_{N}$ we have
\begin{align}\label{eq:symmetrynorm}
    \|A_{N}\|=\max\{\lambda_{\max}(-A_{N}),\lambda_{\max}(A_{N})\},\ \|\Delta_{N}\|=\max\{\lambda_{\max}(-\Delta_{N}),\lambda_{\max}(\Delta_{N})\}
\end{align}
In order to proceed with our results we consider the following assumptions on the entries of $A_{N}$ in line with the assumptions in Section \ref{subsec:ESD of L}.
\begin{enumerate}[label=\textbf{S.\arabic*}]
\item \label{itm:A1_Norm} ({\bf Bounded variance})
Let $\sigma_{i,j}^2= \Var(X_{ij})$ and let $A_{1}=\inf_{i,j\geq 1}\sigma_{i,j}^{2}$ and $A_{2}=\sup_{i,j\geq 1}\sigma_{i,j}^{2}$. Suppose $A_1>0$ and $A_2<\infty$. Without loss of generality we can take $A_2\leq1$.

    \item \label{itm:A2_Norm} ({\bf Higher moments})
    There exists $\delta>0$ and $0<K<\infty$ such that $$\mathbb{E}\left[\left|X_{i,j}-\mathbb{E}X_{i,j}\right|^{6+\delta}\right]\leq K,\quad \text{for all } i,j\in\mathbb{N}.$$
    \item \label{itm:A3_Norm} ({\bf Graphon convergence}) 
    Consider the graph $$G^{\Sigma_{N}}=\bigg([N],\{(i,j):1\leq i\leq j\}, (\sigma_{i,j}^{2})_{1\leq i\leq j\leq N}\bigg)$$ and the corresponding empirical graphon $W_{N}$. Suppose there exists a graphon $W\in\mathcal{W}_{0}$ such that
    \[
    \delta_{\Box}(W_{N},W)\rightarrow 0.
    \]
\end{enumerate}
The assumption \ref{itm:A3} is restated here as assumption \ref{itm:A3_Norm} for convenience. It is needed for the convergence of the ESD of adjacency matrix $A_{N}$ to a compactly supported probability measure.
\begin{remark}It can be easily seen that assumption \ref{itm:A2_Norm} implies assumption \ref{itm:A2} (Lindeberg's Condition) for any constant $\eta>0$, that is
    \begin{align}
        \lim_{n\rightarrow\infty}\frac{1}{N^{2}}\sum_{1\leq i,j\leq N}\mathbb{E}\left[|X_{i,j}-\mathbb{E}X_{i,j}|^{2}\mathbf{1}\left(|X_{i,j}|\geq \eta\sqrt{N}\right)\right]=0, \quad \forall\eta>0
    \end{align}
\end{remark}
Under the above set of assumptions we have the following result on spectral norm of the Laplacian matrix $\Delta_{N}$.
\begin{theorem}\label{thm:spectral_norm_bound}
Suppose that $A_{1}=\inf_{i,j\geq 1}\sigma_{i,j}^{2}$ and $A_{2}=\sup_{i,j\geq 1}\sigma_{i,j}^{2}$, and $\mathbb{E}X_{i,j}=0$ for all $1\leq i\leq j$. Then under the above set of assumptions we have
\begin{equation}\label{eq:upperlowerbounds}
    \liminf_{N\rightarrow\infty}\frac{\|\Delta_{N}\|}{\sqrt{2N\log N}}\geq A_{1}^{\frac{1}{2}}\ a.s., \text{ and } \limsup_{N\rightarrow\infty}\frac{\|\Delta_{N}\|}{\sqrt{2N\log N}}\leq (2A_{2})^{\frac{1}{2}}\ a.s.
\end{equation}

\end{theorem}
We are able to identify the in probability limits under an additional assumption, which we state below
\begin{enumerate}[resume,label=\textbf{S.\arabic*}]
    \item \label{itm:A4_Norm} ({\bf Uniform row stochasticity}) For some $0<r\leq 1$,
    $$
        \lim_{N\rightarrow\infty}\sup_{i\in\mathbb{N}}\left|\frac{1}{N}\sum_{j=1}^{n}\sigma_{i,j}^{2}-r\right|=0.
    $$
    
\end{enumerate}
Recall that by assumption \ref{itm:A1_Norm} we have taken $A_{2}\leq 1$, hence we can take $0<r\leq 1 $. Then we have the following theorem
\begin{theorem}\label{thm: Normidentifylim}
Under the assumptions \ref{itm:A1_Norm}-\ref{itm:A4_Norm} we have
$$
    \lim_{N\rightarrow\infty}\frac{\|\Delta_{N}\|}{\sqrt{2N\log N}}=\sqrt{r}\ \text{ in probability}
$$
where $\|\cdot\|$ is the spectral norm.
\end{theorem}
The proof of Theorem \ref{thm: Normidentifylim} is similar to the proof of Theorem \ref{thm:spectral_norm_bound} and hence it will be omitted. We also look at the situation where the entries of the adjacency matrix have non-zero mean, that is $\exists N\geq 1$ and $1\leq i\leq j\leq N$ such that $\mathbb{E}A_{N}(i,j)\neq 0$. The limit of the spectral norm can be identified under certain assumptions which we outline below. It is an immediate corollary of Theorem~\ref{thm:spectral_norm_bound}.
\begin{enumerate}[resume,label=\textbf{S.\arabic*}]
    \item There exists a constant $m\geq 0$ such that
    \begin{align*}
        \lim_{N\rightarrow\infty}\frac{\left\|\mathbb{E}\Delta_{N}\right\|}{N}=m
    \end{align*}\label{itm:A5_Norm}
\end{enumerate}
Under the above additional assumption we have the following \textcolor{blue}{corollary},
\begin{corollary}\label{thm:spectral_norm_mean}
Under the set of assumptions \ref{itm:A1_Norm}-\ref{itm:A3_Norm}, along with the additional assumption \ref{itm:A5_Norm}, we have
\begin{align*}
    \lim_{N\rightarrow\infty}\frac{\left\|\Delta_{N}\right\|}{N}=m,\quad \text{a.s.}
\end{align*}
\end{corollary}

\section{Some examples}\label{sec:examples}
This section is devoted to examples and we show various models, mainly related to random networks which fall within the purview of the assumptions mentioned in the last section.
\subsection{Inhomogenous Erd\H{o}s-R\'enyi graphs}
We define the setting of inhomogenous Erd\H{o}s-R\'enyi graph following \citet{chakrabarty2018spectra}. Let $f:[0,1]^{2}\rightarrow[0,1]$ be a continuous function, satisfying 
\begin{align}
    f(x,y)=f(y,x)\quad\forall x,y\in[0,1].
\end{align}
Consider $(\varepsilon_{N}:N\geq 1)$ to be a sequence of fixed positive real numbers satisfying
\begin{align}\label{eq:inhomo_epsilon}
    \lim_{N\rightarrow\infty}\varepsilon_{N}=0,\quad\lim_{N\rightarrow\infty}N\varepsilon_{N}=\infty
\end{align}
Consider the random graph $\mathbb{G}_{N}$ on vertices $\{1,\cdots N\}$ where, for each $(i,j)$ with $1\leq i<j\leq N$, an edge is present between vertices $i$ and $j$ with probability $\varepsilon_{N}f\left(\frac{i}{N},\frac{j}{N}\right)$, independently of other pair of vertices. In particular, $\mathbb{G}_{N}$ is an undirected graph with no self edges and no multiple edges. Boundedness of $f$ ensures that $\varepsilon_{N}f\left(\frac{i}{N},\frac{j}{N}\right)\leq 1$ for all $1\leq i<j\leq N$ when $N$ is large enough. Without loss of generality we assume that $\varepsilon_{N}< 1$ for all $N\geq 1$.  If $f\equiv c$ with $c$ a constant, then $\mathbb{G}_{N}$ is the (homogeneous) Erd\H{o}s-R\'enyi graph with edge retention probability $\varepsilon_{N}c$. 
The adjacency matrix is denoted by $A_{N}$. Clearly, $A_{N}$ is a symmetric random matrix whose diagonal entries are zero and whose upper triangular entries are independent Bernoulli random variables, i.e.
\begin{align*}
    A_{N}(i,j)\sim \text{Ber}\left(\varepsilon_{N}f\left(\frac{i}{N},\frac{j}{N}\right)\right).
\end{align*}
Define $\Delta_{N}$ to be the Laplacian of the graph $\mathbb{G}_{N}$ corresponding to $A_{N}$ as in definition \ref{def: laplacian}.
The following theorem, which is a restatement of Theorem 1.2 in \citet{chakrabarty2018spectra} states  the existence of limiting spectral distribution of $\Delta_{N}$ under suitable scaling.

\begin{prop}\citet[Theorem 1.2]{chakrabarty2018spectra}\label{thm: Inhomothm}
There exists an unique, symmetric probability measure $\nu$ on $\mathbb{R}$ such that
$$
    \lim_{N\rightarrow\infty}\text{ESD}\left(\frac{1}{\sqrt{
    N\varepsilon_{N}}}\left(\Delta_{N}-D_{N}\right)\right)=\nu \text{ weakly in probability}
$$
where 
\begin{align*}
    D_{N}=\text{diag}\left(\mathbb{E}\Delta_{N}(1,1),\cdots, \mathbb{E}\Delta_{n}(N,N)\right)
\end{align*}
Furthermore, if $f\neq 0$, then the support of $\nu$ is unbounded.
\end{prop}
\begin{proof}
Observe that using Lemma 2.1 in \citet{chakrabarty2018spectra} we have 
\begin{align*}
    \lim_{N\rightarrow\infty}L\left(\frac{1}{\sqrt{N\varepsilon_{N}}}\text{ESD}(\Delta_{N}^{0}),\frac{1}{\sqrt{N\varepsilon_{N}}}\text{ESD}(\Delta_{N}-D_{N})\right)=0 \text{ in probability},
\end{align*}
where $L$ is the Levy metric\footnote{
The 
 \textit{L\'evy--Prokhorov metric} $L:\mathcal{P}(\mathbb{R})^{2}\rightarrow [0,+\infty)$ between two probability measures $\mu$ and $\nu$ is given by
\begin{align}
    L(\mu,\nu)=\inf\{\epsilon>0|\mu(A)\leq\nu(A^{\epsilon})+\epsilon\text{ and } \nu(A)\leq\mu(A^{\epsilon})+\epsilon\quad\forall A\in \mathcal{B}(\mathbb{R})\}
\end{align}
where $\mathcal{B}(\mathbb{R})$ denotes the Borel $\sigma$-algebra on $\mathbb{R}$ and $A^{\epsilon}$ is the $\epsilon$-neighbourhood of $A$.
} and
\begin{align*}
    \Delta_{N}^{0}=\Delta_{N}-\mathbb{E}\Delta_{N}.
\end{align*}
Hence it is enough to look at the limiting spectral distribution of $\frac{1}{\sqrt{N\varepsilon_{N}}}\text{ESD}(\Delta_{N}^{0})$. Define $B_{N}=\frac{1}{\sqrt{\varepsilon_{N}}}\left(A_{N}-\mathbb{E}A_{N}\right)$. Then $\frac{1}{\sqrt{\varepsilon_{N}}}\Delta_{N}^{0}$ is the centered Laplacian corresponding to $B_{N}$. \\
Observe that $\mathbb{E}B_{ij}^{(N)}=0$ and further it is easy to observe that due to $N\varepsilon_{N}\rightarrow\infty$ and $|B_{ij}^{(N)}|\leq \frac{2}{\sqrt{\varepsilon_{N}}}$, we must have
\begin{align*}
    \lim_{N\rightarrow\infty}\frac{1}{N}\sum_{1\leq i,j\leq N}\mathbb{E}\left[\left|B_{ij}^{(N)}\right|^{2}\one_{\left\{\left|B_{ij}^{(N)}\right|\geq \eta\sqrt{N}\right\}}\right]=0,\ \forall\eta>0
\end{align*}
Hence assumption \ref{itm:A2} is satisfied. Now consider $\Sigma^{(N)}$ to be the variance profile matrix of $B_{N}$. Then for $i\neq j$
\begin{align*}
    \Sigma_{ij}^{(N)}=\frac{\varepsilon_{N}f\left(\frac{i}{N},\frac{j}{N}\right)\left(1-\varepsilon_{N}f\left(\frac{i}{N},\frac{j}{N}\right)\right)}{\varepsilon_{N}}=f\left(\frac{i}{N},\frac{j}{N}\right)-\varepsilon_{N}f\left(\frac{i}{N},\frac{j}{N}\right)^{2}
\end{align*}
and $\Sigma_{ii}^{(N)}=0,\ 1\leq i\leq N$. Then observe that using the fact that $\sup_{N}\varepsilon_{N}<1$ and $|f|\leq 1$, we have $\Sigma_{ij}^{(N)}<C$ for some constant $C$ for $1\leq i,j\leq N$ and for all $N\geq 1$. Let $W_{N}$ be the empirical graphon corresponding to $\Sigma^{(N)}$. Let $$G^{N}=\bigg([N],\left\{ (i,j):1\leq i\leq j\right \}, \left(f\left(\frac{i}{N},\frac{j}{N}\right)\right)_{1\leq i\leq j\leq N}\bigg)$$ be the weighted graph on $[N]$ vertices with edge weights $\beta_{ij}=f\left(\frac{i}{N},\frac{j}{N}\right)$ for edge $i\neq j, 1\leq i,j\leq N$ and $\beta_{ii}=0, 1\leq i\leq N$. Let $\widehat{W}_{N}$ be the empirical graphon corresponding to $G^{N}$. Then we have
\begin{align}\label{eq: Inhomoorder1}
    W_{N}(x,y)=\widehat{W}_{N}(x,y)+o(1)
\end{align}
where $\mathrm{o}(1)$ is uniformly over all $(x,y)\in [0,1]^2$. Now observe that by definition $f$ is also a graphon and is uniformly continuous in $[0,1]^{2}$, hence given $\eta>0$ $\exists N_{\eta}\in\mathbb{N}$ such that $\forall N\geq N_{\eta}$,
\begin{align*}
    \left|f\left(\frac{i}{N},\frac{j}{N}\right)-f(x,y)\right|<\eta, \quad \forall(x,y)\in I_{i}\times I_{j}\quad\forall 1\leq i,j\leq N
\end{align*}
Then observe that 
\begin{align}
    \|\widehat{W}_{N}-f\|_{\Box}
    &=\sup_{S,T\subseteq[0,1]}\left|\int_{S\times T}\widehat{W}_{N}(x,y)-f(x,y)dxdy\right|\nonumber\\
    &\leq \int_{[0,1]^{2}}\left|\widehat{W}_{N}(x,y)-f(x,y)\right|dxdy\nonumber\\
    &=\sum_{i,j=1}^{N}\int_{I_{i}\times I_{j}}\left|\widehat{W}_{N}(x,y)-f(x,y)\right|dxdy\nonumber\\
    &=\sum_{i,j=1}^{N}\int_{I_{i}\times I_{j}}\left|f\left(\frac{i}{N},\frac{j}{N}\right)-f(x,y)\right|dxdy\leq \eta,\ \forall N\geq N_{\eta}\label{eq:inhomo_graphon_convg}
\end{align}
Hence we have $\delta_{\Box}\left(\widehat{W}_{N},f\right)\rightarrow 0$. Then using \eqref{eq: Inhomoorder1} we have $\delta_{\Box}\left(W_N,f\right)\rightarrow 0$. Thus assumptions \ref{itm:A1}-\ref{itm:A3} are satisfied.
Further when $f\not\equiv 0$ there exists an open set $U\subseteq[0,1]^{2}$ such that $f>0$ on $U$. The result follows  from Theorem \ref{thm:laplacian}. 
\end{proof}
Under a different situation, we are well equipped to look at the spectral norm of inhomogenous Erd\H{o}s-R\'enyi graph. Instead of the sparse setup of \eqref{eq:inhomo_epsilon},  we consider the dense situation
\begin{align}\label{eq:Inhomo_eps_norm}
    \lim_{N\rightarrow\infty}\varepsilon_{N}=\varepsilon_{\infty}
\end{align}
for some $\varepsilon_{\infty}\in (0,1)$. Further we assume that there exists a $\lambda>0$ such that
\begin{align}\label{eq:f_lb}
    \inf_{x,y\in [0,1]^2}f(x,y)>\lambda.
\end{align}
Once again we consider the matrix $B_{N}=((B_{ij}^{(N)}))_{i,j}$ defined in Proposition \ref{thm: Inhomothm}. Using the fact that $|B_{ij}^{(N)}|\leq \frac{2}{\sqrt{\vep_{N}}}$ and \eqref{eq:Inhomo_eps_norm} we must have
\begin{align*}
    \mathbb{E}\left[\left|B_{ij}^{(N)}\right|^{7}\right]\leq \frac{2^{7}}{\vep_{N}^{7/2}}\leq K
\end{align*}
for some $K>0$. Thus assumption \ref{itm:A2_Norm} is satisfied. Recall the variance profile matrix of $B_{N}$ given by
\begin{align*}
    \Sigma_{ij}^{(N)}=f\left(\frac{i}{N},\frac{j}{N}\right)-\varepsilon_{N}f\left(\frac{i}{N},\frac{j}{N}\right)^{2}
\end{align*}
Then using the fact that $f$ is a continuous function from a compact set, lower bounded by $\lambda>0$, we can easily see that $\Sigma_{ij}^{(N)}$ satisfies assumption \ref{itm:A1_Norm}. Further continuity along with compact support shows that the function $f^{2}$ is uniformly continuous. Using an argument similar to \eqref{eq:inhomo_graphon_convg}, it can be shown that the empirical graphon corresponding to $\Sigma_{ij}^{(N)}$ converges to
\begin{align*}
    W(x,y)=f(x,y)-\varepsilon_{\infty} f(x,y)^{2}
\end{align*}
in the cut metric. By definition, $W$ is a well defined graphon, and hence assumption \ref{itm:A3_Norm} is satisfied. Thus we have the following result using Theorem \ref{thm:spectral_norm_bound}.
\begin{prop}
Under the additional assumptions \eqref{eq:Inhomo_eps_norm} and \eqref{eq:f_lb}, we have
\begin{align*}
    \left\|\Delta_{N}\right\|=\Theta(\sqrt{N\log N})\ a.s.
\end{align*}
\end{prop}
where $h(n)=\Theta(g(n))$ implies that for all large enough $n$, $c_{1}g(n)\leq h(n)\leq c_{2}g(n)$ for some positive constants $c_{1}$ and $c_{2}$. 
\subsection{Sparse W-random graphs}
Given a graphon $W:[0,1]^{2}\rightarrow[0,1]$, as stated in the definitions of \citet{borgs2019} a sequence of sparse random graphs $G_{N}$ can be generated in the following way. We choose $\vep_{N}$ to be a sparsity parameter such that $\sup_{N}\vep_{N}<1$ and $\vep_{N}\rightarrow 0$ and $N\vep_{N}\rightarrow\infty$. Let $\{X_{i}: i\geq 1\}$ be i.i.d. $U[0,1]$ consider the random graph $G_{N}$, where $i$ and $j$ are connected with probability $\vep_{N}W(X_{i},X_{j})$ independently for all $i\neq j$. The graph $G_{N}$ is called the \textit{sparse W-random graph} and is denoted by $\mathcal{G}(N,W,\vep_{N})$. 
The adjacency matrix of such random graphs were studied in \citet{zhu2020graphon}. In the following theorem we show the existence of an unique limiting spectral distribution of the Laplacian matrix. The proof essentially follows the ideas from~\citet{zhu2020graphon} and we present the ideas for completeness.
\begin{prop}\label{thm:sparse_W}
Let $\mathcal{G}(N,W,\vep_{n})$ be a sequence of sparse W-random graphs with adjacency matrix $A_{N}$ and the corresponding Laplacian matrix $\Delta_{N}$. Further assume that $W$ is regular, that is
\begin{align*}
	\int_{0}^{1}W(x,y)dy=d_{W}, \ \forall x\in [0,1]
\end{align*}
where $d_{W}$ is a constant only depending on the graphon $W$. Then the empirical spectral distribution of $\frac{\Delta_{N}-\mathbb{E}\Delta_{N}}{\sqrt{N\vep_{N}}}$ converges in probability to an unique symmetric probability measure $\nu$. Further if there exists an open set $U\subseteq [0,1]^{2}$ such that $W>0$ on $U$, then $\nu$ have unbounded support.
\end{prop}

\begin{proof}
	The proof will follow from the proof of~\citet[Theorem 5.1]{zhu2020graphon} as long as we can show the following,
\begin{align}\label{eq:obj_wrand}
	L\left(\frac{\Delta_{N}-\mathbb{E}\Delta_{N}}{\sqrt{N\vep_{N}}},\frac{\Delta_{N}-\mathbb{E}\left[\Delta_{N}\middle|X_{1},\cdots,X_{N}\right]}{\sqrt{N\vep_{N}}}\right)\overset{p}{\rightarrow}0
\end{align}
Observe that,
\begin{align}\label{eq:reg_iden}
	\mathbb{E}W(X_{1},X_{2})=\int\int W(x,y)dxdy=d_{W}\text{ and }\mathbb{E}W(X_{1},X_{2})W(X_{1},X_{3})=d_{W}^2
\end{align}
By definition,
\begin{align*}
	\frac{1}{N}&\Tr\left(\frac{\mathbb{E}\left[\Delta_{N}\middle| X_{1},X_{2},\cdots,X_{N} \right]-\mathbb{E}\Delta_{N}}{\sqrt{N\vep_{N}}}\right)^2\\
	&= \frac{\vep_{N}}{N^{2}}\left[\sum_{i=1}^{N}\left(\sum_{j=1,j\neq i}^{N}(W(X_{i},X_{j})-d_{W})\right)^2+\sum_{i=1}^{N}\sum_{j=1,j\neq i}^{N}\left(W(X_{i},X_{j})-d_{W}\right)^{2}\right]\\
	&= \frac{\vep_{N}}{N^{2}}\sum_{i=1}^{N}\left(\sum_{j=1,j\neq i}^{N}W(X_{i},X_{j})-d_{W}\right)^2+O(\vep_{N})\\
	&= \frac{\vep_{N}}{N^{2}}\sum_{i=1}^{N}\sum_{j\neq k, j,k\neq i}\left(W(X_{i},X_{k})-d_{W}\right)\left(W(X_{i},X_{j})-d_{W}\right) + O(\vep_{N})
\end{align*}
Then by \eqref{eq:reg_iden},
\begin{align*}
	\mathbb{E}\frac{1}{N}&\Tr\left(\frac{\mathbb{E}\left[\Delta_{N}\middle| X_{1},X_{2},\cdots,X_{N} \right]-\mathbb{E}\Delta_{N}}{\sqrt{N\vep_{N}}}\right)^2\\
	&= \frac{\vep_{N}}{N^{2}}\mathbb{E}\left[\sum_{i=1}^{N}\sum_{j\neq k, j,k\neq i}\left(W(X_{i},X_{k})-d_{W}\right)\left(W(X_{i},X_{j})-d_{W}\right)\right] + O(\vep_{N})\\
	&=  \frac{\vep_{N}}{N^{2}}\sum_{i=1}^{N}\sum_{j\neq k, j,k\neq i}\left(\mathbb{E}W(X_{i},X_{k})W(X_{i},X_{j})-d_{W}^2\right) + O(\vep_{N})=O(\vep_{N})\\
\end{align*}
Finally an application of \cite[Corollary A.41]{bai2010spectral} proves \eqref{eq:obj_wrand}.
\end{proof}

\subsection{Constrained Random Graph}Constraints in the random graph are very important especially in the context of Gibbs measure related to the graph. The constraints impose certain dependence in the graph. Let $S_{N}$ be the set of all simple graphs on $N$ vertices. One of the important classes is the so-called \textit{canonical ensemble} which puts a probability distribution on the set of simple graphs in a way that entropy is maximized and the average degree is equal to a prescribed value. There has been a recent interest in the breaking of ensemble equivalence study where the canonical ensembles play a crucial role. We refer to \citet{touchette2015equivalence,squartini2015breaking} for further details. The spectrum of the adjacency matrix for constrained random graph was derived in \citet{chakrabarty2018spectra} and here we derive similar results for the Laplacian matrix.

The canonical ensemble measure $P_{N}$ is the unique probability distribution on $S_{N}$ with the following two properties:
\begin{enumerate}
    \item [(I)] The \textit{average degree} of vertex $i$, defined by $\sum_{G\in S_{N}}k_{i}(G)P_{N}(G)$, equals $k_{i}^{\star}$ for all $1\leq i\leq N$, where $k^{\star}=(k_{i}^{\star}): 1\leq i\leq N)$ is a fixed sequence of positive integers of which we only require to be graphical.
    \item [(II)] The \textit{entropy} of $P_{N}$, defined by $-\sum_{G\in S_{N}}P_{N}(G)\log P_{N}(G)$, is maximal.
\end{enumerate}
It is known that because of property (II), $P_{N}$ takes the form 
\begin{align*}
    P_{N}(G)=\frac{1}{Z_{N}(\theta)^{\star}}\exp\left(-\sum_{i=1}^{N}\theta_{i}^{\star}k_{i}(G)\right),\quad G\in S_{N},
\end{align*}
where $\theta^{\star}=\left(\theta_{i}^{\star}: 1\leq i\leq N\right)$  is a sequence of real-valued Lagrange multipliers that must be chosen in such a way that property (I) is satisfied. The normalisation constant $Z_{N}(\theta^{\star})$, which depends on $\theta^{\star}$, is called the partition functions in Gibbs theory. The matching of property (I) uniquely fixes $\theta^{\star}$, namely , it turns out that ~(\citet{squartini2015breaking})
\begin{align*}
    P_{N}(G)=\prod_{1\leq i<j\leq N}^{N}(p_{ij}^{\star})^{A_{N}[G](i,j)}(1-p_{ij}^{\star})^{1-A_{N}[G](i,j)},\quad G\in S_{N},
\end{align*}
where $A_{N}[G]$ is the adjacency matrix of $G$, and $p_{ij}^{\star}$ represent a \textit{reparametrisation} of the Lagrange multipliers, namely,
\begin{align}\label{eq:pijxij}
    p_{ij}^{\star}=\frac{x_{i}^{\star}x_{j}^{\star}}{1+x_{i}^{\star}x_{j}^{\star}}, \quad 1\leq i\neq j\leq N,
\end{align}
with $x_{i}^{\star}=e^{-\theta_{i}^{\star}}$. Thus, we see that $P_{N}$ is nothing other than an inhomogenous Erd\H{o}s-R\'enyi random graph where the probability that vertices $i$ and $j$ are connected by an edge equals $p_{ij}^{\star}$. In order to match property (I), these probabilities must satisfy
\begin{align}\label{eq:kipij}
    k_{i}^{\star}=\sum_{j\neq i, j=1}^{N}p_{ij}^{\star}, \quad 1\leq i\leq N,
\end{align}
which constitutes a set of $N$ equations for the $N$ unknowns $x_{1}^{\star}, \cdots x_{N}^{\star}$.

In order to state the next result, we need to make some assumptions on the sequence $(k_{Ni}^{\star}: 1\leq i\leq N)$. For the sake of simplification the dependence on $N$ would be suppressed from notation. 

\begin{prop}\label{prop:constrained}
Let $(k_{i}^{\star}:1\leq i\leq N)$ be a graphical sequence of positive integers. Define 
\begin{align*}
    m_{N}=\max_{1\leq i\leq N}k_{i}^{\star}
\end{align*}
Assume that
\begin{align}\label{eq:m_Nlimit}
    \lim_{N\rightarrow \infty}m_{N}=\infty,\quad \lim_{N\rightarrow\infty}m_{N}/\sqrt{N}=0,
\end{align}
and consider the graph $G^{k_{N}}=\bigg([N],\{(i,j):1\leq i< j\}, (k_{i}^{\star}k_{j}^{\star}/m_{N})_{1\leq i< j\leq N}\bigg)$ and the corresponding empirical graphon $W_{N}^{k}$. Further assume there exists a graphon $W\in\mathcal{W}_{0}$ such that
\begin{align*}\label{eq:constrained_graphon}
    \delta_{\Box}(W_{N}^{k},W)\rightarrow 0.
\end{align*}
Let $x_{i}^{\star}$ and $p_{ij}^{\star}$ be determined by~\eqref{eq:pijxij} and~\eqref{eq:kipij}. Let $\Delta_{N}$ be the Laplacian matrix of an inhomogenous Erd\H{o}s--R\'enyi random graph on $N$ vertices, with $p_{ij}^{\star}$ the probability of an edge being present between vertices $i$ and $j$ for $1\leq i\neq j\leq N$. Then there exists an unique, symmetric probability measure $\nu_{k}$ on $\mathbb{R}$ such that
\begin{align}
    \lim_{N\rightarrow\infty} \text{ESD}\left(\frac{1}{\sqrt{N\varepsilon_{N}}}(\Delta_{N}-\mathbb{E}(\Delta_{N}))\right)=\nu_{k},\text{ weakly in probability}
\end{align}
where $\varepsilon_{N}=\frac{m_{N}^{2}}{\sum_{1\leq l\leq N}k_{l}^{\star}}$. Further if $W$ is positive on an open set in $[0,1]$, then support of $\nu_{k}$ is unbounded.
\end{prop}

\begin{proof}
Consider $A_{N}$ to be adjacency matrix of the inhomogenous Erd\H{o}s--R\'enyi graph defined in the proposition. Define 
\begin{align*}
    \sigma_{N}=\sum_{1\leq l\leq N}k_{l}^{\star},
\end{align*}
It is known that (\citet{squartini2015breaking})
\begin{align*}
    \max_{1\leq l\leq N}x_{l}^{\star}=o(1),
\end{align*}
in which case \eqref{eq:pijxij} and \eqref{eq:kipij} gives
\begin{align}
    x_{i}^{\star}=[1+o(1)]\frac{k_{i}^{\star}}{\sqrt{\sigma_{N}}},\quad p_{ij}^{\star}=[1+o(1)]\frac{k_{i}^{\star}k_{j}^{\star}}{\sigma_{N}}, \text{ as }N\rightarrow\infty
\end{align}
with the error term \textit{uniform} in $1\leq  i\neq j\leq N$. Then by definition
\begin{align}
    \vep_{N}=\frac{m_{N}^{2}}{\sigma_{N}}
\end{align}
It follows from \eqref{eq:m_Nlimit} that
\begin{align*}
    \lim_{N\rightarrow\infty}\vep_{N}=0,\quad\lim_{N\rightarrow\infty}N\vep_{N}=\infty
\end{align*}
Define $B_{N}=\frac{A_{N}}{\sqrt{\vep_{N}}}$, it is easy to observe that
\begin{align*}
    \left|B_{N}(i,j)-\mathbb{E}B_{N}(i,j)\right|\leq \frac{2}{\sqrt{\vep_{N}}}
\end{align*}
The bound above immediately implies $B_{N}$ satisfies the Lindeberg condition.
Remember that by definition the empirical graphon $W_{N}^{k}$ is such that
\begin{align*}
    W_{N}(x,y)=
    \begin{cases}
        \frac{k_{i}^{\star}k_{j}^{\star}}{m_{N}^{2}} & \text{if }(x,y)\in I_{i}\times I_{j},\ 1\leq i\neq j\leq N\\
        0 & \text{if }(x,y)\in I_{i}\times I_{i},\ 1\leq i\leq N
    \end{cases}
\end{align*}
For $1\leq i\neq j\leq N$,
\begin{align*}
    \text{Var}(B_{N}(i,j))
    &=\frac{1}{\vep_{N}}p_{ij}^{\star}(1-p_{ij}^{\star})\\
    &=\frac{k_{i}^{\star}k_{j}^{\star}}{m_{N}^{2}}\left(1-\vep_{N}\frac{k_{i}^{\star}k_{j}^{\star}}{m_{N}^{2}}\right)+o(1).
\end{align*}
Define an empirical graphon $\widetilde{W}_{N}^{k}$ as 
\begin{align*}
    \widetilde{W}_{N}^{k}(x,y)=
    \begin{cases}
        \text{Var}(B_{N}(i,j)) & \text{if }(x,y)\in I_{i}\times I_{j},\ 1\leq i\neq j\leq N\\
        0 & \text{if }(x,y)\in I_{i}\times I_{i},\ 1\leq i\leq N
    \end{cases}
\end{align*}
Then by definition, for all $1\leq i\neq j\leq N$,
\begin{align}\label{eq:WtildeWNconvg}
    \left|W_{N}^{k}(x,y)-\widetilde{W}_{N}^{k}(x,y)\right|\leq\left|\vep_{N}\left(\frac{k_{i}^{\star}k_{j}^{\star}}{m_{N}^{2}}\right)^{2}\right|+o(1)\overset{N\rightarrow\infty}{\longrightarrow}0,\quad \forall (x,y)\in I_{i}\times I_{j}
\end{align}
DCT combined with \eqref{eq:WtildeWNconvg} gives
\begin{align*}
    \delta_{\Box}(W_{N}^{k},\widetilde{W}_{N}^{k})\leq \int_{[0,1]^{2}}\left|W_{N}^{k}(x,y)-\widetilde{W}_{N}^{k}(x,y)\right|dxdy\rightarrow 0.
\end{align*}
Thus we have $\delta_{\Box}(\widetilde{W}_{N}^{k},W)\rightarrow 0$. Hence by an appeal to Theorem \ref{thm:laplacian} the proof is completed.
\end{proof}

\begin{remark}
A concrete example of a graphical sequence $(k_{i}^{\star}: 1\leq i\leq N)$ satisfying \eqref{eq:m_Nlimit}-\eqref{eq:constrained_graphon} can be constructed from the example considered in \citet[Remark 5.1]{chakrabarty2018spectra}. For $N\geq 1$, let
\begin{align*}
    k_{i}^{\star}=\lfloor i^{1/3}\rfloor,\quad 1\leq i\leq N
\end{align*}
Then Theorem 7.12 from \citet{van2009random} implies that $(k_{i}^{\star}:1\leq i\leq N)$ is graphical for $N$ large enough. Since $m_{N}=\lfloor N^{1/3}\rfloor$ it is immediate that \eqref{eq:m_Nlimit} holds. Define 
\begin{align*}
    W(x,y)=(xy)^{1/3},\quad (x,y)\in [0,1]^{2}
\end{align*}
Then using the $L^{1}$ bound on $\delta_{\Box}(W_{N}^{k},W)$, where $W_{N}^{k}$ is as defined in Proposition \ref{prop:constrained}, we can show that 
\begin{align*}
    \delta_{\Box}(W_{N}^{k},W)\rightarrow 0
\end{align*}
\end{remark}

Among other examples one can consider also the random block matrices and sparse stochastic block models considered in \citet{zhu2020graphon}. The results for the Laplacian matrices hold under the assumptions stated in Section 6 and Section 7 of \citet{zhu2020graphon}. To avoid repetitions we skip the results.

\section{Simulations}\label{sec:simulations}
This section is devoted to simulation study for a clearer picture of the above results. We consider two situations, firstly we consider the inhomogenous Erd\H{o}s-R\'enyi graph and in the second case, we generate the elements of adjacency matrices independently having a Gaussian distribution while respecting the symmetry constant.

For the inhomogenous Erd\H{o}s-R\'enyi graph we choose $f(x,y) = \sqrt{xy}$ i.e.~$f$ is a product of two same functions $r(t) = \sqrt{t}$ in $L^2[0,1]$. Then, in the Figure~\ref{fig: ErdosHist}(A) the eigenvalues of the centered Laplacian matrix are plotted under the scaling $\sqrt{N\vep_N}$ and hence it is clear that the ESD converges to a symmetric distribution. In the simulation
$N$ and $\vep_N$ are chosen to be $1000$ and $0.25$ respectively. In Figure~\ref{fig: ErdosHist} we show a comparison with the limiting spectral distribution under usual Erd\H{o}s-R\'enyi graph with edge retention probability $0.25$.\\
\begin{figure}[!ht]
    \centering
    \subfloat[LSD of $\frac{1}{\sqrt{N\vep_{N}}}\Delta_{N}^{0}$ for Inhomogenous Erd\H{o}s-R\'enyi graph ]{{\includegraphics[width=7cm]{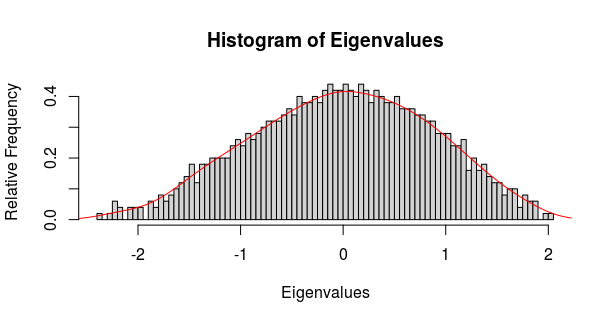} }}%
    \qquad
    \subfloat[LSD of $\frac{1}{\sqrt{N\vep_{N}}}\Delta_{N}^{0}$ for Erd\H{o}s-R\'enyi graph with $p=0.25$ ]{{\includegraphics[width=7cm]{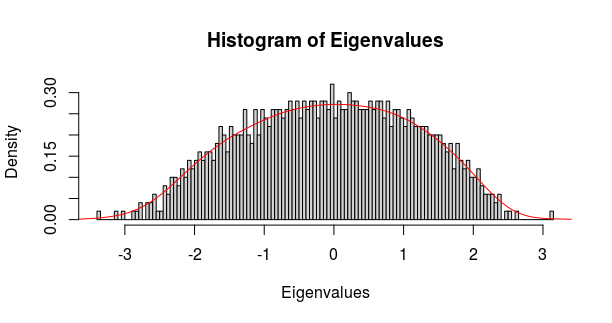} }}%
    \caption{ESD in both Inhomogenous and Homogenous Erd\H{o}s-R\'enyi graph}%
    \label{fig: ErdosHist}%
\end{figure}
For the Gaussian case, we consider two graphons $W$ as $W(x,y)=\sqrt{xy}$ and $W(x,y)=\frac{1}{2}(x(1-y)+y(1-x))$. Then the elements of the adjacency matrix are generated as $X_{ij}^{(n)}\sim\mathrm{N}\left(0,W\left(\frac{i}{n},\frac{j}{n}\right)\right)$. The graphons considerd here are uniformly continuous and hence by \eqref{eq:inhomo_graphon_convg}the empirical graphon constructed using above variance profile converges in the cut norm to $W$. We consider $n=1000$ so that the adjacency matrix is a $1000\times 1000$ matrix. In Figure~\ref{fig: Gaussian1} the eigenvalues of the Laplacian matrix under the scaling $\sqrt{n}$ and hence we can observe that ESD converges to a symmetric distribution.
\begin{figure}[!ht]
    \centering
    \subfloat[$W(x,y)=\sqrt{xy}$]{{\includegraphics[width=7cm]{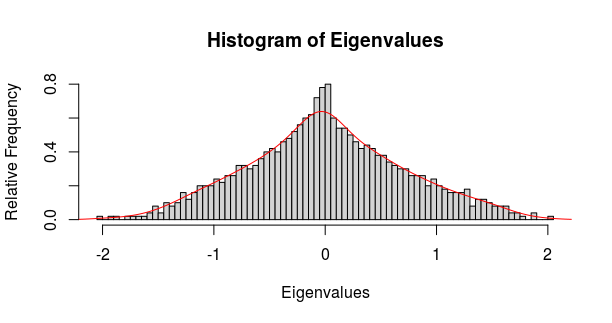} }}%
    \qquad
    \subfloat[$W(x,y)=\frac{1}{2}(x(1-y)+y(1-x))$]{{\includegraphics[width=7cm]{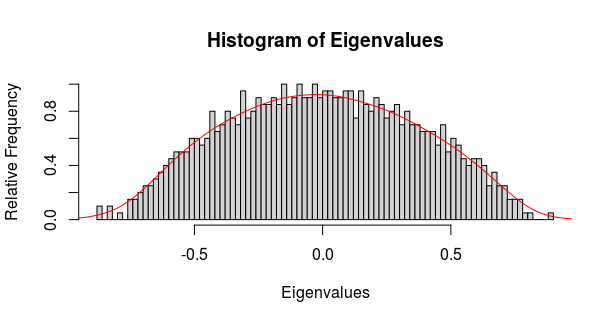} }}%
    \caption{ESD of $\frac{1}{\sqrt{n}}\Delta_{n}$ when edge weights are Gaussian}%
    \label{fig: Gaussian1}%
\end{figure}

\section{Proof of Theorems \ref{thm:laplacian}, \ref{thm:multi_structure} and Corollary \ref{corollary: dingjiang}}\label{sec:proofESD}
\subsection{Preparatory Lemmas for Theorem \ref{thm:laplacian}}
The proof of Theorem \ref{thm:laplacian} rely on several preparatory lemmas which are organised in this section. One of the crucial steps in studying the properties of ESD is to replace each entry by a Gaussian random variable, which we outline in the following lemma. 
 Let $\{G_{i,j},1\leq i\leq j\}$ be a family of i.i.d. standard Gaussian random variables. Define $N\times N$ matrices $A_{N}^{g}$ and $\Delta_{N}^{g}$ by 
\begin{align}
    &
    A_{N}^{g}(i,j)=\frac{\sigma_{i\wedge j,i\vee j}}{\sqrt{N}}G_{i\wedge j,i\vee j} \quad \forall 1\leq i,j\leq N \label{eq:ANgandDelNg}\\
    &
    \Delta_{N}^{g}(i,j)=
    \begin{cases}
        A_{N}^{g}(i,j) & \text{ if } i\neq j\\
        -\sum_{k\neq i,k=1}^{N} A_{N}^{g}(i,k) & \text{ if } i=j
    \end{cases}\label{eq:ANgandDelNg2}
\end{align}
Consider a three times continuously differentiable function $h:\mathbb{R}\rightarrow\mathbb{R}$ such that 
\begin{align*}
    \max_{0\leq j\leq 3}\sup_{x\in \mathbb{R}}\left|h^{(j)}(x)\right|<\infty
\end{align*}
where $h^{(j)}$ denotes the $j$-th derivative. For a $N\times N$ real symmetric matrix $M$ define the Stieltjes transform of the ESD of $M$ as
\begin{align*}
    H_{N}(M)=\frac{1}{N}\Tr\left(\left(M-zI_{N}\right)^{-1}\right), \,\, z\in \mathbb C\setminus \mathbb R.
\end{align*}
The next result shows that real and imaginary part of Steitjes transform of $A_N^0$ and $\Delta_N^0$ are close to the Gaussian counterparts. Since one knows that convergence of ESD is equivalent to showing the convergence of the corresponding Stieltjes transform, one can work with Gaussian random variables.  
\begin{lemma}\label{lemma:Gaussianisation}
(\textbf{Gaussianisation}) Let entries of $A_N^0$ satisfy the assumptions $\ref{itm:A1}$ and $\ref{itm:A2}$ and $A_N^g $ and $\Delta_N^g$ be defined as in \ref{eq:ANgandDelNg} and \ref{eq:ANgandDelNg2} then
\begin{align}
    & \lim_{N\rightarrow\infty} \mathbb{E}\left[h\left(\mathcal{R}H_{N}\left(\Delta_{N}^{g}\right)\right)-h\left(\mathcal{R}H_{N}\left(\Delta_{N}^{0}\right)\right)\right]=0\label{eq:Gaussianisation 3}\\
    & \lim_{N\rightarrow\infty} \mathbb{E}\left[h\left(\mathcal{I}H_{N}\left(\Delta_{N}^{g}\right)\right)-h\left(\mathcal{I}H_{N}\left(\Delta_{N}^{0}\right)\right)\right]=0\label{eq:Gaussianisation 4}
\end{align}
where $\mathcal{R}$ and $\mathcal{I}$ denotes the real and imaginary parts respectively. Similar statement holds true for $A_N^{0}$ and $A_N^g$.
\end{lemma}
The proof of Lemma \ref{lemma:Gaussianisation} is routine and hence is skipped here and presented in the Appendix. The next lemma allows for minor tweaks in the diagonal entries of  $\Delta_{N}^{g}$.
\begin{lemma}\label{lemma:DeltaNbar}
Define a $N\times N$ matrix by
\begin{equation}\label{def:barA}
    \Bar{A}_{N}(i,j)=\frac{\sigma_{i,j}}{\sqrt{N}}G_{i\wedge j,i\vee j},\quad 1\leq i,j\leq N
\end{equation}
and let 
\begin{equation}\label{def:bardelta}
\Bar{\Delta}_{N}=\Bar{A}_{N}-X_{N}
\end{equation}
where $X_{N}$ is a diagonal matrix of order $N$ defined by 
$$
    X_{N}(i,i)=\sum_{k\neq i}\Bar{A}_{N}(i,k),\quad 1\leq i\leq N.
$$
Then
\begin{align}
    &\lim_{N\rightarrow\infty} L\left(ESD\left(\Delta_{N}^{g}\right),ESD\left(\Bar{\Delta}_{N}\right)\right)=0 \text{ in probability.}
\end{align}
\end{lemma}
\begin{proof}
Observe that by \citet[Corollary A.41]{bai2010spectral}
\begin{equation}\label{eq:def:bardelta}
    \mathbb{E}\left[L^{3}\left(ESD\left(\Delta_{N}^{g}\right),ESD\left(\Bar{\Delta}_{N}\right)\right)\right]\leq \frac{1}{N}\mathbb{E}\left[\Tr\left[\Delta_{N}^{g}-\Bar{\Delta}_{N}\right]^{2}\right]
\end{equation}
Since $\Bar{\Delta}_{N}$ and $\Delta_{N}^{g}$ differs only in the diagonal entries, then we have $\left(\Bar{\Delta}_{N}-\Delta_{N}^{g}\right)_{i,i}=A_{N}^{g}(i,i)$, implying
\begin{align*}
    \frac{1}{N}\mathbb{E}\left[\Tr\left[\Delta_{N}^{g}-\Bar{\Delta}_{N}\right]^{2}\right]=\frac{1}{N}\mathbb{E}\left[\sum_{i=1}^{N}A_{N}^{g}(i,i)^{2}\right]=\frac{1}{N^{2}}\sum_{i=1}^{N}\sigma_{i,i}^{2}=O\left(\frac{1}{N}\right)\rightarrow 0
\end{align*}
The last order coming from \ref{itm:A1}. Hence $\mathbb{E}\left[L^{3}\left(ESD\left(\Delta_{N}^{g}\right),ESD\left(\Bar{\Delta}_{N}\right)\right)\right]\rightarrow 0$ which show \eqref{eq:def:bardelta}.
\end{proof}
The (diagonal) entries of $X_{N}$ are nothing but the row sums of $A_{N}^{g}$. However the correlation between an entry of $A_{N}^{g}$ and that of $X_{N}$ is small. The following decoupling lemma, shows that it does not hurt when the entries of $X_{N}$ are replaced by a mean-zero Gaussian random variable of the same variance that is independent of $A_{N}^{g}$. 
\begin{lemma}\citet[Lemma 2.4]{chakrabarty2018spectra}\label{lemma: AN+YN}
Let $(Z_{i}: i\geq 1)$ be a family of i.i.d. standard normal random variables, independent of $\left(G_{i,j}:1\leq i\leq j\right)$. Define a diagonal matrix $Y_{N}$ of order $N$ by
$$
    Y_{N}(i,i)=Z_{i}\sqrt{\frac{1}{N}\sum_{\substack{j=1\\j\neq i}}^{N}\sigma_{i,j}^{2}},\quad 1\leq i\leq N
$$
and let 
\begin{equation}\label{def:tildedelta}
    \widetilde{\Delta}_{N}=\Bar{A}_{N}+Y_{N}
\end{equation}
Then for every $k\in\mathbb{N}$
\begin{align}
    \lim_{N\rightarrow\infty}\frac{1}{N}\mathbb{E}\left(\Tr\left[(\widetilde{\Delta}_{N})^{2k}-(\Bar{\Delta}_{N})^{2k}\right]\right)=0
\end{align}
and
\begin{align}
    \lim_{N\rightarrow\infty}\frac{1}{N^{2}}\mathbb{E}\left(\Tr^{2}\left[(\widetilde{\Delta}_{N})^{k}\right]-\Tr^{2}\left[(\Bar{\Delta}_{N})^{2k}\right]\right)=0
\end{align}
\end{lemma}
We skip the proof of the above lemma since it is verbatim same as Lemma 2.4 of \citet{chakrabarty2018spectra}. In the next lemma we show the convergence of ESD of the above defined diagonal matrix $Y_{N}$. 
\begin{lemma}\label{thm:ESD_YN}
    Under the assumptions \ref{itm:A1}-\ref{itm:A3}, there exists a unique probability distribution $\zeta$ on $\mathbb{R}$  such that 
    \begin{align*}
        \lim_{N\rightarrow\infty}\text{ESD}\left(Y_{N}\right)=\zeta \text{ weakly in probability}
    \end{align*}
    Further if there exists an open set $U\subseteq [0,1]^{2}$ such that $W>0$ on $U$, then $\zeta$ have unbounded support.
\end{lemma}

\begin{proof}
We would be using method of moments to prove our result. After showing moment convergence we would show that the limits uniquely determine the distribution. And finally we would show the unbounded support.\\
\\
\textbf{Part 1: Convergence of Moments}\\
\\
Define $m_{k}^N$ to be the $k^{th}$ moment of ESD of $Y_{N}$,
\begin{align*}
    m_{k}^N=\frac{1}{N}\Tr(Y_{N}^{k})=\frac{1}{N}\sum_{i=1}^{N}\left(\sum_{j\neq i, j=1}^N\frac{\sigma_{i,j}^{2}}{N}\right)^{\frac{k}{2}}Z_{i}^{k}
\end{align*}
We first show that variance of $m_k^N$ goes to zero as $N\to \infty$. Note that using $\{Z_i\}$ is a collection of iid standard Gaussian random variables we have
\begin{align*}
    \Var\left(m_{k}^N\right)
    &=\Var\left(Z_{1}^{k}\right)\frac{1}{N^{2}}\sum_{i=1}^{N}\left(\sum_{j\neq i, j=1}^N\frac{\sigma_{i,j}^{2}}{N}\right)^{k}.
\end{align*}
Now define
\begin{align*}
    S_{1}^{k}:=\frac{1}{N}\sum_{i=1}^{N}\left(\sum_{j\neq i}\frac{\sigma_{i,j}^{2}}{N}\right)^{k}
    &=\frac{1}{N^{k+1}}\sum_{i=1}^{N}\sum_{j_{1}\hdots j_{k}\neq i}\sigma_{i,j_{1}}^{2}\hdots \sigma_{i,j_{k}}^{2}
\end{align*}
and
\begin{align}
    \widetilde{S}_{1}^{k}:=\frac{1}{N^{k+1}}\sum_{i=1}^{N}\sum_{j_{1}\hdots j_{k}}\sigma_{i,j_{1}}^{2}\hdots \sigma_{i,j_{k}}^{2}.
\end{align}
It immediately follows that
\begin{align}
    \left|S_{1}^{k}-\widetilde{S}_{1}^{k}\right|\leq \frac{1}{N^{k+1}}O(N^{k})=O\left(\frac{1}{N}\right)\overset{N\rightarrow\infty}{\longrightarrow}0
\end{align}
Consider the star graph $F_{k}=\left(V_{k},E_{k}\right)$ with $k+1$ vertices and $k$ edges, with the internal node labelled as $i$ and the leaves labelled as $j_{1},\cdots,j_{k}$. An example of such a star graph is shown in Figure \ref{fig:Star_4}.
\begin{figure}
\centering
\begin{tikzpicture}
  [scale=.8,auto=left,every node/.style={circle,fill=blue!40}]
  \node (n6) at (0,0) {i};
  \node (n5) at (2,2)  {$j_{2}$};
  \node (n4) at (-2,2) {$j_{1}$};
  \node (n3) at (-2,-2)  {$j_{4}$};
  \node (n2) at (2,-2)  {$j_{3}$};

  \foreach \from/\to in {n6/n5,n6/n4,n6/n3,n6/n2}
    \draw (\from) -- (\to);
\end{tikzpicture}
\caption{Star graph $F_{4}$}\label{fig:Star_4}
\end{figure}
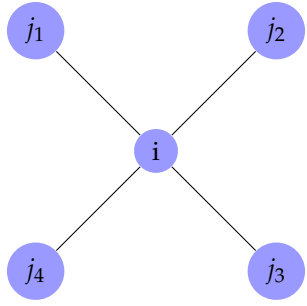
Then we have (recall the definition of homomorphism density~\eqref{def:graphhom})
\begin{align}
    t\left(F_{k},W_{N}\right)
    &=\int_{[0,1]^{k+1}}\prod_{(u,v)\in E_{k}}W_{N}(x_{u},x_{v})\prod_{u\in V_{k}}dx_{u}\nonumber\\
    &=\sum_{i,j_{1}\hdots j_{k}=1}^{N}\int_{I_{i}\otimes_{l=1}^{k}I_{j_{l}}}\sigma_{i,j_{1}}^{2}\hdots\sigma_{i,j_{2}}^{2}dx_{i}dx_{j_{1}}\hdots dx_{j_{k}}\nonumber\\
    &=\frac{1}{N^{k+1}}\sum_{i=1}^{N}\sum_{j_{1}\hdots j_{k}}\sigma_{i,j_{1}}^{2}\hdots \sigma_{i,j_{k}}^{2}=: \widetilde{S}_{1}^{k}\nonumber
   \end{align}
Using \ref{itm:A3} we have $t(F_{k},W_{N})\rightarrow t(F_{k},W)<\infty$, implying
\begin{align}\label{eq:YN_S_convg}
    S_{1}^{k}\overset{N\rightarrow\infty}{\longrightarrow}t(F_{k},W).
\end{align}
Hence,
\begin{align}\label{eq:L2_convg_YN}
    \Var(m_{k}^N)=\frac{\Var\left(Z_{1}^{k}\right)}{N}S_{1}^{k}\overset{N\rightarrow\infty}{\longrightarrow}0.
\end{align}
Now we show that $\mathbb{E}(m_{k}^N)$ converges. Observe that
\begin{align*}
    \mathbb{E}(m_{k}^N)=
    \begin{cases}
        0 & \text{ if $k$ is odd},\\
        \mu_{k}\frac{1}{N}\sum_{i=1}^{N}\left(\sum_{j\neq i}\frac{\sigma_{i,j}^{2}}{N}\right)^{\frac{k}{2}} & \text{ if $k$ is even},
    \end{cases}
\end{align*}
where $\mu_{k}$ is the $k^{th}$ moment of standard Gaussian distribution 
If $k$ is even, then we have $\mathbb{E}(m_{k})=\mu_{k}S_{1}^{\frac{k}{2}}$. So combining \eqref{eq:YN_S_convg} and \eqref{eq:L2_convg_YN}, we infer
\begin{align}
    m_{k}^N\overset{L_{2}}{\longrightarrow}
    \begin{cases}
        0 & \text{ if $k$ is odd}\\
        \mu_{k}t\left(F_{\frac{k}{2}},\, W\right) & \text{ if $k$ is even}
    \end{cases}
\end{align}
\\
\textbf{Part 2: Uniqueness of the Limiting distribution}\\
\\
Before proceeding let us define
\begin{align}
    \eta_{k}=
    \begin{cases}
        0 & \text{ if $k$ is odd}\\
        \mu_{k}t\left(F_{\frac{k}{2}},W\right) & \text{ if $k$ is even}
    \end{cases}
\end{align}
then by Lemma B.1 from \citet{bai2010spectral} we can easily see that there exists a probability measure $\zeta$ identified by the moments $\eta_{k}$. To deal with the uniqueness of $\zeta$ observe that due to \ref{itm:A1}, for all $k$ even we can find $C>0$ such that $t(F_{k/2},W)\leq C^{k/2}$. Then for $k\in \mathbb{N}$
\begin{align}
    \frac{1}{2k}\eta_{2k}^{\frac{1}{2k}}\preceq \frac{1}{2k}\left(\frac{2k!}{k!}\right)^{\frac{1}{2k}}.
\end{align}
Using the Sterling's approximation
it is immediate that $\limsup_{k\rightarrow\infty}\frac{1}{2k}\eta_{2k}^{\frac{1}{2k}}<\infty$.
By applying Theorem 1 of \citet{lin2017recent}, the probability measure $\zeta$ is uniquely identified by the moment sequence $\eta_{k}$.
\\
\textbf{Part 3: Unbounded support of $\mu$}\\
\\
Suppose there exists an open $U\subseteq [0,1]^{2}$ such that $W>0$ on it. For $k\in \mathbb{N}$ define 
\begin{align*}
    \beta_{k}=
    \begin{cases}
        0 & \text{ if $k$ is odd}\\
        t\left(F_{\frac{k}{2}},W\right) & \text{ if $k$ is even}
    \end{cases}
\end{align*}
Since $W>0$ on $U$, then by definition $t\left(F_{\frac{k}{2}},W\right)>0$. Then again using Lemma B.1 from \citet{bai2010spectral} there exists a probability measure $\kappa$ having the above moment sequence. By Theorem 1 of \citet{lin2017recent}, we can say that the moment sequence $\beta_{k}$ uniquely identifies the probability measure $\kappa$. Now consider $X\sim\kappa$ and consider $Z\sim N(0,1)$ independently of $X$. Then it is easy to observe that $XZ$ have the moment sequence $\eta_{k}$ and hence $XZ\sim\zeta$.
It easily follows from here that if we take any $M>0$ then $P(XZ>M)>0$ and
\[
    \sup(\text{supp}(XZ))=\sup(\text{supp}(\zeta))=\infty
  \]  
This completes the proof of Lemma~\ref{thm:ESD_YN}.
\end{proof}

{\bf Notations:} We now recall some notations from~\citet{chakrabarty2018spectra} and we refer the reader to \citet{nica2006lectures} for the combinatorial properties of non-crossing pair partition and Krewaras complement.

For $k\in\mathbb{N}$ and $\Pi$ a partition of $\{1,\cdots, 2k\}$, define
$$
    \Psi\left(\Pi,N\right)=\left\{i\in\{1,\cdots,N\}^{2k}:i_{u}=i_{v}\iff u,v\text{ belong to the same block of }\Pi\right\}
$$

For an even positive integer $k$, $NC_{2}(k)$ is the set of non-crossing pair partitions of $\{1,2,\hdots,k\}$. For $\sigma\in NC_{2}(k)$, its Kreweras complement $K(\sigma)$ is the maximal non-crossing partition $\Bar{\sigma}$ of $\{\Bar{1},\hdots,\Bar{k}\}$, such that $\sigma\cup\Bar{\sigma}$ is a non-crossing partition of $\{1,\Bar{1},\hdots, k,\Bar{k}\}$. 
For $\sigma\in NC_{2}(k)$ and $N\geq 1$, define
\begin{align}
    S(\sigma,N)=\{i\in\{1,\hdots,N\}^{k}: i_{u}=i_{v} \iff \text{ $u,v$ belong to the same block of $K(\sigma)$}\}
\end{align}
and 
\begin{align}
    C(k,N)=\{1,\hdots N\}^{k}\setminus \bigg(\bigcup_{\sigma\in NC_{2}(k)}S(\sigma,N)\bigg)
\end{align}
In other words, $S(\sigma,N)$ is the same as $\Psi(K(\sigma),N)$. We will use this fact in the upcoming proof.

\subsection{Proof of Theorem \ref{thm:laplacian}}
Combining Lemmas \ref{lemma:Gaussianisation} and \ref{lemma:DeltaNbar} we find that ESD$(\Delta_{N}^{0})$  and ESD$(\Bar{\Delta}_{N})$ have the same in probability limit. Next we use method of moments on ESD$(\Bar{\Delta}_{N})$. By Lemma \ref{lemma: AN+YN} it is enough to look at the moments of ESD$(\widetilde{\Delta}_{N})$ where $\widetilde{\Delta}_N$ is as defined in \eqref{def:tildedelta}. The rest of the proof is organised as follows. First we show the $L_{2}$ convergence of moments of ESD$(\widetilde{\Delta}_{N})$. We work separately for even and odd moment. Convergence of the even moment  is more involved, and the odd moment convergence follows along similar lines. In case of even moments we appeal to the fact that if $X$ is a random variable and $\mathbb{E}X$ converges to $\alpha$ for some $\alpha\in\mathbb{R}$, and $\mathbb{E}X^{2}$ converges to $\alpha^{2}$, then $X$ converges to $\alpha$ is $L_{2}$. We will follow the combinatorial ideas from \citet{zhu2020graphon} and express the expected value of moments in terms of graph homomorphism and use graphon convergence assumption (\ref{itm:A3}) to find the limit. Finally an appeal to Theorem 1 from \citet{lin2017recent} would show the uniqueness of the distribution.

\subsubsection{\textbf{Convergence of Moments}}
Fix $k\in\mathbb{N}\cup\{0\}$. We are interested to look at the $L_{2}$ convergence of $k^{th}$ moment of ESD($\Tilde{\Delta}_{N}$) given by $$\frac{1}{N}\Tr[(\Tilde{\Delta}_{N})^{k}].$$ We would deal with the odd and even moments separately. We shall show that the odd moments converge to $0$ and the even moments converge to $\sum_{\mathcal{P}_{k}}\sum_{\sigma\in NC_{2}(\sum m_{p})}\beta(\sigma)\mathcal{E}(\sigma)$
where $\mathcal P_k$ is defined in Section~\ref{sec:moment_description}, $\beta(\sigma)\geq 0$ and $\mathcal{E}(\sigma)$ will be defined in \eqref{eq:graphon_convg_tree_sigma_1} and \eqref{eq:E_sigma} respectively.

\textbf{Case 1: $k$ is even}

By definition we have
\begin{align}\label{eq:trace_expansion}
    \frac{1}{N}\Tr\left(\widetilde{\Delta}_{N}^{k}\right)=\frac{1}{N}\sum_{\substack{m_{1},m_{2}\hdots,m_{k}\\n_{1},n_{2}\hdots,n_{k}}}\Tr\left(\Bar{A}_{N}^{m_{1}}Y_{N}^{n_{1}}\hdots\Bar{A}_{N}^{m_{k}}Y_{N}^{n_{k}}\right)
\end{align}
where the sum is over all the terms in the expansion of $\left(\Bar{A}_{N}+Y_{N}\right)^{k}$, i.e. we have $2^{k}$ many terms and for every choice of $m_{1},n_{1},\hdots,m_{k},n_{k}$ we have $\sum_{i=1}^{k} m_{i}+n_{i}=k$. We can take $k$ many expressions in each term of the expansion, since we allow the exponents to be 0. So enough to look at $L_{2}$ convergence of $\frac{1}{N}\Tr\left(\Bar{A}_{N}^{m_{1}}Y_{N}^{n_{1}}\hdots\Bar{A}_{N}^{m_{k}}Y_{N}^{n_{k}}\right)$. Let
$$M_j= \sum_{p=1}^j m_j+n_j,\, \, j=1,\cdots, k.$$
Observe that $M_{k}=k$. Take $\widetilde{i}$ such that $\widetilde{i}=(\widetilde{i}_{1},\cdots,\widetilde{i}_{M_{k}+1})\in\{1,\cdots N\}^{M_{k}+1}$ and $\widetilde{i}_{1+M_k}=\widetilde{i}_{1}$. Then
\begin{align}\label{eq:trace_exp_i1_im}
    &\frac{1}{N}\Tr\left(\Bar{A}_{N}^{m_{1}}Y_{N}^{n_{1}}\hdots\Bar{A}_{N}^{m_{k}}Y_{N}^{n_{k}}\right)\nonumber
    \\
    =&\frac{1}{N}\sum_{{\widetilde{i}}}\prod_{j=1}^{m_{1}}\Bar{A}_{N}(i_{j},i_{j+1})\prod_{j=m_{1}+1}^{M_1}Y_{N}(i_{j},i_{j+1})\hdots
    \prod_{j=1+M_{k-1}}^{m_{k}+M_{k-1}}\Bar{A}_{N}(i_{j},i_{j+1})\prod_{j=1+m_{k}+M_{k-1}}^{M_k}Y_{N}(i_{j},i_{j+1})\nonumber\\
    =&\frac{1}{N}\sum_{i}\prod_{j=1}^{m_{1}}\Bar{A}_{N}(i_{j},i_{j+1})Y_{N}^{n_{1}}(i_{m_{1}+1},i_{m_{1}+1})\hdots\prod_{j=1+\sum_{p=1}^{k-1}m_{p}}^{\sum_{p=1}^{k}m_{p}}\Bar{A}_{N}(i_{j},i_{j+1})Y_{N}^{n_{k}}(i_{1},i_{1})\nonumber\\
    =&\frac{1}{N}\sum_{i_{1},i_{2},\hdots i_{\sum m_{p}}=1}^{N}\prod_{j=1}^{\sum m_{p}}\Bar{A}_{N}(i_{j},i_{j+1})\prod_{j=1}^{k}Y_{N}^{n_{j}}\left(i_{1+\sum_{p=1}^{j}m_{p}},i_{1+\sum_{p=1}^{j}m_{p}}\right)\nonumber\\
    =&\frac{1}{N}\sum_{i_{1}\hdots i_{\sum m_{p}}}\prod_{j=1}^{\sum m_{p}}G_{i_{j}\wedge i_{j+1},i_{j}\vee i_{j+1}}\prod_{j=1}^{\sum m_{p}}\frac{\sigma_{i_{j},i_{j+1}}}{\sqrt{N}}\prod_{j=1}^{k}\left(\frac{1}{N}\sum_{t\neq i_{1+\sum_{p=1}^{j}m_{p}}}\sigma_{i_{1+\sum_{p=1}^{j}m_{p}},t}^{2}\right)^{\frac{n_{j}}{2}}\prod_{j=1}^{k}Z_{i_{1+\sum_{p=1}^{j}m_{p}}}^{n_{j}}.
\end{align}
Here recall that $(Z_{i}: i\geq 1)$ is a family of i.i.d. standard Normal random variables, independent of $\left(G_{i,j}:1\leq i\leq j\right)$ which are also independent standard Normal random variables. Also above $i=(i_{1},i_{2},\hdots i_{\sum m_{p}})$ with $i_{1+\sum_{p=1}^{k}m_{p}}=i_{1}$.
Using definition of Kreweras complement we have the following decomposition of \eqref{eq:trace_exp_i1_im}
\begin{align}\label{eq:sum_break_Kreweras}
    \sum_{i_{1},i_{2},\hdots,i_{\sum m_{p}}}=\sum_{i\in C\left(\sum m_{p},N\right)}+\sum_{\sigma\in NC_{2}\left(\sum m_{p}\right)}\sum_{i\in S(\sigma,N)}
\end{align}
Now $i\in S(\sigma,N)$ is same as saying $i\in \Psi\left(K(\sigma),N\right)$. Observe that
\begin{align}
    \mathbb{E}\left[\frac{1}{N}\Tr\left(\Bar{A}_{N}^{m_{1}}Y_{N}^{n_{1}}\hdots\Bar{A}_{N}^{m_{k}}Y_{N}^{n_{k}}\right)\right]
    &=\frac{1}{N}\sum_{i_{1}\hdots i_{\sum m_{p}}}\mathbb{E}\left[\prod_{j=1}^{\sum m_{p}}G_{i_{j}\wedge i_{j+1},i_{j}\vee i_{j+1}}\right]\prod_{j=1}^{\sum m_{p}}\frac{\sigma_{i_{j},i_{j+1}}}{\sqrt{N}} \times\nonumber\\
    &\prod_{j=1}^{k}\left(\frac{1}{N}\sum_{t\neq i_{1+\sum_{p=1}^{j}m_{p}}}\sigma_{i_{1+\sum_{p=1}^{j}m_{p}},t}^{2}\right)^{\frac{n_{j}}{2}}\mathbb{E}\left[\prod_{j=1}^{k}Z_{i_{1+\sum_{p=1}^{j}m_{p}}}^{n_{j}}\right]. \label{big:expression}
\end{align}
Using \eqref{eq:sum_break_Kreweras} let us first look at the sum over $S(\sigma,N)$. Consider a partition $\Pi$ of $\{1,2,\hdots,\sum m_{p}\}$, and take $i\in\Pi$, then observe that $\mathbb{E}\left[\prod_{j=1}^{\sum m_{p}}G_{i_{j}\wedge i_{j+1},i_{j}\vee i_{j+1}}\right]$ does not depend on $i$, but on the partition $\Pi$. Define 
$$
    \Phi(\Pi):=\mathbb{E}\left[\prod_{j=1}^{\sum m_{p}}G_{i_{j}\wedge i_{j+1},i_{j}\vee i_{j+1}}\right].
$$
First we focus on the factor 
$$
    \mathbb{E}\left[\prod_{j=1}^{k}Z_{i_{1+\sum_{p=1}^{j}m_{p}}}^{n_{j}}\right].
$$
from \eqref{big:expression}. One easy observation is that it only depends upon the partition $\Pi$ where $i$ belongs. Now consider a partition $\Pi$ and suppose $i\in\Pi$. For notational simplicity we identify $1$ by $1+\sum_{p=1}^{k}m_{p}$. Then consider the blocks where the indices $\{m_{1}+1,m_{1}+m_{2}+1,\hdots,1+\sum_{p=1}^{k} m_{p}\}$ belongs. Then we have
\begin{align}\label{eq:exp_Z}
    \mathbb{E}\left[\prod_{j=1}^{k}Z_{i_{1+\sum_{p=1}^{j}m_{p}}}^{n_{j}}\right]
    &=\mathbb{E}\left[\prod_{u\in\Pi}\prod_{\substack{j=1\\1+\sum_{p=1}^{j}m_{p}\in u}}^{k}Z_{i_{1+\sum_{p=1}^{j}m_{p}}}^{n_{j}}\right]=\prod_{u\in\Pi}\mathbb{E}\left[\prod_{\substack{j=1\\1+\sum_{p=1}^{j}m_{p}\in u}}^{k}Z_{l_{u}}^{n_{j}}\right]
\end{align}
Where $u$ denotes a block in $\Pi$ and $l_{u}$ denotes the corresponding representative element. So now if for some block $u\in\Pi$,
\begin{equation}\label{eq:odd}
\sum_{\substack{j=1\\1+\sum_{p=1}^{j}m_{p}\in u}}^{k}n_{j}\equiv 1 \pmod2 
\end{equation}
then the expectation in \eqref{eq:exp_Z} is 0, and hence in turn the whole expression is 0. So while looking at $\sum_{i\in S(\sigma,N)}$ in \eqref{eq:sum_break_Kreweras} we would only be looking at the case where for all $u\in K(\sigma)$
\begin{equation}\label{eq:even}
\sum_{\substack{j=1\\1+\sum_{p=1}^{j}m_{p}\in u}}^{k}n_{j}\equiv 0 \pmod2 
\end{equation}
holds.  Now remember the expectation in \eqref{eq:exp_Z} does not depend on the choice of $i\in S(\sigma,N)$, hence can go out of the sum. So we would be focusing on
\begin{align}\label{eq:E(P_i)}
    \frac{1}{N}\sum_{i\in S(\sigma,N)}\Phi(K(\sigma))\prod_{j=1}^{\sum m_{p}}\frac{\sigma_{i_{j},i_{j+1}}}{\sqrt{N}}\prod_{j=1}^{k}\left(\frac{1}{N}\sum_{t\neq i_{1+\sum_{p=1}^{j}m_{p}}}\sigma_{i_{1+\sum_{p=1}^{j}m_{p}},t}^{2}\right)^{\frac{n_{j}}{2}}
\end{align}
Define $m=\sum_{p=1}^{k}m_{p}$. We now use a combinatorial identity from proof of Theorem 1.2 in \citet{chakrabarty2018spectra}.
\begin{align}\label{eq:Spectra_Paper_eqn}
    \lim_{N\rightarrow\infty}N^{-(\frac{m}{2}+1)}\Phi(\Pi)\#\Psi(\Pi,N)=
    \begin{dcases}
            1 \text{ if $m$ is even and $\Pi=K(\sigma)$ for some $\sigma\in NC_{2}(m)$}\\
            0 \text{ otherwise }
   \end{dcases}
\end{align}
The above follows from standard arguments leading to proof of Wigner's semicircle law using method of moments. Observe that if 
\begin{align}\label{eq:m_oddNC2}
m\equiv 1 \pmod 2, \text{ then } NC_{2}(m)=\emptyset
\end{align}
where $\emptyset$ denotes the null set.  Hence we can safely ignore this case. So when considering the sum over $NC_{2}(m)$ we assume that $m$ is even. Since we are having $\Pi=K(\sigma)$, and recalling that $\Psi\left(K(\sigma),N\right)=S(\sigma,N)$, then using \eqref{eq:Spectra_Paper_eqn} we have
\begin{align}\label{eq:equation_for_limit}
    &\frac{1}{\#\Psi\left(K(\sigma),N\right)}\sum_{i\in\Psi(K\left(\sigma),N\right)}\prod_{j=1}^{m}\sigma_{i_{j},i_{j+1}}\prod_{j=1}^{k}\left(\frac{1}{N}\sum_{t\neq i_{1+\sum_{p=1}^{j}m_{p}}}\sigma_{i_{1+\sum_{p=1}^{j}m_{p}},t}^{2}\right)^{\frac{n_{j}}{2}}\nonumber\\
    &\approx\frac{1}{N^{\frac{m}{2}+1}}\sum_{i\in\Psi\left(K(\sigma),N\right)}\prod_{j=1}^{m}\sigma_{i_{j},i_{j+1}}\prod_{j=1}^{k}\left(\frac{1}{N}\sum_{t\neq i_{1+\sum_{p=1}^{j}m_{p}}}\sigma_{i_{1+\sum_{p=1}^{j}m_{p}},t}^{2}\right)^{\frac{n_{j}}{2}}
\end{align}
The above follows from  $\#\Psi\left(K(\sigma),N\right)=O\left(N^{\frac{m}{2}+1}\right)$. \footnote{$\approx$ implies they are same in the limit.}

Now let us consider the product $\prod_{j=1}^{m}\sigma_{i_{j},i_{j+1}}$. Observe that if all coordinates of $i$ were distinct then $i_{1}\rightarrow i_{2}\rightarrow\hdots\rightarrow i_{m}\rightarrow i_{1}$ forms a closed walk ($G_{W}$) on $m$ vertices. Now we have $i\in K(\sigma)$, and $K(\sigma)$ can have only $\frac{m}{2}+1$ many distinct block implies there are only $\frac{m}{2}+1$ many distinct values in $i$. Hence we have the following modification of $G_{W}$. Glue together the vertices $i_{a}$ and $i_{b}$ which appear in the same block. Since the previous graph $G_{W}$ was a closed walk, which is a connected graph, then the new graph (denoted by $G=(V,E)$) will be connected. Observe $V$ is the blocks in $i$, that is,  the blocks in $\Pi$ and $E$ is the edges between them (without repetition). Observe that $G$ only depends upon the positions in $i$ which are equal and which are not, hence the graph is independent of choice of $i\in S(\sigma,N)$. $G$ only depends upon the blocks of  $K(\sigma)$.\\

In the product $\prod_{j=1}^{m}\sigma_{i_{j},i_{j+1}}$ the number of times the unordered pair $(i_{j},i_{j+1})$ would appear is same as the number of times the edge between $i_{j}$ and $i_{j+1}$ is traversed in the graph $G$ while following the previous closed walk. If $i_{j}$ and $i_{j+1}$ belong in the same block, then the edge between them is basically a self loop. Since we are looking undirected graph, then the total number of repetition of the edge $(i_{j},i_{j+1})$ is the same as the total number of appearance of $(i_{j},i_{j+1})$ and $(i_{j+1},i_{j})$ in the product. This takes care of the symmetry constraint. Then we have
\begin{align}\label{eq:equal_prod_var}
    \prod_{j=1}^{m}\sigma_{i_{j},i_{j+1}}=\prod_{e\in E}\sigma_{e}^{t_{e}}
\end{align}
where $e=(a,b)$ denotes the edge between vertex $a$ and $b$ and $t_{e}$ is the number of times the edge is repeated in the closed walk. Independence of the graph from $i$ gives
\begin{align*}
    \Phi(K(\sigma))=\mathbb{E}\bigg[\prod_{e\in E}G_{e}^{t_{e}}\bigg]
\end{align*}
\\
We consider three exhaustive cases
\begin{enumerate}[label=\textbf{C.\arabic*}]
    \item $\forall e\in E, \quad t_{e}=2$;\label{itm:Case1}
    \item $\exists\, \,  e\in E$ such that $t_{e}=1$;\label{itm:Case2}
    \item $\forall e\in E, \, t_{e}\geq 2$ and $\exists\, e\in E$\label{itm:Case3} such that $t_{e}>2$.
\end{enumerate}
Let us first deal with \ref{itm:Case2}. Suppose $e_{1}$ is the edge appearing only once. Then
\begin{align*}
    \Phi(K(\sigma))=\mathbb{E}(G_{e_{1}})\mathbb{E}\left(\prod_{e\in E\setminus\{e_{1}\}}G_{e}^{t_{e}}\right)=0\quad\text{since $G_{e}\sim N(0,1)$}
\end{align*}
But this would contradict \eqref{eq:Spectra_Paper_eqn}. Hence this case does not happen.\\

Now let us look at \ref{itm:Case3}. Using \eqref{eq:equal_prod_var} we have $\sum_{e\in E}t_{e}=m$. By \ref{itm:Case3} we have $\sum_{e\in E}t_{e}>2|E|$, where $|E|$ denotes the number of edges in $G$. Since $G$ is a connected graph, then $|E|\geq |V|-1=\frac{m}{2}$. Then $\sum_{e\in E}t_{e}>m$. This is a contradiction again and hence this case is also not possible.\\

In order to deal with \ref{itm:Case1} we break it into two sub-cases.\\
\noindent
\textbf{Sub-Case 1:} The graph $G$ has a self edge. 
Then we have a connected graph $G$ having a self loop. Even if we remove the self loop still the graph remains connected. Say, $e_{1}$ is the self loop. Then consider the spanning tree of $G$ which we denote by $S$, then $|E(S)|=m/2$, then observe
\[
    t_{e_{1}}+\sum_{e\in E(S)}t_{e}\leq \sum_{e\in E}t_{e}\]
    which implies $
  2+\frac{m}{2}\times 2\leq m$ since here $t_{e}=2$ for all $e\in E$. This is a contradiction and it shows that the graph cannot have a self edge.

\textbf{Sub-Case 2:} We have a connected graph $G$ having no self loop and $\frac{m}{2}+1$ vertices and $t_{e}=2, \, \forall e\in E$. Then using $\sum_{e\in E}t_{e}=m$ we have $|E|=\frac{m}{2}$. Hence $G$ is a tree.\\

Since the graph $G$ depends only on $\sigma\in NC_{2}(m)$, then we redefine it as $G\equiv T(\sigma)=\left(V(\sigma),E(\sigma)\right)$, where $V(= V(\sigma))$ and $E(= E(\sigma))$ denotes the vertices and edges respectively.\\

Consider the $s^{th}$ block of $K(\sigma)$ and say the representative element is $l_{s}$ and define
\begin{align*}
    \gamma_{s}=\#\left\{1\leq j\leq k:\quad 1+\sum_{i=1}^{j}m_{i}\in \left\{\text{s$^{th}$ block}\right\}\right\}
\end{align*}
and 
\begin{align*}
    \left\{s_{1},s_{2},\cdots,s_{\gamma_{s}}\right\}=\left\{1\leq j\leq k:\quad 1+\sum_{i=1}^{j}m_{i}\in \left\{\text{s$^{th}$ block}\right\}\right\}
\end{align*}
Then, since $i_{a}=i_{b}$ if $a,b$ are in the same block, it is easy to see that
\begin{align}\label{eq:sigma_prod_sum_1st}
    \prod_{j=1}^{k}\left(\frac{1}{N}\sum_{t\neq i_{1+\sum_{p=1}^{j}m_{p}}}\sigma_{i_{\sum_{p=1}^{j}m_{p}},t}^{2}\right)^{\frac{n_{j}}{2}}=\prod_{s=1}^{\frac{m}{2}+1}\left(\frac{1}{N}\sum_{t\neq l_{s}}\sigma_{l_{s},t}^{2}\right)^{\sum_{j=1}^{\gamma_{s}}n_{s_{j}}/2}
\end{align}
Observe that from \eqref{eq:even} we have $$\sum_{j=1}^{\gamma_{s}}n_{s_{j}}=\sum_{\substack{j=1\\1+\sum_{p=1}^{j}m_{p}\in s}}^{k}n_{j}\equiv 0\pmod 2.$$ So define $\widetilde{n}_{s}=\sum_{j=1}^{\gamma_{s}}n_{s_{j}}/2$ and then $\sum_{s}\widetilde{n}_{s}=\frac{1}{2}\sum_{j=1}^{k}n_{j}$, where $\sum_{s}$ denotes sum over the blocks of $K(\sigma)$. Then
\begin{align}\label{eq:sigma_prod_sum_2nd}
    \prod_{s=1}^{\frac{m}{2}+1}\left(\frac{1}{N}\sum_{t\neq l_{s}}\sigma_{l_{s},t}^{2}\right)^{\sum_{j=1}^{\gamma_{s}}n_{s_{j}}/2}
    &=\frac{1}{N^{\sum \widetilde{n}_{s}}}\prod_{s=1}^{\frac{m}{2}+1}\left(\sum_{p_{1}\hdots p_{\widetilde{n}_{s}}\neq l_{s}}\sigma_{l_{s},p_{1}}^{2}\hdots \sigma_{l_{s},\widetilde{n}_{s}}^{2}\right)
\end{align}
Then combining \eqref{eq:sigma_prod_sum_1st} and \eqref{eq:sigma_prod_sum_2nd},  \eqref{eq:equation_for_limit} becomes
\begin{align}
    &\frac{1}{N^{\frac{m}{2}+1}}\sum_{i\in\Psi\left(K(\sigma),N\right)}\prod_{j=1}^{m}\sigma_{i_{j},i_{j+1}}\prod_{j=1}^{k}\left(\frac{1}{N}\sum_{t\neq i_{1+\sum_{p=1}^{j}m_{p}}}\sigma_{i_{1+\sum_{p=1}^{j}m_{p}},t}^{2}\right)^{\frac{n_{j}}{2}}\nonumber\\
    &=\frac{1}{N^{\frac{m}{2}+1}}\sum_{l_{1}\neq\hdots\neq l_{\frac{m}{2}+1}}\prod_{(u,v)\in E}\sigma_{l_{u},l_{v}}^{2}\prod_{s=1}^{\frac{m}{2}+1}\left(\frac{1}{N^{\widetilde{n}_{s}}}\sum_{p_{1},\hdots,p_{\widetilde{n}_{s}}\neq l_{s}}\sigma_{l_{s},p_{1}}^{2}\hdots\sigma_{l_{s},p_{\widetilde{n}_{s}}}^{2}\right)\label{eq:specify_p(s,i)}\\
    &=\frac{1}{N^{1+\frac{\sum m_{p}+n_{p}}{2}}}\sum_{\substack{l_{1}\neq\hdots\neq l_{m/2+1}\\p_{(s,1)}\hdots p_{(s,\widetilde{n}_{s})}\neq l_{s}\\ \forall s=1\hdots \frac{m}{2}+1}}\prod_{(u,v)\in E}\sigma_{l_{u},l_{v}}^{2}\prod_{s=1}^{\frac{m}{2}+1}\prod_{t=1}^{\widetilde{n}_{s}}\sigma_{l_{s},p_{(s,t)}}^{2}\label{eq:pre_def_S1}
\end{align}
where for a fixed $s\in \{1,\cdots, 1+m/2\}$, $p_{(s,i)}$ denotes the index $p_{i}$ in \eqref{eq:specify_p(s,i)}. We modify the graph $T(\sigma)$ as follows. We take vertex $s$ from $T(\sigma)$ and join $\widetilde{n}_{s}$ many vertices to it, denote those by $\{(s,1),(s,2),\hdots,(s,\widetilde{n}_{s})\}$, so that $s$ becomes the internal node of a star graph with the vertices $\{(s,1),(s,2),\hdots,(s,\widetilde{n}_{s})\}$ forming the leaves. An example of the modification is shown in Figure \ref{fig:Star_s}. 
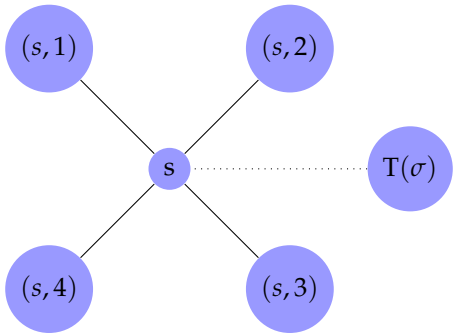
\begin{figure}[!ht]
\centering
\begin{tikzpicture}
  [scale=.8,auto=left,every node/.style={circle,fill=blue!40}]
  \node (n6) at (0,0) {s};
  \node (n5) at (2,2)  {$(s,2)$};
  \node (n4) at (-2,2) {$(s,1)$};
  \node (n3) at (-2,-2)  {$(s,4)$};
  \node (n2) at (2,-2)  {$(s,3)$};
  \node (n1) at (4,0) {T$(\sigma)$};

  \foreach \from/\to in {n6/n5,n6/n4,n6/n3,n6/n2}
    \draw (\from) -- (\to);
  \foreach \from/\to in {n1/n6}
    \draw[dotted] (\from) -- (\to);
\end{tikzpicture}
\caption{Star graph around vertex $s$, with $\widetilde{n}_{s}=4$}\label{fig:Star_s}
\end{figure}
Consider the new graph to be $\widetilde{T}(\sigma)=\left(\widetilde{V}(\sigma),\widetilde{E}(\sigma)\right)$, where $|\widetilde{V}(\sigma)|=\frac{1}{2}\sum_{p=1}^{k}m_{p}+\frac{1}{2}\sum_{s}\widetilde{n}_{s}+1=\frac{k}{2}+1$ and $|\widetilde{E}(\sigma)|=\frac{m}{2}+\sum_{s}\widetilde{n}_{s}=\frac{k}{2}$. Then observe that by construction $\widetilde{T}(\sigma)$ is a tree and the homomorphism density $t\left(\widetilde{T}(\sigma),W_{N}\right)$ becomes
\begin{align}
    t\left(\widetilde{T}(\sigma),W_{N}\right)
    &=\int_{[0,1]^{1+\sum \frac{m_{p}+n_{p}}{2}}}\prod_{(u,v)\in \widetilde{E}}W_{N}(x_{u},x_{v})\prod_{u}dx_{u}\nonumber\\
    &=\sum_{\substack{l_{1}\hdots l_{m/2+1}\\p_{(s,1)}\hdots p_{(s,\widetilde{n}_{s})}\\ \forall s=1\hdots \frac{m}{2}+1}}\int_{\otimes_{j=1}^{\frac{m}{2}+1}I_{l_{j}}\otimes_{s=1}^{\frac{m}{2}+1}\otimes_{t=1}^{\widetilde{n}_{s}}I_{p_{(s,t)}}}\prod_{(u,v)\in\widetilde{E}}W_{N}(x_{u},x_{v})\prod_{u}dx_{u}\nonumber\\
    &=\frac{1}{N^{1+\frac{\sum m_{p}+n_{p}}{2}}}\sum_{\substack{l_{1}\hdots l_{m/2+1}\\p_{(s,1)},\hdots p_{(s,\widetilde{n}_{s})}\\ \forall s=1,\hdots \frac{m}{2}+1}}\prod_{(u,v)\in E}\sigma_{l_{u},l_{v}}^{2}\prod_{s=1}^{\frac{m}{2}+1}\prod_{t=1}^{\widetilde{n}_{s}}\sigma_{l_{s},p_{(s,t)}}^{2}\label{eq:form_of_t}\\
    &\overset{N\rightarrow\infty}{\longrightarrow}t\left(\widetilde{T}(\sigma),W\right):=\beta(\sigma)\label{eq:graphon_convg_tree_sigma}
\end{align}
where the last limit follows from assumption \ref{itm:A1}. The expression from \eqref{eq:pre_def_S1} is redefined as
\begin{align*}
    S_{1}=\frac{1}{N^{1+\frac{\sum m_{p}+n_{p}}{2}}}\sum_{\substack{l_{1}\neq\hdots\neq l_{m/2+1}\\p_{(s,1)}\hdots p_{(s,\widetilde{n}_{s})}\neq l_{s}\\ \forall s=1\hdots \frac{m}{2}+1}}\prod_{(u,v)\in E}\sigma_{l_{u},l_{v}}^{2}\prod_{s=1}^{\frac{m}{2}+1}\prod_{t=1}^{\widetilde{n}_{s}}\sigma_{l_{s},p_{(s,t)}}^{2}
\end{align*}
Then using \eqref{eq:form_of_t} and a counting argument it follows that
\begin{align*}
    \left|S_{1}-t\left(\widetilde{T}(\sigma),W_{N}\right)\right|=O\left(N^{-1-\frac{\sum m_{p}+n_{p}}{2}}N^{\frac{\sum m_{p}+n_{p}}{2}}\right)=O\left(\frac{1}{N}\right)\overset{N\rightarrow \infty}{\longrightarrow 0}
\end{align*}
Hence using homomorphism density convergence from \eqref{eq:graphon_convg_tree_sigma} we conclude
\begin{align}
    S_{1}=\frac{1}{N^{1+\frac{\sum m_{p}+n_{p}}{2}}}\sum_{\substack{l_{1}\neq\hdots\neq l_{m/2+1}\\p_{(s,1)}\hdots p_{(s,\widetilde{n}_{s})}\neq l_{s}\\ \forall s=1\hdots \frac{m}{2}+1}}\prod_{(u,v)\in E}\sigma_{l_{u},l_{v}}^{2}\prod_{s=1}^{\frac{m}{2}+1}\prod_{t=1}^{\widetilde{n}_{s}}\sigma_{l_{s},p_{(s,t)}}^{2}\overset{N\rightarrow \infty}{\longrightarrow}t(\widetilde{T}(\sigma),W)
\end{align}
Recall that by construction $\widetilde{T}(\sigma)$ depends only on $\sigma$ for fixed values of $m_{i}, n_{i}$ for all $1\leq i\leq k$. Remember that we took \eqref{eq:even} holding for all $u\in K(\sigma)$. Hence redefine
\begin{align}\label{eq:graphon_convg_tree_sigma_1}
    \beta(\sigma)=
    \begin{cases}
        t\left(\widetilde{T}(\sigma),W\right) & \text{ if \eqref{eq:even} holds for all $u\in K(\sigma)$}\\
        0 & \text{ otherwise}
    \end{cases}
\end{align}
Also recall that $\mathbb{E}\left[\prod_{j=1}^{k}Z_{i_{1+\sum_{p=1}^{j}m_{p}}}^{n_{j}}\right]$ does not depend on choice of $i$, rather it only depends on $\sigma$. Define 
\begin{align}\label{eq:E_sigma}
    \mathcal{E}(\sigma)=\mathbb{E}\left[\prod_{j=1}^{k}Z_{i_{1+\sum_{p=1}^{j}m_{p}}}^{n_{j}}\right].
\end{align}
Note that by the discussion following \eqref{eq:odd}, $\mathcal{E}(\sigma)$ becomes 0 if \eqref{eq:even} does not hold for all $u\in K(\sigma)$.
Then expectation of \eqref{eq:trace_exp_i1_im} becomes
\begin{align}\label{eq:L2_expression}
    \frac{1}{N}\sum_{i\in S(\sigma,N)}\mathbb{E}\left[\prod_{j=1}^{\sum m_{p}}G_{i_{j}\wedge i_{j+1},i_{j}\vee i_{j+1}}\right]
    &\prod_{j=1}^{\sum m_{p}}\frac{\sigma_{i_{j},i_{j+1}}}{\sqrt{N}}\nonumber\\
    &\prod_{j=1}^{k}\left(\frac{1}{N}\sum_{t\neq i_{1+\sum_{p=1}^{j}m_{p}}}\sigma_{i_{1+\sum_{p=1}^{j}m_{p}},t}^{2}\right)^{\frac{n_{j}}{2}}\mathbb{E}\left[\prod_{j=1}^{k}Z_{i_{1+\sum_{p=1}^{j}m_{p}}}^{n_{j}}\right]\nonumber\\
    &\overset{N\rightarrow\infty}{\longrightarrow}\beta(\sigma)\mathcal{E}(\sigma)
\end{align}
where $\beta(\sigma)$ and $\mathcal{E}(\sigma)$ depends only on the choice of $m_{i},n_{i}$ for all $i=1,2,\hdots k$ and the corresponding partition $\sigma$. (Using \eqref{eq:m_oddNC2} we can safely replace the above expression by 0 when $\sum_{1}^{k}m_{p}$ is odd).

Now observe that we have shown convergence of expectation. But our objective was to show $L_{2}$ convergence. For any $i\in\{1,2,\hdots,N\}^{m}$ define:
\begin{align*}
    P_{i}=\prod_{j=1}^{m}G_{i_{j}\wedge i_{j+1},i_{j}\vee i_{j+1}}\prod_{j=1}^{m}\frac{\sigma_{i_{j},i_{j+1}}}{\sqrt{N}}\prod_{j=1}^{k}\left(\frac{1}{N}\sum_{t\neq i_{1+\sum_{p=1}^{j}m_{p}}}\sigma_{i_{1+\sum_{p=1}^{j}m_{p}},t}^{2}\right)^{\frac{n_{j}}{2}}\prod_{j=1}^{k}Z_{i_{1+\sum_{p=1}^{j}m_{p}}}^{n_{j}}
\end{align*}
and by \eqref{eq:L2_expression} we have $\mathbb{E}\left[\frac{1}{N}\sum_{i\in S(\sigma,N)}P_{i}\right]\rightarrow\beta(\sigma)\mathcal{E}(\sigma)$. To show that $$\frac{1}{N}\sum_{i\in S(\sigma,N)}P_{i}\overset{L_{2}}{\rightarrow}\beta(\sigma)\mathcal{E}(\sigma),$$ it is enough to show that $$\mathbb{E}\left[\left(\frac{1}{N}\sum_{i\in S(\sigma,N)}P_{i}\right)^{2}\right]\rightarrow\beta^{2}(\sigma)\mathcal{E}^{2}(\sigma).$$
With that goal observe that 
$$
    \mathbb{E}\left[\frac{1}{N}\sum_{i\in S(\sigma,N)}P_{i}\right]^{2}\rightarrow\beta^{2}(\sigma)\mathcal{E}^{2}(\sigma).
$$
Let us call $i,j\in\mathbb{N}^{m}$ to be disjoint if no co-ordinate of $i$ matches any co-ordinates of $j$ i.e. $\min_{1\leq u,v\leq m}|i_{u}-j_{v}|\geq 1$.
Now observe
\begin{align}
    \bigg[\mathbb{E}\bigg[\frac{1}{N}\sum_{i\in S(\sigma,N)}P_{i}\bigg]\bigg]^{2}
    &=\frac{1}{N^{2}}\sum_{i,j\in S(\sigma,N)}\mathbb{E}(P_{i})\mathbb{E}(P_{j})\nonumber\\
    &=\frac{1}{N^{2}}\sum_{\substack{i,j\in S(\sigma,N) \\ \text{ are disjoint }}}\mathbb{E}(P_{i})\mathbb{E}(P_{j})+\frac{1}{N^{2}}\sum_{\substack{i,j\in S(\sigma,N) \\ \text{ are not disjoint }}}\mathbb{E}(P_{i})\mathbb{E}(P_{j})\label{eq:Piconvg2}
\end{align}
If $i$ and $j$ are not disjoint then there is at least one common coordinate. Which implies there is a block in $i$ having same values as some block in $j$. Then we have 
\begin{align}
    &\{i,j\in S(\sigma,N) \text{ $i,j$ are not distinct}\}=\bigsqcup_{l=1}^{\frac{m}{2}+1}\{i,j\in S(\sigma,N) \text{ $i,j$ have exactly $l$ blocks common}\}\label{eq:sum_break}
\end{align}
where $\bigsqcup$ denotes disjoint union and define for $l=1,2,\hdots \left(\frac{m}{2}+1\right)$,
\begin{align*}
    &H_{l}=\{i,j\in S(\sigma,N): \,\, \text{ $i,j$ have exactly $l$ blocks common}\}
\end{align*}
By \eqref{eq:E(P_i)} we have
\begin{align}\label{eq:E(P_i)<MN^(power)}
    \mathbb{E}\left[P_{i}\right]
    &=N^{-\frac{m}{2}}\Phi(K(\sigma))\prod_{j=1}^{m}\sigma_{i_{j},i_{j+1}}\prod_{j=1}^{k}\left(\frac{1}{N}\sum_{t\neq i_{1+\sum_{p=1}^{j}m_{p}}}\sigma_{i_{1+\sum_{p=1}^{j}m_{p}},t}^{2}\right)^{\frac{n_{j}}{2}}\mathbb{E}\left[\prod_{j=1}^{k}Z_{i_{1+\sum_{p=1}^{j}m_{p}}}^{n_{j}}\right]\nonumber\\
    &\leq MN^{-\frac{m}{2}}
\end{align}
for some $0<M<\infty$, which depends only on $\sigma$. Observe that the above bound follows from \ref{itm:A1}. 
Then using \eqref{eq:E(P_i)<MN^(power)}, we have for all $l\in\left\{1,2,\hdots,\left(\frac{m}{2}+1\right)\right\}$,
\begin{align}
    \frac{1}{N^{2}}\sum_{i,j\in H_{l}}\mathbb{E}(P_{i})\mathbb{E}(P_{j})\preceq \frac{1}{N^{m+2}}\sum_{i,j\in H_{l}}1\preceq \frac{1}{N^{m+2}}O(N^{m+2-k})\rightarrow 0 \text{ as } N\rightarrow\infty \label{eq:PiPj 0}
\end{align}
Using \eqref{eq:PiPj 0} we conclude
\begin{align*}
    &\lim_{N\rightarrow\infty}\frac{1}{N^{2}}\sum_{i,j\in H_{l}}\mathbb{E}(P_{i})\mathbb{E}(P_{j})=0
\end{align*}
Then using \eqref{eq:sum_break} we have
\begin{align}
    \lim_{N\rightarrow\infty}\frac{1}{N^{2}}\sum_{\substack{i,j\in S(\sigma,N) \\ \text{ are not disjoint }}}\mathbb{E}(P_{i})\mathbb{E}(P_{j})=\lim_{N\rightarrow\infty}\sum_{l=1}^{\frac{m}{2}+1}\frac{1}{N^{2}}\sum_{i,j\in H_{l}}\mathbb{E}(P_{i})\mathbb{E}(P_{j})=0\label{eq:Pi not distinct}
\end{align}
Combining \eqref{eq:Pi not distinct} and \eqref{eq:Piconvg2} we have 
\begin{align}
    \lim_{N\rightarrow\infty}\frac{1}{N^{2}}\sum_{\substack{i,j\in S(\sigma,N) \\ \text{ are disjoint }}}\mathbb{E}(P_{i})\mathbb{E}(P_{j})=\beta^{2}(\sigma)\mathcal{E}^{2}(\sigma)\label{eq:L2convg1}
\end{align}
We have
\begin{align}
    \mathbb{E}\bigg[\bigg(\frac{1}{N}\sum_{i\in S(\sigma,N)}P_{i}\bigg)^{2}\bigg]
    &=\frac{1}{N^{2}}\sum_{i,j\in S(\sigma,N)}\mathbb{E}(P_{i}P_{j})\nonumber\\
    &=\frac{1}{N^{2}}\sum_{\substack{i,j\in S(\sigma,N)\\ \text{ are disjoint }}}\mathbb{E}(P_{i})\mathbb{E}(P_{j})+\frac{1}{N^{2}}\sum_{\substack{i,j\in S(\sigma,N) \text{ are not disjoint }}}\mathbb{E}(P_{i}P_{j})\label{eq:L2convg2}
\end{align}
Then for $i,j\in S(\sigma,N)$ such that they have at least one common coordinate, we have
\begin{align}\label{eq:E(PiPj)}
    \mathbb{E}(P_{i}P_{j})
    &=N^{-m}\widetilde{\Phi}_{i,j}\prod_{l=1}^{m}\sigma_{i_{l},i_{l+1}}^{2}\prod_{l=1}^{k}\left[\frac{1}{N}\sum_{t\neq i_{1+\sum_{p=1}^{l}m_{p}}}\sigma_{i_{1+\sum_{p=1}^{l}m_{p}},t}\right]^{\frac{n_{l}}{2}}\nonumber\\
    &\prod_{q=1}^{m}\sigma_{j_{q},j_{q+1}}^{2}\prod_{q=1}^{k}\left[\frac{1}{N}\sum_{t\neq j_{1+\sum_{p=1}^{q}m_{p}}}\sigma_{j_{1+\sum_{p=1}^{q}m_{p}},t}\right]^{\frac{n_{q}}{2}}\mathbb{E}\left[\prod_{l=1}^{k}Z_{i_{1+\sum_{p=1}^{l}m_{p}}}^{n_{l}}Z_{j_{1+\sum_{p=1}^{l}m_{p}}}^{n_{l}}\right]
\end{align}
where
\begin{align*}
    \widetilde{\Phi}_{i,j}=\mathbb{E}\left[\prod_{l=1}^{m}G_{i_{l}\wedge i_{l+1},i_{l}\vee i_{l+1}}G_{j_{l}\wedge j_{l+1},j_{l}\vee j_{l+1}}\right]
\end{align*}
 It is easy to observe that right hand side of \eqref{eq:E(PiPj)} is bounded by some constant depending on $\sigma\in NC_{2}(m)$. Then as before we have $\mathbb{E}(P_{i}P_{j})\leq \widetilde{M}N^{-m}$ for some $\widetilde{M}>0$, depending on $\sigma$, this again follows similarly as above from \ref{itm:A1}.
Then second term in \eqref{eq:L2convg2} goes to 0 by similar counting argument as in \eqref{eq:PiPj 0}. Then combining \eqref{eq:L2convg1} and \eqref{eq:L2convg2} we have
\begin{align*}
    \lim_{N\rightarrow\infty}\mathbb{E}\bigg[\bigg(\frac{1}{N}\sum_{i\in S(\sigma,N)}P_{i}\bigg)^{2}\bigg]=\beta^{2}(\sigma)\mathcal{E}^{2}(\sigma)
\end{align*}
Hence we have shown that
\begin{align*}
    \frac{1}{N}\sum_{i\in S(\sigma,N)}P_{i}\overset{L_{2}}{\rightarrow}\beta(\sigma)\mathcal{E}(\sigma)
\end{align*}
Now we want to show that $\frac{1}{N}\sum_{i\in C(m,N)}P_{i}\rightarrow 0$ in $L^{2}$. Observe here $m$ can be odd or even. Define $D=\{\Pi\in\text{ all partitions of }[m] \text{ such that }\Pi\neq K(\sigma), \,\, \forall \sigma\in NC_{2}(m)\}$. Observe that 
\begin{align*}
    \frac{1}{N}\sum_{i\in C(m,N)}P_{i}=\frac{1}{N}\sum_{\Pi\in D}\sum_{i\in\Psi(\Pi,N)}P_{i}
\end{align*}
So enough to show $$\frac{1}{N}\sum_{i\in\Psi(\Pi,N)}P_{i}\overset{L^{2}}{\rightarrow}0 \text{ for all $\Pi\in D$}.$$ With that objective fix $\Pi\in D$, 
\begin{align*}
    \mathbb{E}\bigg[\bigg(\frac{1}{N}\sum_{i\in\Psi(\Pi,N)}P_{i}\bigg)^{2}\bigg]
    &=\frac{1}{N^{2}}\sum_{i,j\in\Psi(\Pi,N)}\mathbb{E}(P_{i}P_{j})\\
    &=\frac{1}{N^{2}}\sum_{\substack{i,j\in \Psi(\Pi,N) \\ \text{ are disjoint }}}\mathbb{E}(P_{i})\mathbb{E}(P_{j})+\frac{1}{N^{2}}\sum_{\substack{i,j\in \Psi(\Pi,N) \\ \text{ are not disjoint }}}\mathbb{E}(P_{i}P_{j})
\end{align*}
Define
\begin{align*}
    &E(\Pi)=\{i,j\in \Psi(\Pi,N): \text{ $i,j$ are disjoint}\}\\
    &H_{l}(\Pi)=\{i,j\in \Psi(\Pi,N): \text{ $i,j$ have exactly $l$ blocks common}\}
\end{align*}
Suppose number of blocks in $\Pi$ is $b(\Pi)$. Then observe
\begin{align*}
    &\{i,j\in \Psi(\Pi,N): \text{ $i,j$ are not distinct}\}=\bigsqcup_{k=1}^{b(\Pi)}H_{k}(\Pi)
\end{align*}
Recall that $\Phi(\Pi)$ is independent of choice of $i,j\in\Pi$, it only depends upon $\Pi$. Then
\begin{align}
    0\leq\frac{1}{N^{2}}\sum_{i,j\in E(\Pi)}E(P_{i})E(P_{j})
    &\preceq \frac{1}{N^{m+2}}\Phi(\Pi)^{2}|E(\Pi)|\frac{1}{|E(\Pi)|}\sum_{i,j\in E(\Pi)}1\\
    &\leq \bigg(\frac{1}{N^{\frac{m}{2}+1}}\Phi(\Pi)\#\Psi(\Pi,N)\bigg)^{2}\rightarrow 0 \text{ as } N\rightarrow\infty.\label{eq:Equal_limit}
\end{align}
The last limit follows from \eqref{eq:Spectra_Paper_eqn}. Using above definitions we have
\begin{align}\label{eq:Decom_distinct}
    \frac{1}{N^{2}}\sum_{\substack{i,j\in \Psi(\Pi,N) \\ \text{ are not disjoint }}}\mathbb{E}(P_{i}P_{j})=\frac{1}{N^{2}}\sum_{k=1}^{b(\Pi)}\sum_{i,j\in H_{k}(\Pi)}\mathbb{E}(P_{i}P_{j})
\end{align}
Let us look at the case when we have exactly one block common in $i$ and $j$. Then $|H_{1}(\Pi)|=O(N^{2b(\Pi)-1})$. It is easy to see that there exists $h>0$, depending on $\Pi$ such that 
$$
    \mathbb{E}(P_{i}P_{j})\leq hN^{-m}, \,\,  \forall i,j\in \Psi(\Pi,N).
$$
The above bound follows from the discussion after \eqref{eq:E(PiPj)}, replacing $K(\sigma)$ by $\Pi$. Now consider two cases depending upon the values of $b(\Pi)$.\\

\textit{Case 1:} Let $b(\Pi)\leq \frac{m}{2}+1$. Then we have
\begin{align}
    \frac{1}{N^{2}}\sum_{i,j\in H_{1}(\Pi)}\mathbb{E}(P_{i}P_{j})\preceq O\left(N^{-m-2}N^{m+1}\right)=O(N^{-1})\overset{N\rightarrow\infty}{\longrightarrow} 0.
\end{align}

\textit{Case 2:} Let $b(\Pi)>\frac{m}{2}+1$. Observe
\begin{align*}
    \mathbb{E}\left(P_{i}P_{j}\right)\leq\frac{C}{N^{m}}\mathbb{E}\left[\prod_{l=1}^{m}G_{i_{l}\wedge i_{l+1},i_{l}\vee i_{l+1}}\prod_{l=1}^{m}G_{j_{l}\wedge j_{l+1},j_{l}\vee j_{l+1}}\right],
\end{align*}
for some constant $C>0$. Now consider the two closed walks (remember we had $i_{m+1}=i_{1}$ and $j_{m+1}=j_{1}$)
\[
   i_{1}\rightarrow i_{2}\rightarrow i_{3}\rightarrow\hdots i_{m}\rightarrow i_{m+1} \text{ and }
    j_{1}\rightarrow j_{2}\rightarrow j_{3}\rightarrow\hdots j_{m}\rightarrow j_{m+1}.
\]
Now join them at the common coordinates that is where $i_{p}=j_{q}$ and for some $p,q\in\{1,\cdots, m\}$ and glue together the coordinates of $i$ which are in the same blocks of $\Pi$, and do the same for $j$. Then consider the graph $\widehat{G}=\left(\widehat{V},\widehat{E}\right)$ and then we have
\begin{align}\label{eq:construction_equiv}
    \prod_{l=1}^{m}G_{i_{l}\wedge i_{l+1},i_{l}\vee i_{l+1}}\prod_{l=1}^{m}G_{j_{l}\wedge j_{l+1},j_{l}\vee j_{l+1}}=\prod_{e\in \widehat{E}}G_{e}^{t_{e}}
\end{align}
and hence
\begin{align}\label{eq:graph_case2_estimate}
    \mathbb{E}(P_{i}P_{j})\leq\frac{C}{N^{m}}\mathbb{E}\left[\prod_{e\in \widehat{E}}G_{e}^{t_{e}}\right]
\end{align}
where $t_{e}$ is the number of times the edge $e$ is repeated in the graph $\widehat{G}$. Observe that each edge is repeated at least once, since the two closed walks were connected graph and in the resultant graph we did not remove any edge. So if $\mathbb{E}(P_{i}P_{j})\neq 0$, then $\forall e\in \widehat{E}, t_{e}\geq 2$. Since $i,j\in H_{1}(\Pi)$, then total number of distinct block in the combined vector $(i,j)$ would be $2b(\Pi)-1$. Now by construction the blocks form the vertices of $\widehat{G}$, hence as the graph is connected then $|\widehat{E}|\geq 2b(\Pi)-2$. Then $$\sum_{e\in\widehat{E}}t_{e}\geq2(2b(\Pi)-2)>2\left(2\left(\frac{m}{2}+1\right)-2\right)=2m ,$$ but by \eqref{eq:construction_equiv} we have $\sum_{e\in\widehat{E}}t_{e}=2m$. Hence a contradiction which implies that in this case $\mathbb{E}(P_{i}P_{j})=0,\forall i,j\in H_{1}(\Pi)$.\\
\\
So we have shown that 
\begin{align}
    \frac{1}{N^{2}}\sum_{i,j\in H_{1}(\Pi)}\mathbb{E}(P_{i}P_{j})\overset{N\rightarrow\infty}{\longrightarrow}0
\end{align}
Similarly we can draw the same conclusion for any $l\in\left\{1,2,\hdots, b(\Pi)\right\}$. Hence we can conclude by \eqref{eq:Decom_distinct} that
\begin{align}\label{eq:distinct_limit}
    \frac{1}{N^{2}}\sum_{\substack{i,j\in\Psi(\Pi,N) \\ \text{ are not distinct}}}\mathbb{E}\left(P_{i}P_{j}\right)\overset{N\rightarrow\infty}{\longrightarrow}0
\end{align}
Then combining \eqref{eq:Equal_limit} and \eqref{eq:distinct_limit} we have proved the $L_{2}$ convergence for $C(m,N)$. 
Hence using \eqref{eq:trace_expansion} we have
\begin{align}
    \frac{1}{N}\Tr\left(\widetilde{\Delta}_{N}^{k}\right)\overset{L_{2}}{\rightarrow}\sum_{\mathcal{P}_{k}}\sum_{\sigma\in NC_{2}(\sum m_{p})}\beta(\sigma)\mathcal{E}(\sigma)
\end{align}
where $\mathcal{P}_{k}$ is as defined in Section \ref{sec:moment_description}.\\

\textbf{Case 2: $k$ is odd}\\

Recall the expansion of moment expression as in \eqref{eq:trace_expansion} in the beginning of the previous case.
\begin{align}\label{eq:odd_expr}
    \frac{1}{N}\Tr\left(\widetilde{\Delta}_{N}^{k}\right)=\frac{1}{N}\sum_{\substack{m_{1},m_{2}\hdots,m_{k}\\n_{1},n_{2}\hdots,n_{k}}}\Tr\left(\Bar{A}_{N}^{m_{1}}Y_{N}^{n_{1}}\hdots\Bar{A}_{N}^{m_{k}}Y_{N}^{n_{k}}\right)
\end{align}
where the sum is over all the terms in the expansion of $\left(\Bar{A}_{N}+Y_{N}\right)^{k}$, that is we have $2^{k}$ many terms and for every choice of $m_{1},n_{1},\hdots,m_{k},n_{k}$ we have $\sum_{p=1}^{k} m_{p}+n_{p}=k$. So in order to show $\frac{1}{N}\Tr\left(\widetilde{\Delta}_{N}^{k}\right)\overset{L_{2}}{\rightarrow}0$, it is enough to show that 
$$
    \frac{1}{N}\Tr\left(\Bar{A}_{N}^{m_{1}}Y_{N}^{n_{1}}\hdots\Bar{A}_{N}^{m_{k}}Y_{N}^{n_{k}}\right)\overset{L_{2}}{\rightarrow}0
$$ Then due to \ref{itm:A1} we have
\begin{align}\label{eq:Trace_odd_bound}
    \left|\frac{1}{N}\Tr\left(\Bar{A}_{N}^{m_{1}}Y_{N}^{n_{1}}\hdots\Bar{A}_{N}^{m_{k}}Y_{N}^{n_{k}}\right)\right|\preceq 
    &\frac{1}{N}\sum_{i_{1}\hdots i_{\sum m_{p}}}\left|\prod_{j=1}^{\sum m_{p}}G_{i_{j}\wedge i_{j+1},i_{j}\vee i_{j+1}}\right|\prod_{j=1}^{\sum m_{p}}\frac{1}{\sqrt{N}}\left|\prod_{j=1}^{k}Z_{i_{1+\sum_{p=1}^{j}m_{p}}}^{n_{j}}\right|
\end{align}
\textbf{Sub-Case 1:} $\sum_{p=1}^{k}m_{p}$ is odd\\

Observe that from \eqref{eq:Trace_odd_bound},
\begin{align*}
    \mathbb{E}\left[\frac{1}{N^{2}}\Tr^{2}\left(\Bar{A}_{N}^{m_{1}}Y_{N}^{n_{1}}\hdots\Bar{A}_{N}^{m_{k}}Y_{N}^{n_{k}}\right)\right]
    &\preceq\mathbb{E}\left[\frac{1}{N^{2+\sum m_{p}}}\sum_{i,j}\prod_{l=1}^{\sum m_{p}}G_{i_{l}\wedge i_{l+1},i_{l}\vee i_{l+1}}\prod_{l=1}^{\sum m_{p}}G_{j_{l}\wedge j_{l+1},j_{l}\vee j_{l+1}}\right]\nonumber\\
    &\mathbb{E}\left[\prod_{l=1}^{k}Z_{i_{1+\sum_{p=1}^{l}m_{p}}}^{n_{l}}Z_{j_{1+\sum_{p=1}^{l}m_{p}}}^{n_{l}}\right]
\end{align*}
where $i,j$ are the $m$ dimensional vectors defined in \eqref{eq:trace_exp_i1_im}. The last term is uniformly bounded over $i,j$. (Since there are finitely such values depending on the partitions of $\{1,\cdots,m\}$, where $m=\sum_{j=1}^{k}m_{j}$). Then 
\begin{align*}
    \mathbb{E}\left[\frac{1}{N^{2}}\Tr^{2}\left(\Bar{A}_{N}^{m_{1}}Y_{N}^{n_{1}}\hdots\Bar{A}_{N}^{m_{k}}Y_{N}^{n_{k}}\right)\right]
    &\preceq\mathbb{E}\left[\frac{1}{N^{2+\sum m_{p}}}\sum_{i,j}\prod_{l=1}^{\sum m_{p}}G_{i_{l}\wedge i_{l+1},i_{l}\vee i_{l+1}}\prod_{l=1}^{\sum m_{p}}G_{j_{l}\wedge j_{l+1},j_{l}\vee j_{l+1}}\right]\nonumber\\
    &\leq\mathbb{E}\bigg[\bigg(\frac{1}{N}\Tr(U_{N}^{\sum m_{p}})\bigg)^{2}\bigg]\overset{N\rightarrow\infty}{\longrightarrow}0
\end{align*}
where $U_{N}$ is a $N\times N$ Wigner matrix having all entries i.i.d. from $\mathrm{N}(0,1/N)$ and satisfying the symmetry constraint. Observe that $\frac{1}{N}\Tr(U_{N}^{\sum m_{p}})$ is the $(\sum m_{p})^{th}$ moment of ESD$(Z_{N})$. The limit follows by the classical estimates for Wigner matrices with i.i.d.\ entries (see for example \citet{anderson2010introduction}).\\
\\
\textbf{Sub-Case 2:} $\sum_{p=1}^{k}m_{p}$ is even\\

Define 
$$
    \widetilde{P}_{i}=\frac{1}{N^{\sum m_{p}/2}}\prod_{l=1}^{\sum m_{p}}G_{i_{l}\wedge i_{l+1},i_{l}\vee i_{l+1}}
$$
So we have to show
\begin{align*}
    \frac{1}{N}\sum_{i}\widetilde{P}_{i}\prod_{l=1}^{k}Z_{i_{1+\sum_{p=1}^{l}m_{p}}}^{n_{l}}\overset{L_{2}}{\longrightarrow}0
\end{align*}
Now take $\Pi$ to be a partition of $\{1,\cdots ,m\}$, then due to a decomposition similar to \eqref{eq:sum_break_Kreweras} enough to show that 
\begin{align*}
    \frac{1}{N}\sum_{i\in\Psi(\Pi,N)}\widetilde{P}_{i}\prod_{l=1}^{k}Z_{i_{1+\sum_{p=1}^{l}m_{p}}}^{n_{l}}\overset{L_{2}}{\longrightarrow}0
\end{align*}
Then observe
\begin{align*}
    &\mathbb{E}\left[\left[\frac{1}{N}\sum_{i\in\Psi(\Pi,N)}\widetilde{P}_{i}\prod_{l=1}^{k}Z_{i_{1+\sum_{p=1}^{l}m_{p}}}^{n_{l}}\right]^{2}\right]
    =\frac{1}{N^{2}}\sum_{E(\Pi)}\mathbb{E}(\widetilde{P}_{i}\widetilde{P}_{j})\mathbb{E}\left(\prod_{l=1}^{k}Z_{i_{1+\sum_{p=1}^{l}m_{p}}}^{n_{l}}\right)\mathbb{E}\left(\prod_{l=1}^{k}Z_{j_{1+\sum_{p=1}^{l}m_{p}}}^{n_{l}}\right)\nonumber\\
    &+\frac{1}{N^{2}}\sum_{\substack{i\text{ and }j \\ \text{ are not distinct}} }\mathbb{E}\left(\widetilde{P}_{i}\widetilde{P}_{j}\prod_{l=1}^{k}Z_{i_{1+\sum_{p=1}^{l}m_{p}}}^{n_{l}}\prod_{l=1}^{k}Z_{j_{1+\sum_{p=1}^{l}m_{p}}}^{n_{l}}\right)
\end{align*}
Now since $k$ is odd, then $\sum m_{p}+n_{p}=k$ is odd, and hence $\sum n_{p}$ is odd. Hence it's easy to see that the first term is $0$. So we focus on the second term only. Using independence and the uniform bound of 
$$
    \mathbb{E}\left(\prod_{l=1}^{k}Z_{i_{1+\sum_{p=1}^{l}m_{p}}}^{n_{l}}\prod_{l=1}^{k}Z_{j_{1+\sum_{p=1}^{l}m_{p}}}^{n_{l}}\right)
$$
we only focus on $\mathbb{E}(\widetilde{P}_{i}\widetilde{P}_{j})$. Similar to \eqref{eq:Decom_distinct} it is enough to show that
\begin{align*}
    \frac{1}{N^{2}}\sum_{H_{k}(\Pi)}\mathbb{E}(\widetilde{P}_{i}\widetilde{P}_{j})\rightarrow 0
\end{align*}
Consider $b(\Pi)$ to be number of blocks in $\Pi$ and remember $m=\sum m_{p}$. Let us look at the case when we have exactly one block common in $i$ and $j$. Then $|H_{1}(\Pi)|=O(N^{2b(\Pi)-1})$. It is easy to see that  there exists $h>0$, depending on $\Pi$ such that $\mathbb{E}(\widetilde{P}_{i}\widetilde{P}_{j})\leq hN^{-m}$ for all $i,j\in \Psi(\Pi,N)$. This essentially follows as above in previous case. Now consider two cases depending upon the values of $b(\Pi)$.

\textit{Case 1:} $b(\Pi)\leq \frac{m}{2}+1$.  Then we have $2b(\Pi)-1\leq m+1$, which implies
\begin{align*}
    \frac{1}{N^{2}}\sum_{i,j\in H_{1}(\Pi)}\mathbb{E}(\widetilde{P}_{i}\widetilde{P}_{j})\preceq O\left(N^{-m-2}N^{m+1}\right)=O(N^{-1})\overset{N\rightarrow\infty}{\longrightarrow} 0
\end{align*}

\textit{Case 2:} $b(\Pi)>\frac{m}{2}+1$. Observe
\begin{align*}
    \mathbb{E}\left(\widetilde{P}_{i}\widetilde{P}_{j}\right)=\frac{1}{N^{m}}\mathbb{E}\left[\prod_{l=1}^{m}G_{i_{l}\wedge i_{l+1},i_{l}\vee i_{l+1}}\prod_{l=1}^{m}G_{j_{l}\wedge j_{l+1},j_{l}\vee j_{l+1}}\right]
\end{align*}
This case is exactly similar to Case 2 as before. Using the same estimates as in \eqref{eq:graph_case2_estimate} we can show 
\begin{align*}
    \frac{1}{N^{2}}\sum_{i,j\in H_{1}(\Pi)}\mathbb{E}(\widetilde{P}_{i}\widetilde{P}_{j})\overset{N\rightarrow\infty}{\longrightarrow}0.
\end{align*}
Similarly can conclude the same for any $k\in\left\{1,2,\hdots, b(\Pi)\right\}$. Hence we have
\begin{align*}
    \frac{1}{N^{2}}\sum_{\substack{i,j\in\Psi(\Pi,N) \\ \text{ are not distinct}}}\mathbb{E}\left(\widetilde{P}_{i}\widetilde{P}_{j}\right)\overset{N\rightarrow\infty}{\longrightarrow}0
\end{align*}
Then we have shown that \eqref{eq:odd_expr} converges to 0 in $L^{2}$, that is
\begin{align*}
    \frac{1}{N}\Tr\left(\Tilde{\Delta}_{N}^{k}\right)\overset{L_{2}}{\longrightarrow}0.
\end{align*}

\subsubsection{\textbf{Existence of unique limiting distribution}}

Before proceeding let us define
\begin{align*}
    \beta_{k}=
    \begin{cases}
        \sum_{\mathcal{P}_{k}}\sum_{\sigma\in NC_{2}(\sum m_{p})}\beta(\sigma)\mathcal{E}(\sigma) & \text{ if $k$ is even}\\
        0 & \text{ if $k$ is odd}
    \end{cases}
\end{align*}
then by Lemma B.1 from \citet{bai2010spectral} we can easily see that there exists a probability measure $\nu$ identified by the moments $\beta_{k}$. To show uniqueness of $\nu$ we invoke again Theorem 1 of \citet{lin2017recent}, that is show that
\begin{align}\label{eq:moment_lim_cond}
        \limsup_{k\rightarrow\infty}\frac{1}{2k}\beta_{2k}^{1/2k}<\infty
\end{align}
For all $l\in\{1,2,\hdots,\sum_{p=1}^{k}m_{p}\}$, define 
\begin{align*}
    \alpha_{l}=
    \begin{cases}
        n_{j} & \text{ if } l=1+\sum_{p=1}^{j}m_{p},\,\,\forall j=1,\hdots, k-1,\\
        n_{k} & \text{ if } l=1, \\
        0 & \text{ otherwise }.
    \end{cases}
\end{align*}
Recall that we use $1+\sum_{p=1}^{k}m_{p}$ as an identifier for $1$. Then it is easy to observe that 
\begin{align*}
    \prod_{j=1}^{k}Z_{i_{1+\sum_{p=1}^{j}m_{p}}}^{n_{j}}=\prod_{u\in K(\sigma)}Z_{u}^{\sum_{l\in u}\alpha_{l}},\quad\forall i\in \Psi(K(\sigma),N)
\end{align*}
where the product is taken over the blocks in $K(\sigma)$. Hence by definition of $\mathcal{E}(\sigma)$ we have 
\begin{align*}
    \mathcal{E}(\sigma)=\mathbb{E}\left[\prod_{u\in K(\sigma)}Z_{u}^{\sum_{l\in u}\alpha_{l}}\right]=\prod_{u\in K(\sigma)}\mathbb{E}\left[Z_{1}^{\sum_{l\in u}\alpha_{l}}\right].
\end{align*}
The above equality follows since $Z_{u}$'s are all i.i.d. N$(0,1)$ for all $u\in K(\sigma)$. Using the expression for moments of the Gaussian distribution,
\begin{align*}
    \prod_{u\in K(\sigma)}\mathbb{E}\left[Z_{1}^{\sum_{l\in u}\alpha_{l}}\right]
    &\leq \prod_{u\in K(\sigma)}\frac{\left(2\sum_{l\in u}\alpha_{l}\right)!}{2^{\sum_{l\in u}\alpha_{l}}\left(\sum_{l\in u}\alpha_{l}\right)!}\nonumber\\
    &\leq 2^{\sum_{u\in K(\sigma)}\sum_{l\in u}\alpha_{l}}\left(\sum_{u\in K(\sigma)}\sum_{l\in u}\alpha_{l}\right)!\nonumber\\
    &\leq 2^{k}k!
\end{align*}
The last inequality follows from $\sum_{j=1}^{k}n_{j}\leq k$. Observe that by definition of $\beta(\sigma)$ it can be easily seen that there exists a constant  $C>0$ such that $\beta(\sigma)\leq C^{k}$. Recall $C_{l}$ is the $l^{th}$ catalan number and $C_{a}\leq C_{b}$ whenever $a\leq b$, then taking $k\in 2\mathbb{N}$
\begin{align*}
    \beta_{k}
    &\leq \sum_{\mathcal{P}_{k}}\sum_{\sigma\in NC_{2}(\sum m_{p})}C^{k}2^{k}k!
    \leq \sum_{\mathcal{P}_{k}}C_{\sum m_{p}}(2C)^{k}k!\nonumber\\
    &\leq \sum_{\mathcal{P}_{k}}C_{k}(2C)^{k}k!
    \leq (4C)^{k}C_{k}k!
\end{align*}
The above inequalities follows since $\sum_{p=1}^{k}m_{p}\leq k$ and the sum over $\mathcal{P}_{k}$ can have at most $2^{k}$ many terms. 
Using Sterling's approximation we have
\begin{align*}
    \frac{1}{k}\beta_{k}^{\frac{1}{k}}\leq 4C\frac{1}{(k+1)^{\frac{1}{k}}}\frac{4e^{-\left(1+\frac{1}{k}\right)}}{\pi^{\frac{1}{k}}}
\end{align*}
and subsequently we conclude that
\begin{align}\label{eq:Mgf_limit_cond}
    \limsup_{k\rightarrow\infty}\frac{1}{2k}\beta_{2k}^{\frac{1}{2k}}<\infty.
\end{align}
Hence the measure $\nu$ is uniquely identified by the moments $\beta_{k}$. Finally note that the moment generating function of $\nu$ is finite around origin and the odd moments vanish it easily follows that $\nu$ is symmetric around origin.



\subsubsection{\textbf{$\nu$ has unbounded support}}
To prove the unbounded support we shall use the following lemma from \citet{chakrabarty2018spectra}. 
\begin{lemma}\citet[Fact A.5]{chakrabarty2018spectra}\label{Fact: A5}
Suppose that for al $n\geq 1, Z_{n1}\geq\hdots\geq Z_{nn}$ are random variables such that 
\begin{align}
    \lim_{n\rightarrow\infty}\frac{1}{n}\sum_{j=1}^{n}\delta_{Z_{nj}}=\mu \text{ weakly in probability},
\end{align}
for some probability measure $\mu$ on $\mathbb{R}$, where $\delta_{x}$ is the probability measure that puts mass 1 at $x$. Then,
\begin{align}
    \lim_{p\rightarrow 0}\limsup_{n\rightarrow\infty}Z_{n[np]}=\sup(\text{supp}(\mu)) \text{ almost surely}
\end{align}
where $[x]$ denotes the smallest integer larger than or equal to $x$.
\end{lemma}

Since there exists an open set $U\subseteq [0,1]^{2}$ such that $W>0$ on $U$, then using Lemma \ref{Fact: A5} and Lemma \ref{thm:ESD_YN} we have
\begin{align}\label{eq:bound_esd_YN}
    \lim_{p\rightarrow0}\limsup_{N\rightarrow\infty}\lambda_{[Np]}(Y_{N})=\sup(\text{supp}(\zeta))=\infty \text{ almost surely}
\end{align}
Since in the following we would be dealing with almost sure statements, then we refrain from writing almost surely every time. Here $[x]$ denotes the smallest integer larger than or equal to $x$ and $\lambda_{k}(\Sigma)$ denotes the $k^{th}$ largest eigenvalue of the matrix $\Sigma$. By Weyl's inequality we have
\begin{align*}
    \lambda_{2[Np]-1}(Y_{N})\leq \lambda_{[Np]}(\Bar{A}_{N}+Y_{N})+\lambda_{[Np]}(-\Bar{A}_{N})
\end{align*}
Now by definition of $\Bar{A}_{N}$, observe that $-\Bar{A}_{N}$ satisfies the assumptions of Theorem 1 of \citet{zhu2020graphon}, and let
\begin{align*}
    \text{ESD}\left(-\Bar{A}_{N}\right)\rightarrow \kappa_{A}
\end{align*}
for some unique measure $\kappa_{A}$ on $\mathbb{R}$. Suppose $\widetilde{X}\sim\kappa_{A}$. Then we have from~\eqref{moments:zhu}
\begin{align*}
    \mathbb{E}\left[\widetilde{X}^{k}\right]=
    \begin{cases}
        0 & k \text{ is odd}\\
        \sum_{j=1}^{C_{k/2}}t\left(T_{j}^{k/2+1},W\right) & k\text{ is even}
    \end{cases}
\end{align*}
where $T_{j}^{k+1}$ denotes the $j^{th}$ rooted planar tree of $k+1$ vertices and $C_{k}$  is the $k^{th}$ Catalan number. Then observe that
\begin{align*}
    \|\widetilde{X}\|_{2k}=\left(\sum_{j=1}^{C_{k}}t(T_{j}^{k+1},W)\right)^{\frac{1}{2k}}\leq \left(C_{k}C^{k}\right)^{\frac{1}{2k}}
\end{align*}
where $\|\cdot\|_{2k}$ denotes the $L^{2k}$ norm and $C$ is from \ref{itm:A1}. It follows that there exists a constant $C^{\circ}_{1}\in(0, \,\infty)$ such that 
\begin{align*}
    \lim_{k\rightarrow\infty}\left(C_{k}C^{k}\right)^{\frac{1}{2k}}= C^{\frac{1}{2}}C^{\circ}_{1}.
\end{align*}
Hence we can conclude that there exists a constant  $C^{\circ}>0$ such that
\begin{align*}
    \left\|\widetilde{X}\right\|_{2k}\leq C^{\circ},\quad\forall k\geq 1
\end{align*}
Now since $\left\|\widetilde{X}\right\|_{k}\leq\left\|\widetilde{X}\right\|_{2k}$ we have
\begin{align*}
    \left\|\widetilde{X}\right\|_{k}\leq C^{\circ},\quad\forall k\geq 1
\end{align*}
Now define $f=\min\left(|\widetilde{X}|,C^{\circ}+1\right)\in L^{\infty}$. Then $\|f\|_{p}\leq \|\widetilde{X}\|_{p}\leq C^{\circ}$. Now since $f\in L^{\infty}\cap L^{p}$ for some $p\geq 1$, then $\|f\|_{\infty}=\lim_{p\rightarrow\infty}\|f\|_{p}\leq C^{\circ}$. Then we have $\min\left(|\widetilde{X}|,C^{\circ}+1\right)\leq C^{\circ} \text{ a.e. }$ which implies $$ |\widetilde{X}|\leq C^{\circ} \text{ a.e.}$$ Thus 
\begin{align}\label{eq:compact_sup}
    \sup(\text{supp}(\kappa_{A}))\leq C^{\circ}.
\end{align}
Hence using Lemma \ref{Fact: A5} we have
\begin{align*}
    \lim_{p\rightarrow0}\limsup_{N\rightarrow\infty}\lambda_{[Np]}(-\Bar{A}_{N})=\sup(\text{supp}(\kappa_{A}))\leq C^{\circ}
\end{align*}
Fix $\varepsilon>0$, then there exists $\eta>0$ such that for all $p\in (0,\eta\wedge\frac{1}{2})$, $\limsup_{N\rightarrow\infty}\lambda_{[Np]}(-\Bar{A}_{N})\leq C^{\circ}+\varepsilon$. So for all $p\in (0,\eta\wedge\frac{1}{2})$ we have
\begin{align*}
    \limsup_{N\rightarrow\infty}\lambda_{[Np]}(\Bar{A}_{N}+Y_{N})
    &\geq \limsup_{N\rightarrow\infty}\lambda_{2[Np]-1}(Y_{N})-\liminf_{N\rightarrow\infty}\lambda_{[Np]}(-\Bar{A}_{N})\\
    &\geq \limsup_{N\rightarrow\infty}\lambda_{2[Np]-1}(Y_{N})-(C^{\circ}+\varepsilon)
\end{align*}
Now from \eqref{eq:bound_esd_YN} it can be easily shown that
$    \lim_{p\rightarrow 0}\limsup_{N\rightarrow\infty}\lambda_{2[Np]-1}(Y_{N})=\infty
$
and hence we have 
\begin{align*}
    \sup(\text{supp}(\nu))=\lim_{p\rightarrow 0}\limsup_{N\rightarrow\infty}\lambda_{[Np]}(\Bar{A}_{N}+Y_{N})=\infty
\end{align*}
which implies that $\nu$ has unbounded support.

\subsubsection{\textbf{Identification of moments of $\nu$}}
The proof of the previous result already gave the convergence of the moments. We briefly browse through the expressions to write them in terms of the description used in Section~\ref{sec:moment_description}.
First, remember that the odd moments are 0. The case for $k=0$ is trivial, so fix $k$ in $2\mathbb{N}$. Then it is enough to look at the convergence of the terms (recall \eqref{big:expression})
\begin{align}\label{eq:full_sum}
    \mathbb{E}\left[\frac{1}{N}\Tr\left(\Bar{A}_{N}^{m_{1}}Y_{N}^{n_{1}}\hdots\Bar{A}_{N}^{m_{k}}Y_{N}^{n_{k}}\right)\right]
    &=\frac{1}{N}\sum_{i_{1}\hdots i_{\sum m_{p}}}\mathbb{E}\left[\prod_{j=1}^{\sum m_{p}}G_{i_{j}\wedge i_{j+1},i_{j}\vee i_{j+1}}\right]\prod_{j=1}^{\sum m_{p}}\frac{\sigma_{i_{j},i_{j+1}}}{\sqrt{N}}\nonumber\\
    &\prod_{j=1}^{k}\left(\frac{1}{N}\sum_{t\neq i_{1+\sum_{p=1}^{j}m_{p}}}\sigma_{i_{1+\sum_{p=1}^{j}m_{p}},t}^{2}\right)^{\frac{n_{j}}{2}}\mathbb{E}\left[\prod_{j=1}^{k}Z_{i_{1+\sum_{p=1}^{j}m_{p}}}^{n_{j}}\right].
\end{align}
where $\sum_{p=1}^{k}m_{p}+n_{p}=k$. Now if $\sum_{p=1}^{k}m_{p}$ is odd then
\begin{align*}
    \mathbb{E}\left[\prod_{j=1}^{\sum m_{p}}G_{i_{j}\wedge i_{j+1},i_{j}\vee i_{j+1}}\right]=0
\end{align*}
and the corresponding moment becomes 0. So we need only consider $\sum_{p=1}^{k}m_{p}$ is even. Once again we look at the closed walk $i_{1}\rightarrow\cdots i_{m}\rightarrow i_{1}$. (Remember that $i_{m+1}=i_{1}$ and hence we identify 1 by $m+1$ for notational convenience.) Arguing as in the convergence of moments part we can show that it is enough to look at a closed walk on a tree of $m/2+1$ vertices where each edge is visited twice. We know that there exists one correspondence between such a walk and a depth first search over a rooted planar tree having vertices chosen from $[N]$. Hence consider the $r^{th}$ labelled rooted planar tree $T_{r,l}^{m/2+1}=(V,E)$ with the labelling $l=(l_{1},\cdots,l_{m/2+1})$ which corresponds to this walk. Then it is easy to see that
\begin{align*}
    \mathbb{E}\left[\prod_{j=1}^{\sum m_{p}}G_{i_{j}\wedge i_{j+1},i_{j}\vee i_{j+1}}\right]=\mathbb{E}\left[\prod_{e\in E}G_{e}^{2}\right]=1
\end{align*}
and using the notations defined in the setup of Section~\ref{sec:moment_description} we have
\begin{align*}
    \prod_{j=1}^{k}\left(\frac{1}{N}\sum_{t\neq i_{1+\sum_{p=1}^{j}m_{p}}}\sigma_{i_{1+\sum_{p=1}^{j}m_{p}},t}^{2}\right)^{\frac{n_{j}}{2}}=\prod_{s=1}^{\frac{m}{2}+1}\left(\frac{1}{N}\sum_{t\neq l_{s}}\sigma_{l_{s},t}^{2}\right)^{\sum_{j=1}^{\eta_{s}}n_{s_{j}}/2}
\end{align*}
and
\begin{align*}
    \mathbb{E}\left[\prod_{j=1}^{k}Z_{i_{1+\sum_{p=1}^{j}m_{p}}}^{n_{j}}\right]=\mathbb{E}\left[\prod_{s=1}^{\frac{m}{2}}Z_{l_{s}}^{\sum_{j=1}^{\eta_{s}}n_{s_{j}}}\right]=\prod_{s=1}^{\frac{m}{2}}\mathbb{E}\left[Z_{l_{s}}^{\sum_{j=1}^{\eta_{s}}n_{s_{j}}}\right]=f\left(\widetilde{T}_{r,l}^{m/2+1}\right)
\end{align*}
where $f$ is defined in ~\eqref{def:f}. Hence the contribution of this term is 0 if $\sum_{j=1}^{\eta_{s}}n_{s_{j}}$ is odd for some $s\in \{1,2,\cdots, m/2+1\}$. So we consider the situation where $\sum_{j=1}^{\eta_{s}}n_{s_{j}}$ is even for all $s\in \{1,2,\cdots, m/2+1\}$. Then \eqref{eq:full_sum} becomes
\begin{align}\label{eq:sum_in_tree}
    \frac{1}{N^{m/2+1}}\sum_{r=1}^{C_{m/2}}\sum_{l_{1}\neq l_{2}\neq\cdots\neq l_{m/2+1}}\prod_{e\in E\left(T_{r,l}^{m/2+1}\right)}\sigma_{e}^{2}\prod_{s=1}^{\frac{m}{2}+1}\left(\frac{1}{N}\sum_{t\neq l_{s}}\sigma_{l_{s},t}^{2}\right)^{\sum_{j=1}^{\eta_{s}}n_{s_{j}}/2}f\left(\widetilde{T}_{r,l}^{m/2+1}\right)
\end{align}
Observing that $f$ does not depend on the labelling of the tree, and going similarly as in \eqref{eq:graphon_convg_tree_sigma} we can show that 
\begin{align*}
    \frac{1}{N^{m/2+1}}\sum_{l_{1}\neq l_{2}\neq\cdots\neq l_{m/2+1}}\prod_{e\in E\left(T_{r,l}^{m/2+1}\right)}\sigma_{e}^{2}\prod_{s=1}^{\frac{m}{2}+1}
    &\left(\frac{1}{N}\sum_{t\neq l_{s}}\sigma_{l_{s},t}^{2}\right)^{\sum_{j=1}^{\eta_{s}}n_{s_{j}}/2}f\left(\widetilde{T}_{r,l}^{m/2+1}\right)\nonumber\\
    &\rightarrow t(\widetilde{T}_{r}^{m/2+1},W)f(\widetilde{T}_{r}^{m/2+1})
\end{align*}
Then combining with \eqref{eq:sum_in_tree} and \eqref{eq:trace_expansion} we are done. \qed

\begin{proof}[Proof of Corollary \ref{corollary: dingjiang}]

It is easy to observe that $\lambda_{i}(A+\alpha I_{N})=\lambda_{i}(A)+\alpha$ and $\lambda_{i}\left(\frac{A}{\alpha}\right)=\frac{1}{\alpha}\lambda_{i}(A)$ for any $N\times N$ matrix $A$ and for any $\alpha\neq 0$. Then observe that 
\begin{align*}
    \frac{\lambda_{i}(\Delta_{N})-N\mu_{N}}{\sqrt{N}\sigma_{N}}=\lambda_{i}\left(\frac{\Delta_{N}-N\mu_{N}I_{N}}{\sqrt{N}\sigma_{N}}\right)
\end{align*}
Now observe that the centered laplacian is 
\begin{align*}
    \Delta_{N}^{0}=\frac{\Delta_{N}-N\mu_{N}I_{N}+\mu_{N}J_{N}}{\sqrt{N}\sigma_{N}}
\end{align*}
where $J_{N}$ is the $N\times N$ matrix having all entries equal to 1. Define $\Delta_{N}^{1}=\frac{\Delta_{N}-N\mu_{N}I_{N}}{\sqrt{N}\sigma_{N}}$. Now by rank inequality (\citet[Theorem A.43]{bai2010spectral}) we have
\begin{align}
    \left\|F^{\Delta_{N}^{0}}-F^{\Delta_{N}^{1}}\right\|\leq \frac{1}{N}\text{rank}\left(\Delta_{N}^{0}-\Delta_{N}^{1}\right)=\frac{1}{N}\text{rank}\left(\frac{\mu_{N}}{\sqrt{N}\sigma_{N}}J_{N}\right)=O\left(\frac{1}{N}\right)\rightarrow0
\end{align}
where $F^{A_{N}}(x)=\frac{1}{N}\sum_{i=1}^{N}I\left\{\lambda_{i}(A_{N})\leq x\right\}$ for any $N\times N$ symmetric matrix $A_{N}$ and its eigenvalues $\{\lambda_{i}(A_{N}):1\leq i\leq N\}$. Hence it is enough to look at convergence of ESD of $\Delta_{N}^{0}$.\\
Now define $B_{N}=\frac{1}{\sigma_{N}}\left(A_{N}-\mathbb{E}A_{N}\right)=\frac{1}{\sigma_{N}}\left(A_{N}-\mu_{N}J_{N}\right)$. Then it is easy to see that $\sqrt{N}\Delta_{N}^{0}$ is the laplacian corresponding to $B_{N}$. Now observe that for any $\eta>0$
\begin{align}
    \frac{1}{N^{2}}\sum_{1\leq i,j\leq N}\mathbb{E}\left[\left|B_{N}(i,j)\right|^{2}\one\left\{\left|B_{N}(i,j)\right|\geq \eta\sqrt{N}\right\}\right]
    &\leq \frac{1}{N^{2}}\sum_{1\leq i,j\leq N}\frac{1}{\eta^{\delta}N^{\frac{\delta}{2}}}\mathbb{E}\left[\left|\frac{A_{N}(i,j)-\mu_{n}}{\sigma_{N}}\right|^{2+\delta}\right]\\
    &=O\left(N^{-\frac{\delta}{2}}\right)\rightarrow 0
\end{align}
Thus $B_{N}$ satisfies assumption \ref{itm:A2}. Assumption \ref{itm:A1} is immediate from the given assumptions in the theorem. Observe that
\begin{align}
    \mathbb{E}\left[B_{N}(i,j)^{2}\right]=1,\ \forall 1\leq i,j\leq N
\end{align}
Taking $W\equiv 1$, a graphon, we can easily see that assumption \ref{itm:A3} is satisfied. Hence by Theorem \ref{thm:laplacian} there exists a symmetric probability measure $\nu$ of unbounded support where ESD of $\Delta_{N}^{0}$ converges to weakly in probability.\\
Recall from the proof of Theorem \ref{thm:laplacian}, the moments which identify $\nu$ are given by the limits of $\frac{1}{N}\Tr\left(\Bar{A}_{N}+Y_{N}\right)^{k}$ for all $k\geq 1$ where $\Bar{A}_{N}$ is as defined in the Lemma \ref{lemma:DeltaNbar} with $\sigma_{i,j}=1\ \forall i,j$ and $Y_{N}$ is the diagonal $N\times N$ matrix as defined in Lemma \ref{lemma: AN+YN} with $\sigma_{i,j}=1$. It is easy to observe that $\frac{1}{N}\Tr\left(\Bar{A}_{N}+Y_{N}\right)^{k}$ is the $k^{th}$ moment of ESD$\left(\Bar{A}_{N}+Y_{N}\right)$. Since $\nu$ is uniquely identified by it's moments, then using method of moments, we can say that ESD$\left(\Bar{A}_{N}+Y_{N}\right)\Longrightarrow\nu$ in probability. By the strong law of large numbers, with probability 1, ESD$(Y_{N})\Longrightarrow\gamma_{1}$, where $\gamma_{1}$ denotes the standard normal  distribution. Also we know that,with probability 1, ESD$(\Bar{A}_{N})\Longrightarrow\gamma_{0}$, where $\gamma_{0}$ denotes the semicircle law. Further 
\begin{align*}
    \sup_{N}\mathbb{E}\int |x|d\text{ESD}(Y_{N})=\sup_{N}\mathbb{E}\frac{1}{N}\sum_{i=1}^{N}|Z_{i}|<\infty
\end{align*}
where $Z_{i}$ are i.i.d. standard normal. Also $\mathbb{E}\int|x|d\text{ESD}(\Bar{A}_{N})\leq \frac{1}{N}\sqrt{\mathbb{E}\Tr(\Bar{A}_{N}^{2})}=1$. Then by Theorem 2.1 of \citet{pastur2000law} we have ESD$\left(\Bar{A}_{N}+Y_{N}\right)$ converges weakly in probability to $\gamma_{0}\boxplus\gamma_{1}=\gamma_{M}$. Hence by uniqueness of $\nu$ we must have $\nu=\gamma_{M}$.
\end{proof}

\begin{proof}[Proof of Theorem \ref{thm:multi_structure}]

The case when $r\equiv 0$ is trivially true. Hence we assume that $r>0$ at some point in [0,1]. By the proof of Theorem \ref{thm:laplacian}
\begin{align*}
    \text{ESD}\left(\Bar{A}_{N}+Y_{N}\right)\rightarrow\nu,\text{ weakly in probability}
\end{align*}
where $\Bar{A}_{N}$ and $Y_{N}$ are as defined in Lemma \ref{lemma: AN+YN}. Define 
\begin{align}\label{eq:eta_def}
    \eta^{W}\left(\frac{i}{N},\frac{j}{N}\right)=N^{2}\int_{I_{i}\times I_{j}}W(x,y)dxdy,\quad \forall 1\leq i,j\leq N,
\end{align}
where $\{I_i\}$ is the partition of $[0,1]$ as defined in Definition~\ref{def:empiricalgraphon}. Then define the $N\times N$ matrix $\Bar{Z}_{N}$ as
\begin{align*}
    \Bar{Z}_{N}(i,j)=\sqrt{\frac{\eta^{W}\left(\frac{i}{N},\frac{j}{N}\right)}{N}}G_{i\wedge j,i\vee j},\quad 1\leq i,j\leq N
\end{align*}
where the collection $\{G_{i,j}: \,\, 1\leq i,j\leq N\}$ is defined in \eqref{eq:ANgandDelNg}. Further define the $N\times N$ diagonal matrix $Y_{N}^{Z}$ by
\begin{align*}
    Y_{N}^{Z}(i,i)=\sqrt{\frac{1}{N}\sum_{j\neq i}\eta^{W}\left(\frac{i}{N},\frac{j}{N}\right)}Z_{i},\quad 1\leq i\leq N
\end{align*}
where $Z_{i}$ for all $1\leq i\leq N$ are as defined in Lemma \ref{lemma: AN+YN}. Observe that
\begin{align}\label{eq:AN_ZN}
    \frac{1}{N}\mathbb{E}\Tr\left(\Bar{A}_{N}-\Bar{Z}_{N}\right)^{2}
    &=\frac{1}{N^{2}}\sum_{i,j=1}^{N}\left(\sigma_{i,j}-\sqrt{\eta^{W}\left(\frac{i}{N},\frac{j}{N}\right)}\right)^2\nonumber\\
    &=\frac{1}{N^{2}}\sum_{i,j=1}^{N}\eta^{W}\left(\frac{i}{N},\frac{j}{N}\right)+\frac{1}{N^{2}}\sum_{i,j=1}^{N}\sigma_{i,j}^{2}-\frac{2}{N^{2}}\sum_{i,j=1}^{N}\sigma_{i,j}\sqrt{\eta^{W}\left(\frac{i}{N},\frac{j}{N}\right)}
\end{align}
Using \eqref{eq:eta_def} and the fact that $|\sqrt{x}-\sqrt{y}|\leq|x-y|/\sqrt{y}$ for $x, y\in (0,\infty)$, we have
\begin{align*}
    \Bigg|\frac{1}{N^{2}}\sum_{i,j=1}^{N}\sigma_{i,j}
    &\sqrt{\eta^{W}\left(\frac{i}{N},\frac{j}{N}\right)}-\int_{[0,1]^{2}}W(x,y)dxdy\Bigg|\\
    &=\left|\sum_{i,j=1}^{N}\left(\int_{I_{i}\times I_{j}}W(x,y)dxdy\right)^{1/2}\left[\left(\int_{I_{i}\times I_{j}}W_{N}(x,y)dxdy\right)^{1/2}-\left(\int_{I_{i}\times I_{j}}W(x,y) dxdy\right)^{1/2}\right]\right|\\
    &\leq \sum_{i,j=1}^{N}\int_{I_{i}\times I_{j}}\left|W_{N}(x,y)-W(x,y)\right|dxdy\overset{N\rightarrow\infty}{\longrightarrow}0.
\end{align*}
From \eqref{eq:AN_ZN} we conclude that
\begin{align}\label{eq:AN_ZN_convg}
    \frac{1}{N}\mathbb{E}\Tr\left(\Bar{A}_{N}-\Bar{Z}_{N}\right)^{2}\rightarrow 0.
\end{align}
A similar computation as above shows that 
\begin{align}\label{eq:YN_YNZ_convg}
    \frac{1}{N}\mathbb{E}\Tr\left(Y_{N}-Y_{N}^{Z}\right)^{2}\rightarrow 0.
\end{align}
Combining \eqref{eq:AN_ZN_convg} and \eqref{eq:YN_YNZ_convg}, along with \citet[Corollary A.41]{bai2010spectral} we have
\begin{align*}
    L\left(\text{ESD}\left(\Bar{A}_{N}+Y_{N}\right),\text{ESD}\left(\Bar{Z}_{N}+Y_{N}^{Z}\right)\right)\overset{P}{\longrightarrow}0.
\end{align*}
Define 
\begin{align}\label{eq:def_gi}
    g_{i}=\int_{I_{i}}r(x)dx,\quad\forall 1\leq i\leq N
\end{align}
and define the $N\times N$ diagonal matrix $\widetilde{Y}_{N}^{Z}$ as 
\begin{align*}
    \widetilde{Y}_{N}^{Z}(i,i)=\alpha\sqrt{Ng_{i}}Z_{i},\quad \forall 1\leq i\leq N.
\end{align*}
By definition in \eqref{eq:def_gi}
\begin{align*}
    \frac{1}{N}\sum_{j\neq i}\eta^{W}\left(\frac{i}{N},\frac{j}{N}\right)=Ng_{i}(\alpha^{2}-g_{i}),\quad\forall 1\leq i\leq N.
\end{align*}
Thus
\begin{align}
    \frac{1}{N}\mathbb{E}\Tr\left(\left(\Bar{Z}_{N}+Y_{N}^{Z}\right)-\left(\Bar{Z}_{N}+\widetilde{Y}_{N}^{Z}\right)\right)^2
    &=\frac{1}{N}\mathbb{E}\Tr\left(Y_{N}^{Z}-\widetilde{Y}_{N}^{Z}\right)^{2}\nonumber\\
    &=\frac{1}{N}\sum_{i=1}^{N}\left(\sqrt{Ng_{i}(\alpha^{2}-g_{i})}-\alpha\sqrt{Ng_{i}}\right)^{2}\mathbb{E}Z_{i}^{2}\nonumber\\
    &=\sum_{i=1}^{N}g_{i}\left(\alpha-\sqrt{\alpha^{2}-g_{i}}\right)^{2}\nonumber\\
    &\leq \sum_{i=1}^{n}g_{i}\frac{1}{\alpha}\left(\alpha^{2}-\left(\alpha^{2}-g_{i}\right)^{2}\right)\label{eq:g_i_ineq}\\
    &\leq \sum_{i=1}^{N} \frac{1}{\alpha}g_{i}^{3}\rightarrow 0.\nonumber
\end{align}
where \eqref{eq:g_i_ineq} follows from the inequality $|\sqrt{x}-\sqrt{y}|\leq \frac{|x-y|}{\sqrt{y}}$ and the final limit follows since $g_{i}\leq 1/N$ for all $1\leq i\leq N$. Now for $N\geq 1$, define the $N\times N$ matrices
\[
    G_{N}(i,j)=N^{-1/2}G_{i\wedge j,i\vee j}, \quad 1\leq i,j\leq N,
\]
\[
    R_{N}=\text{Diag}\left(\sqrt{Ng_{1}},\sqrt{Ng_{2}},\cdots,\sqrt{Ng_{N}}\right)
\]
and
\[
    U_{N}=\text{Diag}\left(Z_{1},\cdots,Z_{N}\right)
\]
Then observe that
\[
    \Bar{Z}_{N}=R_{N}G_{N}R_{N},
\]
and
\[
    \widetilde{Y}_{N}^{Z}=\alpha R_{N}^{1/2}U_{N}R_{N}^{1/2}
\]
The proof would follow similarly as the proof of Theorem 1.3 from \citet{chakrabarty2018spectra} if we can show that for any $K\geq 0$ and all $k\geq 1$, $m_{1},\cdots, m_{k}$ and $n_{1},\cdots, n_{k}\geq 0$,
\begin{align}\label{eq:lim_poly}
    &\lim_{N\rightarrow\infty}\frac{1}{N}\Tr\left(R_{N}^{m_{1}}U_{NK}^{n_{1}}\cdots R_{N}^{m_{k}}U_{NK}^{n_{k}}\right)\\
    &=\int_{0}^{1}r^{\sum_{i=1}^{k}m_{i}/4}(u)du\int_{-K}^{K}\frac{1}{\sqrt{2\pi}}x^{\sum_{i=1}^{k}n_{i}}e^{-x^{2}/2}dx\text{   a.s. }
\end{align}
where for all $K\geq 0$, the $N\times N$ diagonal matrix $U_{NK}$ is given by
\[
    U_{NK}=\text{Diag}\left(Z_{1}\one_{\{|Z_{1}|\leq K\}},\cdots,Z_{N}\one_{\{|Z_{N}|\leq K\}}\right)
\]
An application of SLLN shows that it is enough to look at the limit of 
\begin{align*}
    \frac{1}{N}\sum_{i=1}^{N}\left(Ng_{i}\right)^{\sum_{j=1}^{k}m_{j}/4}
    &\mathbb{E}\left(Z_{i}^{\sum_{j=1}^{k}n_{j}}\one_{\{|Z_{i}|\leq K\}}\right)\\
    &=\frac{1}{N}\sum_{i=1}^{N}\left(Ng_{i}\right)^{\sum_{j=1}^{k}m_{j}/4}\int_{-K}^{K}\frac{1}{\sqrt{2\pi}}x^{\sum_{i=1}^{k}n_{i}}e^{-x^{2}/2}dx
\end{align*}
Define 
\begin{align*}
    m(x)=\int_{0}^{x}r(t)dt,\quad\forall x\in[0,1]
\end{align*}
Then observe that 
\begin{align*}
    \frac{1}{N}\sum_{i=1}^{N}\left(N\int_{I_{i}}r(x)dx\right)^{\sum_{j=1}^{k}m_{j}/4}=\frac{1}{N}\sum_{i=1}^{N}\left(\frac{m(i/N)-m(i-1/N)}{1/N}\right)^{\sum_{j=1}^{k}m_{j}/4}
\end{align*}
It is easy to see that $m$ is uniformly differentiable on $[0,1]$.\footnote{Let $f$ be defined (and real valued) on $[a,b]$ and the derivative $f^{\prime}$ exists on $[a,b]$ (considering the left and right derivatives at the boundary points). Then the function $f$ is said to be uniformly differentiable if for all $\vep>0$, there exists $\delta>0$ such that whenever $0\leq |t-x|\leq \delta$, $a\leq x,t\leq b$
\begin{align*}
    \left|\frac{f(x)-f(t)}{x-t}-f^{\prime}(x)\right|<\varepsilon
\end{align*}

It follows from \citet[Exercise 5.8]{rudin1964principles} that a function $f$ differentiable on $[a,b]$ is uniformly differentiable if $f^{\prime}$ is continuous on $[a,b]$. }
Hence given $\vep>0$, for large enough $N$ we have
\begin{align*}
    \left(\frac{m(i/N)-m(i-1/N)}{1/N}\right)^{\sum_{j=1}^{k}m_{j}/4}=r\left(\frac{i}{N}\right)^{\sum_{j=1}^{k}m_{j}/4}+O(\vep)
\end{align*}
Hence taking $N\rightarrow\infty$ we have
\begin{align*}
    \frac{1}{N}\sum_{i=1}^{N}\left(\frac{m(i/N)-m(i-1/N)}{1/N}\right)^{\sum_{j=1}^{k}m_{j}/4}\rightarrow\int_{0}^{1}r^{\sum_{i=1}^{k}m_{i}/4}(u)du+O(\vep)
\end{align*}
Since the above is true for all $\vep>0$, then we have shown \eqref{eq:lim_poly}.
\end{proof}

\section{Proof of Theorem \ref{thm:spectral_norm_bound} and Corollary \ref{thm:spectral_norm_mean}} \label{sec:proofSpectralnorm}
Before proving that spectral norm of $\Delta_{N}$ scales as $\sqrt{N\log N}$, we show that the spectral norm of the adjacency matrix would scale slower than $\sqrt{N\log N}$, which is crucially needed for proving the theorem.
\begin{lemma}\label{lemma: AnNorm}
For a $N\times N$ generalised Wigner matrix $A_{N}$ satisfying the above stated assumptions \ref{itm:A1_Norm}-\ref{itm:A3_Norm},
$$
    \frac{\|A_{N}\|}{\sqrt{N\log N}}\overset{a.s.}{\longrightarrow}0
$$
where $\|\cdot\|$ is the spectral norm.
\end{lemma}
\begin{proof}
Following Theorem 3.2 from \citet{zhu2020graphon} we know that 
\begin{align}
    \lim_{N\rightarrow\infty} \text{ESD}\big(A_{N}\big)=\mu \text{ weakly almost surely}
\end{align}
where $\mu$ is identified by the moments given in \eqref{moments:zhu} 
Then going similar as in proof of \eqref{eq:compact_sup} we can conclude that there exists $C_{0}>0$ such that
\begin{align*}
    \sup(\text{supp}(\mu))\leq C^{\circ}
\end{align*}
Then by \citet[Lemma 2.8]{ding2010spectral} there exists $\alpha\in\mathbb{R}$ such that $|\alpha|<\infty$ and  
\begin{align}\label{eq:liminfmax}
    \liminf_{N\rightarrow\infty}\frac{\lambda_{\max}(A_{N})}{\sqrt{N}}\geq\alpha\text{ almost surely}.
\end{align}
(Observe that due to above bound on $\sup(\text{supp}(\mu))$ we can conclude that $|\alpha|<\infty$).
Thus it is enough to prove the upper bound 
\begin{align*}
    \limsup_{N\rightarrow\infty}\frac{\lambda_{\max}(A_{N})}{\sqrt{N}}\leq 2C_{2}^{1/2} \text{almost surely}
\end{align*}
We omit the proof of the upper bound since it follows similar to the proof of Lemma 2.1 from \citet{ding2010spectral}.
\end{proof}

\begin{proof}[Proof of Theorem \ref{thm:spectral_norm_bound}]
The proof follows the line of argument in \citet[Theorem 1.5]{bryc2006spectral}. Before proceeding with the proof in order to make notations clearer we will use $A_{N}=\left(X_{ij}^{(N)}\right)_{1\le i\le j\le N}$ to indicate the dependence on $N$ and similarly $\sigma_{N}(i,j)$ instead of $\sigma_{i,j}$ in corresponding places. 
Define 
$$
    D_{N}=\text{diag}\left(\sum_{j=1}^{N}X_{ij}^{(N)}\right)_{i=1}^{N}
$$
Then using triangle inequality we have 
\begin{align}\label{eq: TrainagleNorm}
    \left|\|\Delta_{N}\|-\|D_{N}\|\right|\leq \|A_{N}\|
\end{align}
Hence by Lemma \ref{lemma: AnNorm} it is enough to look at $\frac{\|D_{N}\|}{\sqrt{N\log N}}$. Define
$$
    T_{N}=\max_{1\leq i\leq N}\left|\sum_{j=1}^{N}X_{ij}^{(N)}\right|\text{ which implies } \frac{T_{N}}{\sqrt{N\log N}}=\frac{\|D_{N}\|}{\sqrt{N\log N}}
$$
Fix $1\leq i\leq N$, then using Lemma 2.1 of \citet{bryc2006spectral} (a result on strong Gaussian approximation) there exists $\left\{Y_{ij}^{(N)}: 1\leq j\leq N\right\}$ where $Y_{ij}^{(N)}\sim \mathrm{N}\left(0,\sigma_{N}(i,j)^{2}\right)$ and are independent such that $\forall\alpha>0$,
\begin{align}
    \mathbb{P}\left(\left|\sum_{j=1}^{N}X_{ij}^{(N)}-\sum_{j=1}^{N}Y_{ij}^{(N)}\right|\geq\alpha\sqrt{N\log N}\right)
    &\leq \frac{C}{1+(\alpha\sqrt{N\log N})^{6}}\sum_{j=1}^{N}\mathbb{E}\left|X_{ij}^{(N)}\right|^{6}\nonumber\\
    &\leq \frac{C_{0}}{N^{2}(\log N)^{3}}
\end{align}
The last inequality follows from assumption \ref{itm:A2_Norm}. Then it can be concluded that $\forall\alpha>0$,
\begin{align}\label{eq: Sakhanenkolimsup}
    \max_{i=1}^{N}\mathbb{P}\left(\left|\sum_{j=1}^{N}X_{ij}^{(N)}-\sum_{j=1}^{N}Y_{ij}^{(N)}\right|\geq\alpha\sqrt{N\log N}\right)\leq \frac{C_{0}}{N^{2}(\log N)^{3}}
\end{align}
Now observe that
\begin{align}
    \left|\sum_{j=1}^{N}X_{ij}^{(N)}\right|\leq \left|\sum_{j=1}^{N}Y_{ij}^{(N)}\right|+\left|\sum_{j=1}^{N}X_{ij}^{(N)}-\sum_{j=1}^{N}Y_{ij}^{(N)}\right|
\end{align}
Hence we have
\begin{align}\label{eq:TNgreaterthan}
    \mathbb{P}\left(T_{N}\geq (\alpha+2\epsilon)\sqrt{A_{2}N\log N}\right)
    &\leq N\max_{i=1}^{N}\mathbb{P}\left(\left|\sum_{j=1}^{N}X_{ij}^{(N)}\right|\geq(\alpha+2\epsilon)\sqrt{A_{2}N\log N}\right)\nonumber\\
    &\leq N\max_{i=1}^{N}\mathbb{P}\left(\left|\sum_{j=1}^{N}Y_{ij}^{(N)}\right|\geq(\alpha+\epsilon)\sqrt{A_{2}N\log N}\right)\nonumber\\
    &+ N\max_{i=1}^{N}\mathbb{P}\left(\left|\sum_{j=1}^{N}X_{ij}^{(N)}-\sum_{j=1}^{N}Y_{ij}^{(N)}\right|\geq\epsilon\sqrt{A_{2}N\log N}\right)
\end{align}
We know that $\sum_{j=1}^{N}Y_{ij}^{(N)}\sim \mathrm{N}\left(0,\sum_{j=1}^{N}\sigma_{N}(i,j)^{2}\right)$. Hence $\sum_{j=1}^{N}Y_{ij}^{(N)}\overset{d}{=}\left(\sum_{j=1}^{N}\sigma_{N}(i,j)^{2}\right)^{1/2}Z$, where $Z\sim\mathrm{N}(0,1)$. Using Gaussian tail inequality (Mill's ratio) we have

\begin{align}
    \mathbb{P}\left(\left|\sum_{j=1}^{N}Y_{ij}^{(N)}\right|\geq(\alpha+\epsilon)\sqrt{A_{2}N\log N}\right)
    &=\mathbb{P}\left(|Z|\geq \frac{(\alpha+\epsilon)\sqrt{A_{2}N\log N}}{\left(\sum_{j=1}^{N}\sigma_{N}(i,j)^{2}\right)^{1/2}}\right)\nonumber\\
    &\leq \mathbb{P}\left(|Z|\geq \frac{(\alpha+\epsilon)\sqrt{A_{2}N\log N}}{(A_{2}N)^{1/2}}\right)\\
    &=\mathbb{P}\left(|Z|\geq (\alpha+\epsilon)\sqrt{\log N}\right)\nonumber\\
    &\leq \frac{2}{\sqrt{2\pi}(\alpha+\epsilon)\sqrt{\log N}}\exp\left(-\frac{(\alpha+\epsilon)^{2}\log N}{2}\right)\nonumber\\
    &\leq C_{1}N^{-\frac{(\alpha+\epsilon)^{2}}{2}}
\end{align}
for some constant $C_{1}>0$ and all $N$ sufficiently large. Then by \eqref{eq:TNgreaterthan} we have
\begin{align}
    \mathbb{P}\left(T_{N}\geq (\alpha+2\epsilon)\sqrt{A_{2}N\log N}\right)
    &\leq C_{1}N^{1-\frac{(\alpha+\epsilon)^{2}}{2}}+\frac{C_{0}}{N(\log N)^{3}}
\end{align}
Now taking $\alpha=2$ we find that R.H.S. of above equation is $O(N^{-1}(\log N)^{-3})$. Then 
\begin{align}
    \sum_{N\geq 1}\mathbb{P}\left(T_{N}\geq (\alpha+2\epsilon)\sqrt{A_{2}N\log N}\right)<\infty
\end{align}
Hence by Borel Cantelli Lemma we have for all $\epsilon>0$
$$
    \limsup_{N\rightarrow\infty}\frac{T_{N}}{\sqrt{A_{2}N\log N}}\leq 2+2\epsilon\quad a.s.
$$
Hence
\begin{align}\label{eq: limsupbdd}
    \limsup_{N\rightarrow\infty}\frac{T_{N}}{\sqrt{2N\log N}}\leq (2A_{2})^{1/2} \quad a.s.
\end{align}
Now define $k_{N}=[N/\log N]$. Further define
\begin{align}
    V_{N}=\max_{i=1}^{k_{N}}\left|\sum_{j=1}^{k_{N}}X_{ij}^{(N)}\right|
\end{align}
Then observe that 
\begin{align}\label{eq: TNVN}
    T_{N}\geq \max_{i=1}^{k_{N}}\left|\sum_{j=k_{N}+1}^{N}X_{ij}^{(N)}\right|-V_{N}
\end{align}
Now observe that by similar computations as \eqref{eq: limsupbdd} we have $\limsup_{N\rightarrow\infty}V_{N}/\sqrt{2k_{N}\log k_{N}}\leq C$ almost surely for some $C>0$. Since $\sqrt{N\log N}/\sqrt{k_{N}\log k_{N}}\rightarrow\infty$. Hence $\lim_{N\rightarrow\infty}V_{N}/\sqrt{N\log N}=0$ almost surely. Then enough to look at 
\begin{align}
    \liminf_{N\rightarrow\infty}\frac{1}{\sqrt{2N\log N}}\max_{i=1}^{k_{N}}\left|\sum_{j=k_{N}+1}^{N}X_{ij}^{(N)}\right|
\end{align}
Now then once again for using a Gaussian approximation lemma (Lemma 2.1 of \citet{bryc2006spectral}), for large enough $N$ we have
\begin{align}\label{eq: Sakhaliminf}
    \max_{i=1}^{k_{N}}\mathbb{P}\left(\left|\sum_{j=1}^{N}X_{ij}^{(N)}-\sum_{j=1}^{N}Y_{ij}^{(N)}\right|\geq\alpha\sqrt{N\log N}\right)\leq \frac{C_{0}}{N^{2}(\log N)^{3}}
\end{align}
for $\left\{Y_{ij}^{(N)}: 1\leq j\leq k_{N}\right\}$ such that the random variables are independent for fixed $1\leq  i\leq k_{N}$ and $Y_{ij}^{(N)}\sim \mathrm{N}\left(0,\sigma_{N}(i,j)^{2}\right)$. Now
\begin{align}
    \mathbb{P}\left(\max_{i=1}^{k_{N}}\left|\sum_{j=k_{N}+1}^{N}X_{ij}^{(N)}\right|\leq (\beta-\epsilon)\sqrt{A_{1}N\log N}\right)
    &\leq \prod_{i=1}^{k_{N}}\mathbb{P}\left(\left|\sum_{j=k_{N}+1}^{N}X_{ij}^{(N)}\right|\leq (\beta-\epsilon)\sqrt{A_{1}N\log N}\right)
\end{align}
Observe the above uses the independence because of $i<j$. Then for fixed $1\leq i\leq k_{N}$ we have
\begin{align}\label{eq: liminfexpr}
    \mathbb{P}\left(\left|\sum_{j=k_{N}+1}^{N}X_{ij}^{(N)}\right|\leq (\beta-2\epsilon)\sqrt{A_{1}N\log N}\right)
    &\leq \mathbb{P}\left(\left|\sum_{j=k_{N}+1}^{N}Y_{ij}^{(N)}\right|\leq (\beta-\epsilon)\sqrt{A_{1}N\log N}\right)\nonumber\\
    &+\mathbb{P}\left(\left|\sum_{j=k_{N}+1}^{N}X_{ij}^{(N)}-\sum_{j=k_{N}+1}^{N}Y_{ij}^{(N)}\right|\geq \epsilon\sqrt{A_{1}N\log N}\right)
\end{align}
Now once again observe that $\sum_{j=k_{N}+1}^{N}Y_{ij}^{(N)}\sim\mathrm{N}(0,\sum_{j=k_{N}+1}^{N}\sigma_{N}(i,j)^{2})$. Then
\begin{align}
    \mathbb{P}\left(\left|\sum_{j=k_{N}+1}^{N}Y_{ij}^{(N)}\right|\leq (\beta-\epsilon)\sqrt{A_{1}N\log N}\right)
    &=\mathbb{P}\left(|Z|\leq\frac{(\beta-\epsilon)\sqrt{A_{1}N\log N}}{\sqrt{\sum_{j=k_{N}+1}^{N}\sigma_{N}(i,j)^{2}}} \right)\nonumber\\
    &\geq 1-\mathbb{P}\left(|Z|>\frac{(\beta-\epsilon)\sqrt{A_{1}N\log N}}{\sqrt{(N-k_{N})A_{1}}} \right)
\end{align}
Now observe that taking $0<\epsilon<\beta$, we must have for $N$ sufficiently large $(\beta-\epsilon)\sqrt{\frac{N}{N-k_{N}}}\leq (\beta-(\epsilon/2))$. Thus by using Mill's ratio for sufficiently large $N$ we have 
\begin{align}
    \mathbb{P}\left(|Z|>\frac{(\beta-\epsilon)\sqrt{A_{1}N\log N}}{\sqrt{(N-k_{N})A_{1}}} \right)
    &\geq \mathbb{P}\left(|Z|>(\beta-\epsilon)\sqrt{\frac{N}{N-k_{N}}}\sqrt{\log N}\right)\nonumber\\
    &\geq \frac{C}{N^{(\beta-\epsilon/2)^{2}/2}\log N}
\end{align}
for some constant $C>0$. Then combining the above equation with \eqref{eq: Sakhaliminf} and \eqref{eq: liminfexpr} we have
\begin{align}
    \mathbb{P}\left(\left|\sum_{j=k_{N}+1}^{N}X_{ij}^{(N)}\right|\leq (\beta-2\epsilon)\sqrt{A_{1}N\log N}\right)
    &\leq 1-\frac{C}{N^{(\beta-\epsilon/2)^{2}/2}\log N}+\frac{C_{0}}{N^{2}(\log N)^{3}}\nonumber\\
    &\leq 1-\frac{C_{2}}{N^{(\beta-\epsilon/2)^{2}/2}\log N}
\end{align}
for all sufficiently large enough $N$ and some constant $C_{2}>0$. Then using the inequality $1-x\leq e^{-x}$ for any $x>0$, we have
\begin{align}
    \mathbb{P}\left(\max_{i=1}^{k_{N}}\left|\sum_{j=k_{N}+1}^{N}X_{ij}^{(N)}\right|\leq (\beta-\epsilon)\sqrt{A_{1}N\log N}\right)
    &\leq \exp\left(-\frac{C_{2}}{N^{(\beta-\epsilon/2)^{2}/2}\log N}k_{N}\right)\nonumber\\
    &=O\left(\exp(-C_{2}N^{1-(\beta-\epsilon/2)^{2}/2})\right)
\end{align}
Taking $\beta=\sqrt{2}$, it can be seen that
\begin{align}
    \sum_{N\geq 1}\mathbb{P}\left(\max_{i=1}^{k_{N}}\left|\sum_{j=k_{N}+1}^{N}X_{ij}^{(N)}\right|\leq (\sqrt{2}-\epsilon)\sqrt{A_{1}N\log N}\right)<\infty
\end{align}
Then by Borel-Cantelli Lemma we can conclude that 
\begin{align}
    \liminf_{N\rightarrow\infty}\frac{1}{\sqrt{N\log N}}\max_{i=1}^{k_{N}}\left|\sum_{j=k_{N}+1}^{N}X_{ij}^{(N)}\right|\geq (\sqrt{2}-\epsilon)\sqrt{A_{1}},\ a.s. \quad 0<\epsilon<\sqrt{2}.
\end{align}
Hence combining with \eqref{eq: TNVN} we must have
\begin{align}
    \liminf_{N\rightarrow\infty}\frac{T_{N}}{\sqrt{2N\log N}}\geq \sqrt{A_{1}}.
\end{align}
Then using \eqref{eq: TrainagleNorm} we have \eqref{eq:upperlowerbounds}.
\end{proof}

\begin{proof}[Proof of Corollary \ref{thm:spectral_norm_mean}]
We use the notation of Theorem \ref{thm:spectral_norm_bound}. 
Define a $N\times N$ diagonal matrix $E_{N}$ as
\begin{align*}
    E_{N}=\mathbb{E}X_{ij}^{(N)}-\sum_{j=1}^{N}\mathbb{E}X_{ij}^{(N)}\one_{i=j}
\end{align*}
Using Theorem \ref{thm:spectral_norm_bound} it is immediate that
\begin{align*}
    \lim_{N\rightarrow\infty}\|\widetilde{\Delta}_{N}\|/N\rightarrow 0\text{ almost surely},
\end{align*}
where $\widetilde{\Delta}_{N}=\Delta_{N}-E_{N}$. Observe that by assumption \ref{itm:A5_Norm} we have $\frac{\|E_{N}\|}{N}\rightarrow m$
Hence using triangle inequality we can conclude that
\begin{align*}
    \frac{\|\Delta_{N}\|}{N}\rightarrow m\text{ almost surely}.
\end{align*}
\end{proof}

\section*{Appendix}
In the appendix, we provide the proofs of some of the lemma used before. The methods are straightforward and hence they are recalled in the appendix. First, we provide proof of Lemma~\ref{lemma:Gaussianisation}. We shall use some notations and results from \citet{chatterjee2005simple}.

\begin{definition}\label{def:partial_sup}
For any open interval $I$ containing 0, any positive integer $n$, any function $f:I^{n}\rightarrow\mathbb{C}$ which is thrice differentiable in each coordinate, and $1\leq r\leq 3$, let
\begin{align}
    \lambda_{r}\left(f\right)=\sup\left\{\left|\partial_{i}^{p}f(x)\right|^{\frac{r}{p}}: 1\leq i\leq n, 1\leq p\leq r, x\in I^{n}\right\}
\end{align}
where $\partial_{i}^{p}$ denotes $p$-fold differentiation with respect to the $i^{th}$ co-ordinate.
\end{definition}
\begin{lemma}\citet[Theorem 1.1]{chatterjee2005simple}\label{lemma: Chatterjee}
Let $f:I^{n}\rightarrow\mathbb{R}$ be thrice differentiable in each argument. If we set $U=f(\mathbf{X})$ and $V=f(\mathbf{Y})$, then for any thrice differentiable $g:\mathbb{R}\rightarrow\mathbb{R}$ and any $K>0$,
\begin{align*}
    \left|\mathbb{E}g(U)-\mathbb{E}g(V)\right|
    &\leq C_{1}(g)\lambda_{2}(f)\sum_{i=1}^{n}\left[\mathbb{E}\left(X_{i}^{2}\one_{\left\{|X_{i}|>K\right\}}\right)+\mathbb{E}\left(Y_{i}^{2}\one_{\left\{|Y_{i}|>K\right\}}\right)\right]\\
    &+C_{2}(g)\lambda_{3}(f)\sum_{i=1}^{n}\left[\mathbb{E}\left(X_{i}^{3}\one_{\left\{|X_{i}|\leq K\right\}}\right)+\mathbb{E}\left(Y_{i}^{3}\one_{\left\{|Y_{i}|\leq K\right\}}\right)\right]
\end{align*}
where $C_{1}(g)=\|g'\|_{\infty}+\|g''\|_{\infty}$ and $C_{2}(g)=\frac{1}{6}\|g'\|_{\infty}+\frac{1}{2}\|g''\|_{\infty}+\frac{1}{6}\|g'''\|_{\infty}$
\end{lemma}
\begin{proof}[Proof of Lemma~\ref{lemma:Gaussianisation}]
For the proof we shall use a result from \citet{chatterjee2005simple} which is recalled later as Lemma~\ref{lemma: Chatterjee}. Consider $f,r,n$ as defined in Definition \ref{def:partial_sup}. Then observe that $\left|\partial_{i}^{p}\mathcal{R}f\right|=\left|\mathcal{R}\partial_{i}^{p}f\right|\leq\left|\partial_{i}^{p}f\right|$. Now  using the fact that $\frac{r}{p}>0$ we have $\left|\partial_{i}^{p}\mathcal{R}f\right|^{\frac{r}{p}}\leq \left|\partial_{i}^{p}f\right|^{\frac{r}{p}}$,  hence we have 
\begin{align}\label{eq:lambda_inequality}
    \lambda_{r}\left(\mathcal{R}f\right)\leq \lambda_{r}\left(f\right).
\end{align}
Let define $\widetilde{X}=\left(A_{N}^{0}(i,j)\right)_{1\leq i<j\leq N}$ and $\widetilde{Y}=\left(A_{N}^{g}(i,j)\right)_{1\leq i<j\leq N}$. Take $n=\frac{N(N-1)}{2}$. Then for all $ x=(x_{i,j})_{1\leq i<j\leq N}\in\mathbb{R}^{n}$ define a real symmetric $N\times N$ matrix $\Delta(x)$ as
\begin{align}
    \Delta(x)(i,j)=
    \begin{cases}
        x_{i\wedge j,i\vee j} & \text{ if } i\neq j\\
        -\sum_{k\neq i,k=1}^{N}x_{i\wedge k,i\vee k} & \text{ if } i=j
    \end{cases}
\end{align}
Define $\Phi(x)=H_{N}(\Delta(x))$. Observe that $\Delta(\Tilde{X})=\Delta_{N}^{0}$ and $\Delta(\Tilde{Y})=\Delta_{N}^{g}$. Since $\mathbb{E}\left(A_{N}^{0}(i,j)\right)=\mathbb{E}\left(A_{N}^{g}(i,j)\right)=0$ and $\mathbb{E}\left((A_{N}^{0}(i,j))^2\right)=\mathbb{E}\left((A_{N}^{g}(i,j))^2\right)=\sigma_{i,j}^{2}/N$. Hence the assumptions on $\widetilde{X}$ and $\widetilde{Y}$ for Lemma \ref{lemma: Chatterjee} are satisfied. Note that real part of $\Phi$ is  thrice differentiable as $\Phi$ is thrice differentiable.
Observe that $\pdv{\Delta(x)}{x_{i,j}}$ is a $N\times N$ matrix having $-1$ at $i^{th}$ and $j^{th}$ diagonals and $1$ at $(i,j)$ and $(j,i)$ positions. Using matrix identities derived in \citet{chatterjee2005simple} we have 
\begin{align}
    & \pdv{\Phi}{x_{i,j}}=-\frac{1}{N}\Tr\left(\pdv{\Delta}{x_{i,j}}K^{2}\right)\nonumber\\
    & \pdv[2]{\Phi}{x_{i,j}}=2\frac{1}{N}\Tr\left(\pdv{\Delta}{x_{i,j}}K\pdv{\Delta}{x_{i,j}}K^{2}\right)\\
    & \pdv[3]{\Phi}{x_{i,j}}=-6\frac{1}{N}\Tr\left(\pdv{\Delta}{x_{i,j}}K\pdv{\Delta}{x_{i,j}}K\pdv{\Delta}{x_{i,j}}K^{2}\right)\nonumber
\end{align}
where $K(x)=\left(\Delta(x)-zI_{N}\right)^{-1}$. These identities along with some standard norm inequlities give us
 \begin{align*}
     \left\|\pdv{\Phi}{x_{i,j}}\right\|_{\infty}\leq \frac{C_{1}}{N}, \quad \left\|\pdv[2]{\Phi}{x_{i,j}}\right\|_{\infty}\leq \frac{C_{2}}{N}, \quad \left\|\pdv[3]{\Phi}{x_{i,j}}\right\|_{\infty}\leq \frac{C_{3}}{N}
 \end{align*}
Then by definition we have 
\begin{align*}
    &\lambda_{2}(\Phi)\leq \sup\left\{\|\pdv{\Phi}{x_{i,j}}\|^{2}_{\infty},\|\pdv[2]{\Phi}{x_{i,j}}\|_{\infty}\right\}\leq \frac{K_{1}}{N}\\
    &\lambda_{3}(\Phi)\leq \sup\left\{\|\pdv{\Phi}{x_{i,j}}\|_{\infty}^{3},\|\pdv[2]{\Phi}{x_{i,j}}\|_{\infty}^{\frac{3}{2}},\|\pdv[3]{\Phi}{x_{i,j}}\|_{\infty}\right\}\leq \frac{K_{2}}{N}
\end{align*}
for some $K_{1},K_{2}>0$. Take $U=\mathcal{R}\Phi(\widetilde{X})$ and $V=\mathcal{R}\Phi(\widetilde{Y})$. Now using Lemma \ref{lemma: Chatterjee}, we have, $\forall\epsilon>0$
\begin{align*}
   & \left|\mathbb{E}[h(U)]-\mathbb{E}[h(V)]\right|\leq\\
   &
     C_{1}(h)\lambda_{2}(\Phi)\sum_{1\leq i<j\leq N}\bigg[\mathbb{E}\left(|A_{N}^{0}(i,j)|^{2}\mathbf{1}(|A_{N}^{0}(i,j)|>\epsilon)\right)
     +\mathbb{E}\left(|A_{N}^{g}(i,j)|^{2}\mathbf{1}(|A_{N}^{g}(i,j)|>\epsilon)\right)\bigg]\nonumber\\
    & +C_{2}(h)\lambda_{3}(\Phi)\sum_{1\leq i<j\leq N}\bigg[\mathbb{E}\left(|A_{N}^{0}(i,j)|^{3}\mathbf{1}(|A_{N}^{0}(i,j)|\leq\epsilon)\right)\nonumber
    +\mathbb{E}\left(|A_{N}^{g}(i,j)|^{3}\mathbf{1}(|A_{N}^{g}(i,j)|\leq\epsilon)\right)\bigg]\nonumber
\end{align*}
Denote $\mathbb{E}[X_{i,j}]=\mu_{i,j}$ and using $\lambda_{2}(\Phi)=O(N^{-1})$ we have\footnote{we use the notation $f_N\preceq g_N$ if $f_N\le C g_N$ for some $C$ for all $N$ large enough.}
\begin{align}
    &\lambda_{2}(\Phi)\sum_{1\leq i<j\leq N}\mathbb{E}\bigg[|A_{N}^{0}(i,j)|^{2}\mathbf{1}(|A_{N}^{0}(i,j)|>\epsilon)\bigg]\\
    &\preceq \frac{1}{N^{2}}\sum_{1\leq i<j\leq N}\mathbb{E}\bigg[|X_{i,j}-\mu_{i,j}|^{2}\mathbf{1}(|X_{i,j}-\mu_{i,j}|>\epsilon\sqrt{N})\bigg] \rightarrow 0, \,\,   \text{ as } N\rightarrow\infty.\nonumber
\end{align}
The last limit follows from \ref{itm:A2}. Using the definition of $A_{N}^{g}$ from \eqref{eq:ANgandDelNg} and $\sup_{i,j}\sigma_{i,j}<C_{0}$ for some $C_{0}>0$ (\ref{itm:A1}) we have
\begin{align}
    &\lambda_{2}(\Phi)\sum_{1\leq i<j\leq N}\mathbb{E}\bigg[|A_{N}^{g}(i,j)|^{2}\mathbf{1}(|A_{N}^{g}(i,j)|>\epsilon)\bigg]\\
    &\preceq \frac{1}{N^{2}}\sum_{1\leq i<j\leq N}\mathbb{E}\bigg[|G_{i,j}|^{2}\mathbf{1}(|G_{i,j}|>\frac{\epsilon\sqrt{N}}{C_{0}})\bigg]\nonumber
    \preceq \mathbb{E}\bigg[|G_{i,j}|^{2}\mathbf{1}(|G_{i,j}|>\frac{\epsilon\sqrt{N}}{C_{0}})\bigg]\nonumber\rightarrow 0 \text{ as } N\rightarrow\infty\nonumber
\end{align}
Now we deal with the factor involving the third derivative. Again using bounds from \ref{itm:A1} we have
\begin{align}
    \lambda_{3}(\Phi)\sum_{1\leq i<j\leq N}\mathbb{E}\bigg[|A_{N}^{0}(i,j)|^{3}\mathbf{1}(|A_{N}^{0}(i,j)| \leq\epsilon)\bigg]
    &\preceq \epsilon\frac{1}{N^{2}}\sum_{1\leq i<j\leq N}\mathbb{E}\bigg[|X_{i,j}-\mu_{i,j}|^{2}\bigg]\nonumber\\
    &\preceq\epsilon\quad\text{ as }N\rightarrow\infty.
\end{align}
Similarly for the Gaussian case we have,
\begin{align}
    \lambda_{3}(\Phi)\sum_{1\leq i<j\leq N}\mathbb{E}\bigg[|A_{N}^{g}(i,j)|^{3}\mathbf{1}(|A_{N}^{g}(i,j)|\leq\epsilon)\bigg]
    &\preceq \epsilon \frac{1}{N^{2}}\sum_{1\leq i<j\leq N}\mathbb{E}|G_{i,j}|^{2}\nonumber\\
    &\preceq \epsilon \text{ as } N\rightarrow\infty\nonumber
\end{align}
Hence for any $\epsilon>0$ we have
\begin{align}
    |\mathbb{E}(h(U))-\mathbb{E}(h(V))|\preceq\epsilon\text{ as } N\rightarrow\infty.
\end{align}
We have thus proved~\ref{eq:Gaussianisation 3}. Similarly, one can prove \ref{eq:Gaussianisation 4}.

\end{proof}

\bibliographystyle{abbrvnat}
\bibliography{eigenbib}

\end{document}